\documentclass[a4paper,english]{amsart}

\usepackage{amsmath}
\usepackage{amssymb}
\usepackage{amsthm}
\usepackage{mathrsfs}
\usepackage{wasysym}
\usepackage{bbm}		%
\usepackage{mathtools}
\usepackage[shortlabels]{enumitem}
\usepackage{braket}		%
\usepackage{mdwlist}		%
\usepackage{xcolor}
\usepackage{tikz}
\usepackage{graphicx}
\usepackage[a4paper, left=2.5cm, right=2.5cm, top=3.0cm, bottom=2cm]{geometry}
\usetikzlibrary{patterns,positioning,arrows,decorations.markings,calc,decorations.pathmorphing,decorations.pathreplacing}
\usepackage{hyperref}
\usepackage{cleveref}
\usepackage{color}
\usepackage{float}
\usepackage{soul}   %
\usepackage{caption}
\usepackage{subcaption}
\usepackage{circledsteps}
\usepackage{multirow}

\usepackage{todonotes}

\theoremstyle{plain}
\newtheorem{theorem}{Theorem}[section]
\newtheorem{lemma}[theorem]{Lemma}

\newtheorem{proposition}[theorem]{Proposition}
\newtheorem{conjecture}[theorem]{Conjecture}
\newtheorem{definition}[theorem]{Definition}

\newtheorem{question}[theorem]{Question}

\theoremstyle{definition}
\newtheorem{remark}[theorem]{Remark}

\newtheorem{maintheorem}{Theorem}

\newcommand{\ZZ}{\mathbb{Z}}			%
\newcommand{\NN}{\mathbb{N}}			%
\newcommand{\RR}{\mathbb{R}}			%
\newcommand{\QQ}{\mathbb{Q}}			%
\newcommand{\Z}{\mathbb{Z}}			%
\newcommand{\N}{\mathbb{N}}			%
\newcommand{\R}{\mathbb{R}}			%
\newcommand{\C}{\mathbb{C}}			%

\newcommand{\ZZd}{{\mathbb{Z}^d}}

\newcommand{\bg}{{\boldsymbol{g}}}
\newcommand{\bk}{{\boldsymbol{k}}}
\newcommand{\bn}{{\boldsymbol{n}}}
\newcommand{\bm}{{\boldsymbol{m}}}
\newcommand{\bp}{{\boldsymbol{p}}}
\newcommand{\bq}{{\boldsymbol{q}}}
\newcommand{\bu}{{\boldsymbol{u}}}
\newcommand{\bv}{{\boldsymbol{v}}}
\newcommand{\bx}{{\boldsymbol{x}}}

\newcommand{\bT}{{\boldsymbol{T}}}

\newcommand{\bomega}{{\boldsymbol{\omega}}}
\newcommand{\vecomega}{{\bomega}}

\newcommand{\Acal}{\mathcal{A}}

\newcommand{\Lcal}{\mathcal{L}}
\newcommand{\Pcal}{\mathcal{P}}
\newcommand{\T}{\mathcal{T}}
\newcommand{\Xcal}{\mathcal{X}}

\newcommand{\dist}{\mathrm{dist}}
\newcommand{\card}{\mathrm{card}}

\newcommand{\norm}[1]{%
	\left\lVert#1\right\rVert%
}

\title{%
	Nonexpansive directions in the Jeandel-Rao Wang shift
}

\author[S.~Labb\'e]{S\'ebastien Labb\'e}
\address[S.~Labb\'e]{Univ. Bordeaux, CNRS,  Bordeaux INP, LaBRI, UMR 5800, F-33400, Talence, France}
\email{sebastien.labbe@labri.fr}

\author[C.~Mann]{Casey Mann}
\address[C.~Mann]{University of Washington Bothell, 18115 Campus Way NE, Bothell, WA 98011-8246}
\email{cemann@uw.edu}

\author[J.~McLoud-Mann]{Jennifer McLoud-Mann}
\address[J.~McLoud-Mann]{University of Washington Bothell, 18115 Campus Way NE, Bothell, WA 98011-8246}
\email{jmcloud@uw.edu}

\makeatletter
\@namedef{subjclassname@2020}{\textup{2020} Mathematics Subject Classification}
\makeatother



\date{}

\begin{document}

	\maketitle

		\begin{abstract}
We show that $\{0,\varphi+3,-3\varphi+2,-\varphi+\frac{5}{2}\}$ is the set of slopes of nonexpansive directions for a minimal subshift in the Jeandel-Rao Wang shift, where $\varphi=(1+\sqrt{5})/2$ is the golden mean. This set is a topological invariant allowing to distinguish the Jeandel-Rao Wang shift from other subshifts. Moreover, we describe the combinatorial structure of the two resolutions of the Conway worms along the nonexpansive directions in terms of irrational rotations of the unit interval. The introduction finishes with pictures of nonperiodic Wang tilings corresponding to what Conway called the cartwheel tiling in the context of Penrose tilings. The article concludes with open questions regarding the description of octopods and essential holes in the Jeandel-Rao Wang shift.
		\end{abstract}
	
	\medskip

	\section{Introduction}\label{sec:introduction}

A \emph{tiling} of the Euclidean plane $\mathbb{E}^2$ is a collection of sets called \emph{tiles} (typically topological disks) whose interiors are pairwise disjoint and whose union is $\mathbb{E}^2$. A   \emph{protoset} for a tiling $T$ is a minimal collection $\mathcal{T}$ of tiles of $T$ such that every tile in $T$ is congruent to a tile in $\mathcal{T}$, and in this case we say that $\mathcal{T}$ \emph{admits} the tiling $T$. If $\mathcal{T}$ admits a tiling $T$, it may also admit other tilings, and the collection of all tilings admitted by $\mathcal{T}$ is called the \emph{tiling space} of $\mathcal{T}$.  Restrictions are sometimes made regarding the kinds of rigid motions allowed in forming tilings from copies of tiles in a protoset; it is not uncommon to allow only direct rigid motions (no reflections) or to allow only translations. The \emph{symmetry group} of a tiling $T$, denoted $S(T)$, is the collection of rigid motions $\sigma$ such that $\sigma(T) = T$. If $S(T)$ contains two nonparallel translations, we say that $T$ is \emph{periodic}; otherwise, $T$ is \emph{nonperiodic}. It is a special kind of protoset that admits only nonperiodic tilings; such protosets are called \emph{aperiodic protosets}, and such a protoset is the focus of this article.

Aperiodic protosets have an interesting history, which we shall briefly touch on here. The first example of an aperiodic protoset was discovered by R. Berger and contained 20426 distinct tiles \cite{MR216954}. Berger's protoset consisted of squares having edge matching rules, with the restriction that only translates of the tiles in the protoset can be used to form a tiling. The squares of such protosets are today known as \emph{Wang tiles}, named after Berger's thesis advisor H. Wang who had conjectured that any protoset of Wang tiles that admits a tiling must admit at least one periodic tiling; thus Berger proved Wang's conjecture is false with his discovery of the first aperiodic protoset. From \cite{GS1} we recount some of the early history of aperiodic protosets of Wang tiles: Not long after Berger's initial discovery, Berger himself was able to reduce the number of tiles needed to form an aperiodic protoset of Wang tiles down to 104 tiles; in 1968, Knuth reduced it further to 92 \cite[p.~384]{Knuth68}. Around the same time, though a correct version of it was not published until much later, H. L\"{a}uchli reduced the number to 40, and soon thereafter R. M. Robinson was able to reduce the number to 35. All of these protosets were based on Berger's original aperiodic Wang tile protoset, but in the 1970s R. Penrose discovered a new non-Wang tile aperiodic protoset consisting of only 2 tiles \cite{Penrose79}. Penrose showed that his order-2 aperiodic protoset could be used to produce an order-34 aperiodic Wang tile protoset, and R. M. Robinson subsequently used this construction to produce an aperiodic protoset of 32 Wang tiles (it is pointed out in \cite{GS1} that Robinson's construction can actually produce an order-24 aperiodic Wang protoset). Later, Robinson discovered a way to produce another aperiodic Wang tile protoset, this time consisting of 24 tiles, based on Ammann's order-2 aperiodic protoset. Around this time (late 1970s) Ammann found a way to reduce the number of tiles in an aperiodic protoset of Wang tiles to 16 via the use of ``Ammann bars" applied to his order-2 aperiodic protoset. More recent reductions in the number of tiles in aperiodic protosets of Wang tiles have been discovered by Kari (order-14, \cite{Kari96}) and Culik (order-13, \cite{Culik96}), finally culminating with the order-11 aperiodic protoset of Wang tiles discovered recently by Jeandel and Rao \cite{JR1}; in this work, the authors proved that 11 is the smallest possible size of an aperiodic protoset of Wang tiles.  If we do not restrict attention to Wang tiles, order-2 aperiodic protosets are known, such as the sets discovered by Penrose and Ammann mentioned above, and recently an aperiodic protoset consisting of a single tile was discovered \cite{smith_aperiodic_2023}.

The connection of aperiodic protosets and nonperiodic order in tilings to quasicrystals (discovered in 1982 by Shechtman \cite{Shec84}) stimulated much research in the 1980s onward. An excellent overview of order in aperiodic tiling spaces is given in \cite{BG13}. This work summarizes the state of the art at the time of its publication, and included various methods for analyzing tiling spaces of aperiodic protosets, including the cut-and-project method and, of particular importance to this article, modeling tiling spaces as dynamical systems. Another important contribution in this area is in \cite{MR1355301}, where the space of tilings admitted by the Penrose tiling are explained in terms of symbolic dynamics. This will be the point of view in this article; in particular, we will view a protoset $\mathcal{T}$ as a finite alphabet, and tilings $T$ admitted by $\mathcal{T}$ can be realized as configurations $x \in \mathcal{T}^{\ZZ^2} = \{ x \!:\! \ZZ^2 \rightarrow \mathcal{T}\}$ that do not contain any patterns from a finite set of forbidden patterns (corresponding to the ways in which the Wang tiles cannot be placed adjacent to one another). As such, the set of all tilings admitted by $\mathcal{T}$ is a shift of finite type. Translating questions about tilings admitted by $\mathcal{T}$ to the language of dynamical systems (and vice versa) gives interesting connections to previously unrelated concepts. In this article, we will examine the dynamical systems notion of \emph{nonexpansive directions} in the context of the space of tilings admitted by the Jeandel-Rao minimal order aperiodic Wang protoset.

\begin{figure}[h]
\centering
\begin{subfigure}[h]{.45\textwidth} 
\centering
\includegraphics[width=\textwidth]{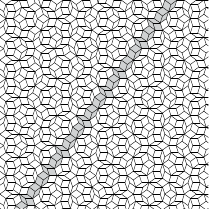} 
\end{subfigure}
\begin{subfigure}[h]{.45\textwidth} 
\centering
\includegraphics[width=\textwidth]{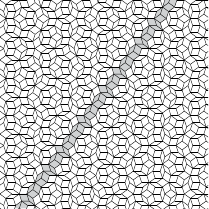}  
\end{subfigure}\caption{Two resolutions of singular tilings by unmarked Penrose rhombs. The parts of the tilings where the two tilings are different is shaded and are called Conway worms.}\label{fig:conway_worms} 
\end{figure}

\subsection*{Conway Worms and Nonexpansive Directions}

The notion of \emph{Conway worms} was considered in \cite[\S 10.5]{GS1} in the
context of tilings by Penrose kites and darts. It was then defined as ``\emph{a
sequence of bow ties placed end to end}'' and it was proved that
every tiling by Penrose kites and darts contains arbitrarily long finite Conway worms, see
\cite[10.5.8]{GS1}. Also it was noted that there are 5 different possible
slopes for these Conway worms and the difference between any two of them is a
multiple of $\frac{\pi}{5}$.

The understanding of Penrose tilings was greatly improved
by N. G. de Bruijn
who for the first time expressed them in terms of cut and project schemes where
the aperiodic tilings are described as the projection of a lattice living in the product
of the physical space of dimension two and some internal space of dimension three \cite{MR609465}.
Based on this work, Robinson further developed  the dynamical properties of
Penrose tilings \cite{MR1355301}.
In particular, he expressed Conway worms appearing in the singular Penrose
tilings in terms of coincidences happening in the internal space
and noted that Conway worms come in pairs
\cite[\S 6]{MR1355301} that he called positive and negative resolutions of
a Conway worm (see also the same idea appearing in \cite[Figures 12 and 13]{MR609465}).
Figure~\ref{fig:conway_worms} shows a portion of a singular tiling by unmarked Penrose rhombs containing a Conway worm and its two resolutions, see also \cite[Figure 7.22]{BG13}.
A reproduction of Figure 8 from \cite{MR1355301} illustrating the two ways to resolve Conway worms in
the context of Penrose tilings is shown in \Cref{fig:conway-worm-penrose}.
Notice that the existence of an infinite Conway worm of a given slope
$\alpha\in\RR\cup\{\infty\}$ implies the existence of a tiling of some half-plane
delimited by a line of slope $\alpha$ which has more than one completion to a
tiling of the whole plane. 
The notion of Conway worms may give more insights on a family of tilings.
For instance, it allows one to prove 
that tiles occur in only finitely many orientations 
in parallelogram tilings using a finite number of shapes \cite{MR3038528}.

\begin{figure}[h]
\begin{center}
    \begin{tikzpicture}[very thick, scale=.5]
    \def\A{++ (18:1)}  \def\a{++ (198:1)}
    \def\B{++ (90:1)}  \def\b{++ (270:1)}
    \def\C{++ (162:1)} \def\c{++ (342:1)}
    \def\D{++ (234:1)} \def\d{++ (54:1)}
    \def\E{++ (306:1)} \def\e{++ (126:1)}

    \def\OuterBoundary{
    \node[left] at (0,0) {$\cdots$};
    \draw (0,0) --\B --\A --\c --\A --\c --\d --\E --\A --\c --\d --\E -- \b node[right] {$\cdots$};
    \draw (0,0)      --\c --\A --\c --\A --\E --\d --\c --\A --\E --\d;
    \draw (0,0)   \B   \A   \c --\b;
    \draw (0,0)   \B   \A   \c   \A   \c --\b;
    \draw (0,0)   \B   \A   \c   \A   \c   \d   \E --\b;
        \draw (0,0)   \B   \A   \c   \A   \c   \d   \E   \A   \c --\b;}

\node (worm) at (0,0) {
    \begin{tikzpicture}[very thick, scale=.5]
        \OuterBoundary
    \end{tikzpicture}
};

\node (neg) at (-7,-4) {
    \begin{tikzpicture}[very thick, scale=.5]
        \OuterBoundary
    \draw (0,0)   \B --\c --\A --\c --\A --\E --\d --\c --\A --\E --\d;
    \draw (0,0)   \B   \c --\b;
    \draw (0,0)   \B   \c   \A   \c --\b;
    \draw (0,0)   \B   \c   \A   \c   \A   \E --\b;
    \draw (0,0)   \B   \c   \A   \c   \A   \E   \d   \c --\b;
    \draw (0,0)   \B   \c   \A   \c   \A   \E   \d   \c   \A   \E --\b;
    \end{tikzpicture}
};

\node (pos) at (7,-4) {
    \begin{tikzpicture}[very thick, scale=.5]
        \OuterBoundary
    \draw (0,0)      --\A --\c --\A --\c --\d --\E --\A --\c --\d --\E;
    \draw (0,0)   \B   \A --\b;
    \draw (0,0)   \B   \A   \c   \A --\b;
    \draw (0,0)   \B   \A   \c   \A   \c   \d --\b;
    \draw (0,0)   \B   \A   \c   \A   \c   \d   \E   \A --\b;
    \draw (0,0)   \B   \A   \c   \A   \c   \d   \E   \A   \c   \d --\b;
    \end{tikzpicture}
};

        \draw[->,bend right] (worm) to node[align=left,left] {one\\resolution} (neg);
        \draw[->,bend left]  (worm) to node[align=left,right] {another\\resolution} (pos);
\end{tikzpicture}
\end{center}
\caption{An illustration of an unresolved Conway worm made of two kinds of
hexagons together with its two resolutions within a Penrose tiling.}
\label{fig:conway-worm-penrose}
\end{figure}
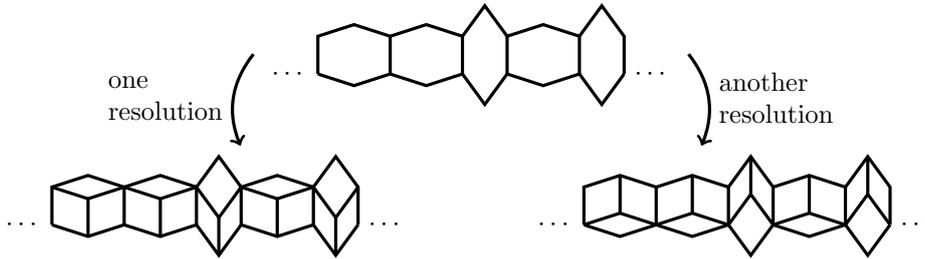

In the context of subshifts, the concept of Conway worms
is formalized in terms of nonexpansiveness. 
Let $F$ be a subspace of $\RR^d$. Given $t>0$, the \emph{$t$-neighborhood} of $F$ is defined by $F^t:=\{g\in\ZZ^d\colon\dist(g,F)\leq t\}$. Let $X\subset\Acal^\ZZd$ be a subshift, and for any subset $S \subset\ZZ^d$ and configuration $x \in X$, let $x|_S$ denote the restriction of $x$ to $S$. Following Boyle and Lind \cite{MR1355295}, a subspace $F\subset\RR^d$ is \emph{expansive} on $X$ if there exists $t>0$ such that for any $x,y \in X$, $x|_{F^t}=y|_{F^t}$ implies that $x=y$. Moreover, a subspace $F$ is \emph{nonexpansive} if for all $t > 0$, there exist $x,y \in X$ such that $x|_{F^t} = y|_{F^t}$ but $x \neq y$. 
If $F$ is expansive, then every translate of $F$ is expansive.
Thus, in the 2-dimensional case, which will be the focus of this article, we
refer to \emph{nonexpansive directions}. 

Boyle and Lind \cite[Theorem 3.7]{MR1355295} showed that if $X\subset\Acal^\ZZd$ is an infinite subshift, then, for each $0\leq n<d$, there exists a $n$-dimensional subspace of $\RR^d$ that is nonexpansive on $X$. Answering a question of Boyle and Lind, Hochman proved that any one-dimensional subspace in the plane $\RR^2$ occurs as the unique nonexpansive one-dimensional subspace of a $\ZZ^2$-action \cite{MR2755923}. As a consequence, Hochman proved that a set of one-dimensional subspaces occurs as the set of nonexpansive directions for a subshift $X\subset\Acal^{\ZZ^2}$ if and only if it is closed and non-empty. The notions of expansive and nonexpansive directions was used to obtain partial results toward solving Nivat's conjecture, an important problem in symbolic dynamics, see for instance \cite{MR3356945,colle_nivats_2019}.

The notion of nonexpansive direction can also be stated equivalently in terms of 
nonexpansive half-spaces.
Let $X\subset\Acal^\ZZd$ be a subshift and $\sigma$ be a $\Z^d$-action on $X$.
We say that a half-space $H\subset\R^d$ is \emph{nonexpansive} for $\sigma$ if there exist
$x,y \in X$ such that $x|_{\Z^d\cap H} = y|_{\Z^d\cap H}$ but $x \neq y$.
It was proved in the preliminary section of \cite{MR1869066} that a codimension
1 subspace $V$ of $\R^d$ is nonexpansive for $\sigma$ if and only if there is
a half-space $H$ whose boundary is $V$ and which is nonexpansive for $\sigma$,
see \Cref{lem:half-space-is-good}.
The set of nonexpansive directions is difficult to compute in general and
brings a deeper understanding of a subshift since it is a topological
invariant, see Lemma~\ref{lem:top_invar}.

Conway worms can be defined in the context of subshifts from nonexpansiveness.
Let $x,y\in X$ be two configurations.
The support of positions where $x$ and $y$ are distinct is the set
$D(x,y)=\{\bn\in\Z^d\mid x_\bn\neq y_\bn\}$.
We say that the set $D(x,y)$ is a 
\emph{Conway worm} associated to a subspace $F$ if
there exists $t>0$ such that $\varnothing\neq D(x,y)\subset F^t$.
Observe that if $S\subset\Z^d$ is a Conway worm associated to a subspace $F$,
then $F$ is nonexpansive.
Also, reusing the vocabulary proposed in \cite{MR1355301},
we say that the restriction of the configurations $x$ and $y$
to the support $D(x,y)$ are two \emph{resolutions of the Conway
worm}.
In this article, we are interested in describing
the Conway worms and their resolutions in the Jeandel-Rao Wang shift.

\subsection*{The Jeandel-Rao Wang shift}

\begin{figure}[h]
\begin{center}
    \includegraphics[scale=.9]{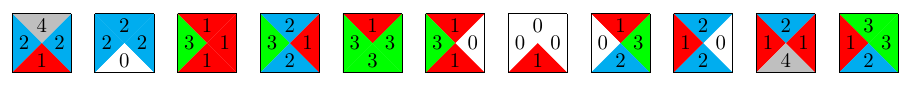}
\end{center}\caption{The aperiodic set $\T_0$ of 11 Wang tiles discovered by Jeandel and Rao in 2015 \cite{JR1}}
\label{fig:JR-tile-set}
\end{figure}

\begin{figure}[h]
\begin{center}
    \includegraphics[scale=.9]{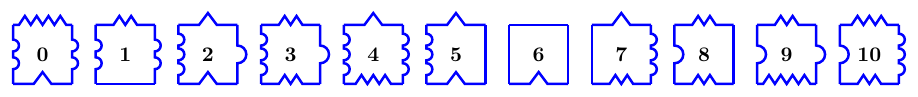}
\end{center} 
\caption{Jeandel-Rao tiles can be encoded into a set of equivalent geometrical shapes in the sense that every tiling using Jeandel-Rao tiles can be transformed into a unique tiling with the corresponding geometrical shapes and vice versa.} \label{fig:JR-tile-set-geometric} 
\end{figure}

    Wang tiles are unit squares with labeled edges. From a finite collection $\mathcal{T}$ of Wang tiles, called a \emph{protoset}, we place copies of tiles from $\mathcal{T}$ at points of $\ZZ^2$ to form a configuration. Thus, a \emph{configuration} is a map $x:\ZZ^2 \rightarrow \mathcal{T}$ where we think of $\mathcal{T}$ as a finite alphabet (i.e., $x \in \mathcal{T}^{\ZZ^2}$).  Such a configuration is \emph{valid} if it assigns tiles to $\ZZ^2$ so that contiguous edges have the same color. Let $\Omega_{\mathcal{T}} \subset \mathcal{T}^{\ZZ^2}$ denote the set of valid configurations; we call this the \emph{Wang shift} of $\mathcal{T}$. Because of the finiteness of the protoset $\mathcal{T}$, it is seen that there is a finite set of forbidden patterns (i.e., when two Wang tiles from $\mathcal{T}$ may not meet along an edge) so that $\Omega_{\mathcal{T}}$ is a shift of finite type (SFT).
    A configuration $x \in \Omega_{\mathcal{T}}$ is \emph{periodic} if there
    exists $u \in \ZZ^2 \setminus \{0\}$ such that $x + u = x$. If
    $\Omega_{\mathcal{T}} \neq \varnothing$ and $x$ is not periodic for all $x
    \in \Omega_{\mathcal{T}}$, then we say the protoset $\mathcal{T}$ is
    \emph{aperiodic}. 

The aperiodic set of 11 Wang tiles discovered by Jeandel and Rao \cite{JR1}
is shown in Figure~\ref{fig:JR-tile-set}. An equivalent geometrical
representation of Jeandel-Rao tiles is shown in Figure~\ref{fig:JR-tile-set-geometric}.
    
\subsection*{Main results}
As noticed in \cite{JR1}, there exist tilings of the plane containing a
bi-infinite horizontal strip of tiles numbered 0. Since only a tile numbered 9 can be on top of a tile numbered 0, we have the following bi-infinite strip of height 2:
\begin{center}
  $\cdots$ 
    \includegraphics[width=.8\linewidth]{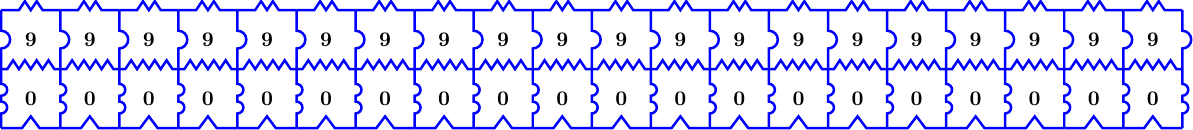}
  $\cdots$ 
\end{center}
It turns out that the above strip is a Conway worm. Indeed its other resolution can be
obtained by replacing the tiles numbered 9 with tiles numbered 1 and replacing the tiles numbered 0 with tiles numbered 6: 
\begin{center}
  $\cdots$ 
  \includegraphics[width=.8\linewidth]{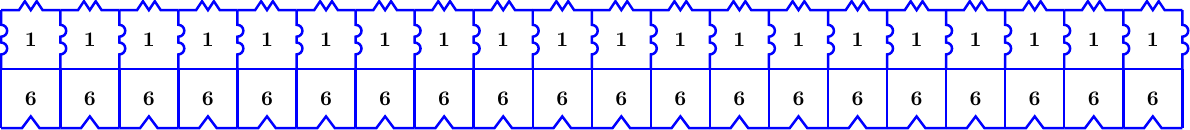}
  $\cdots$ 
\end{center}
We observe that both strips have the same constraints on top and at the bottom
making them replaceable by one another.
The tiling above and below of the two strips is shown in \Cref{fig:partial-tiling-70}.
This means that 0 is the slope of a nonexpansive direction within Jeandel-Rao Wang shift.
\begin{figure}[h]
\begin{center}
    \begin{tikzpicture}[scale=.45]
        \node at (0,0) {\includegraphics[scale=.45]{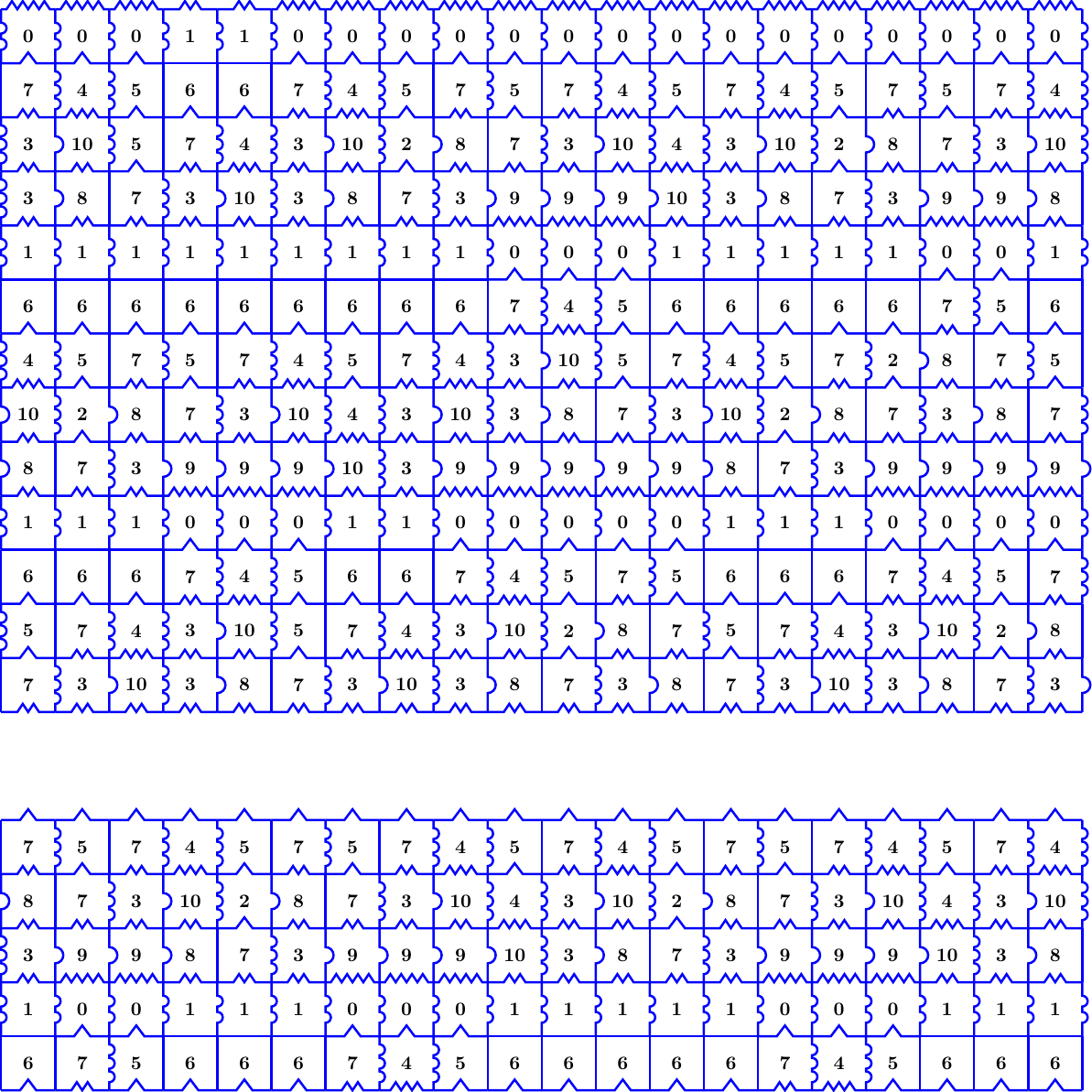}};
        \draw[dashed] (-14,-4) -- node[sloped,near end,fill=white] {Conway worm of slope 0} (14,-4);
    \end{tikzpicture}
\end{center}
    \caption{A partial tiling of the plane with an unresolved Conway worm of slope 0.}
    \label{fig:partial-tiling-70}
\end{figure}

In this article, we compute the nonexpansive directions for the minimal subshift $X_0$ of the Jeandel-Rao Wang shift $\Omega_0$.  
The description of the minimal subshift $X_0$ was given
as a subshift of finite type in \cite{MR4226493}
and as the symbolic dynamical system of a toral $\ZZ^2$-rotation coded by a polygonal partition in
\cite{Lab1}. The equality of the two descriptions was proved in \cite{Labb-Rauzy-2021}.
A review of these results are given in Section~\ref{sec:jeandel-rao-wang-shift}. 
As opposed to the nonexpansive directions in Penrose tilings which are the directions perpendicular to the fifth roots of unity, see \cite[Theorem 5.1.1]{jang_directional_2021}, we obtain a more surprising and far less symmetric result for the minimal subshift $X_0$.

\begin{maintheorem}\label{thm:main-theorem}
    The minimal subshift $X_0$ of the Jeandel-Rao Wang shift 
    contains exactly 4 nonexpansive directions whose slopes are
    $\{0,\varphi+3,-3\varphi+2,-\varphi+\frac{5}{2}\}$. \label{thm:main}
\end{maintheorem}

While slope 0 is not a surprise, the other slopes are irrational and their values are unexpected.
In particular, we show that there is a link between the slopes that appear in the Markov partition
provided in \cite{Lab1} and studied more deeply in \cite{Labb-Rauzy-2021}
and the slopes of nonexpansive directions, but the relation is not equality.
This contrasts with well-known cases like Penrose tilings where the symmetry of
the tilings hides a more complex relation.
More precisely, we show that slopes of nonexpansive directions within
Jeandel-Rao Wang shift are related to slopes that appear in the associated Markov
partition according to the following table (see
Proposition~\ref{prop:orbit-remaining-in-partition-boundary}):
\begin{center}
    \begin{tabular}{c|c}
    slope in the Markov Partition & slope of associated nonexpansive direction\\
    \hline
    0           & $0$\\
    $\infty$    & $\varphi + 3$\\
    $\varphi$   & $-3\varphi+2$\\
    $\varphi^2$ & $-\varphi+\frac{5}{2}$\\
\end{tabular}
\end{center}
The three other nonexpansive directions are illustrated in 
\Cref{fig:other-3-Conway-worms-in-JR}
and \Cref{fig:all-Conway-worms-in-JR-in-one-image}.
In Theorem~\ref{thm:resolution} stated in \Cref{sec:resolution}, 
we describe the Conway worms associated to each of the nonexpansive directions
in the Jeandel-Rao Wang shift as well as their resolutions.

\begin{figure}
\begin{center}
\begin{tabular}{c|c}
  \includegraphics[width=.47\linewidth]{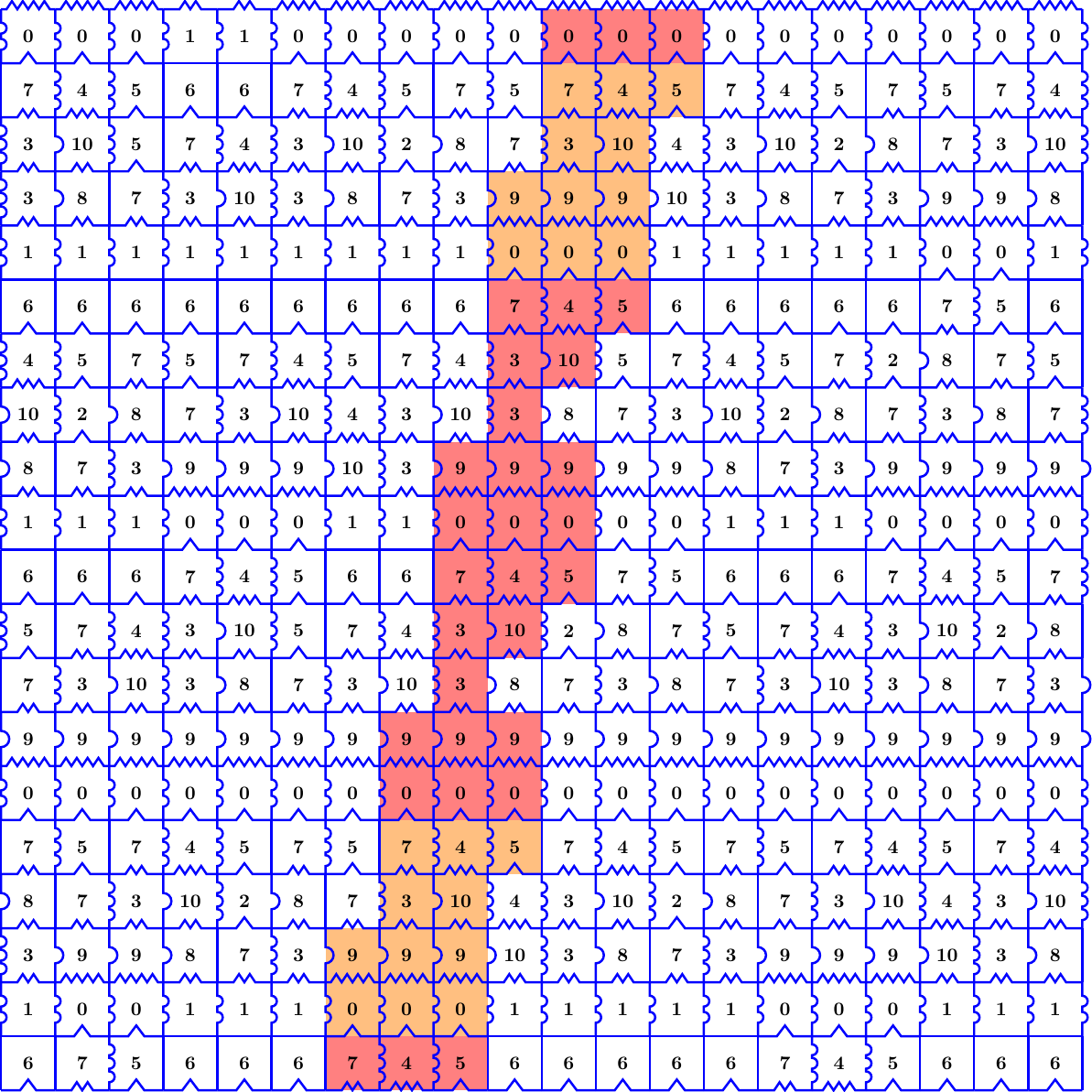}
& \includegraphics[width=.47\linewidth]{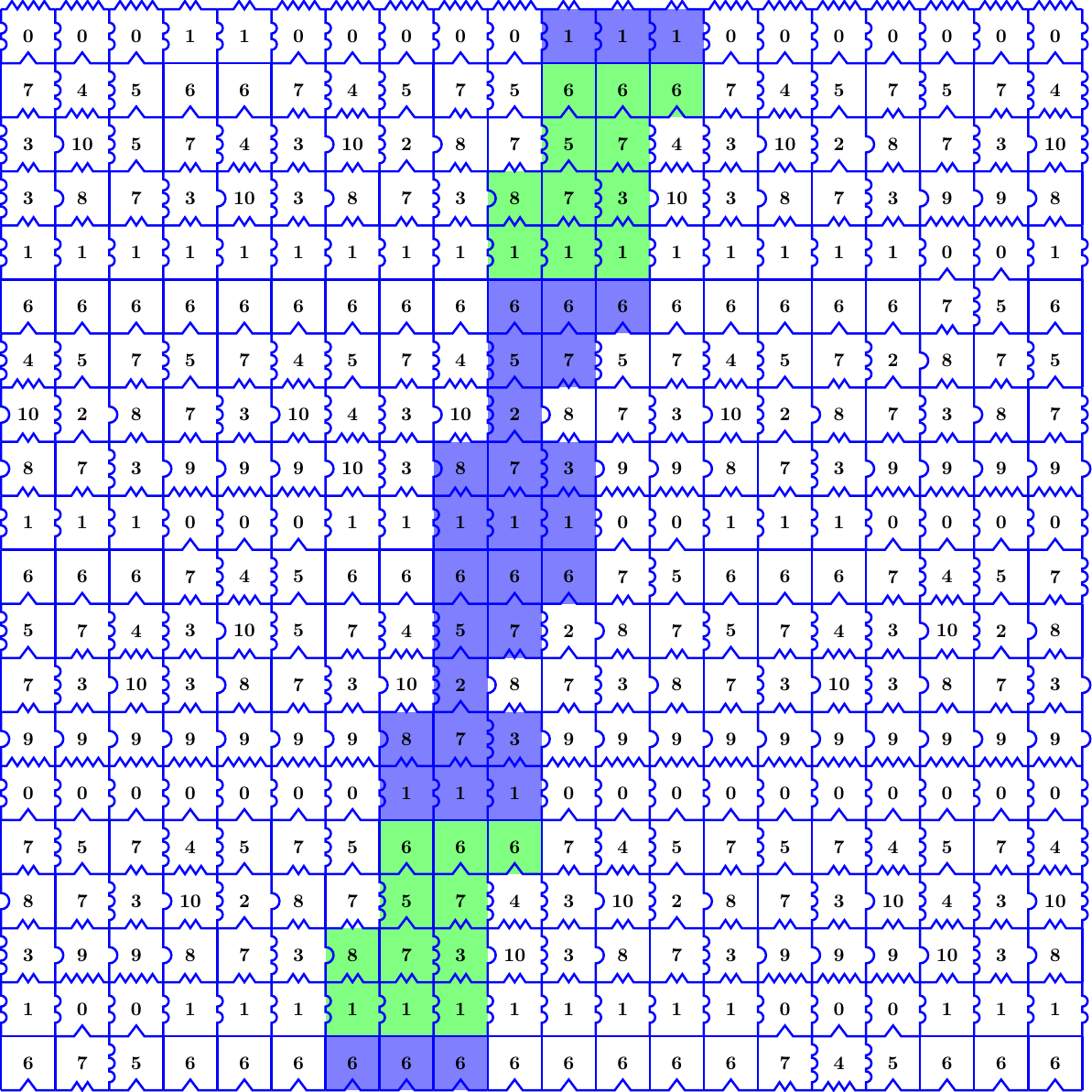}\\[2mm]
  \includegraphics[width=.47\linewidth]{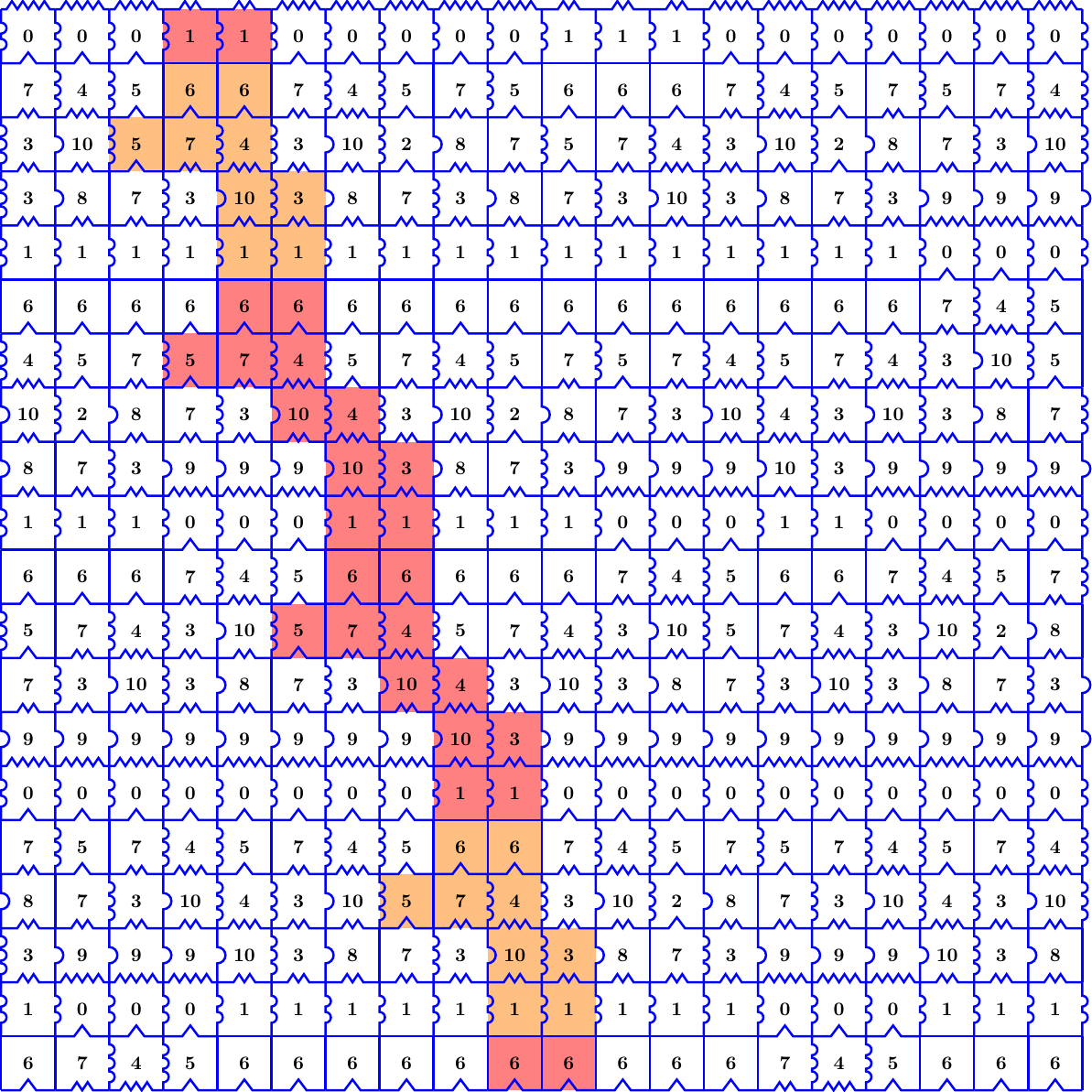}
& \includegraphics[width=.47\linewidth]{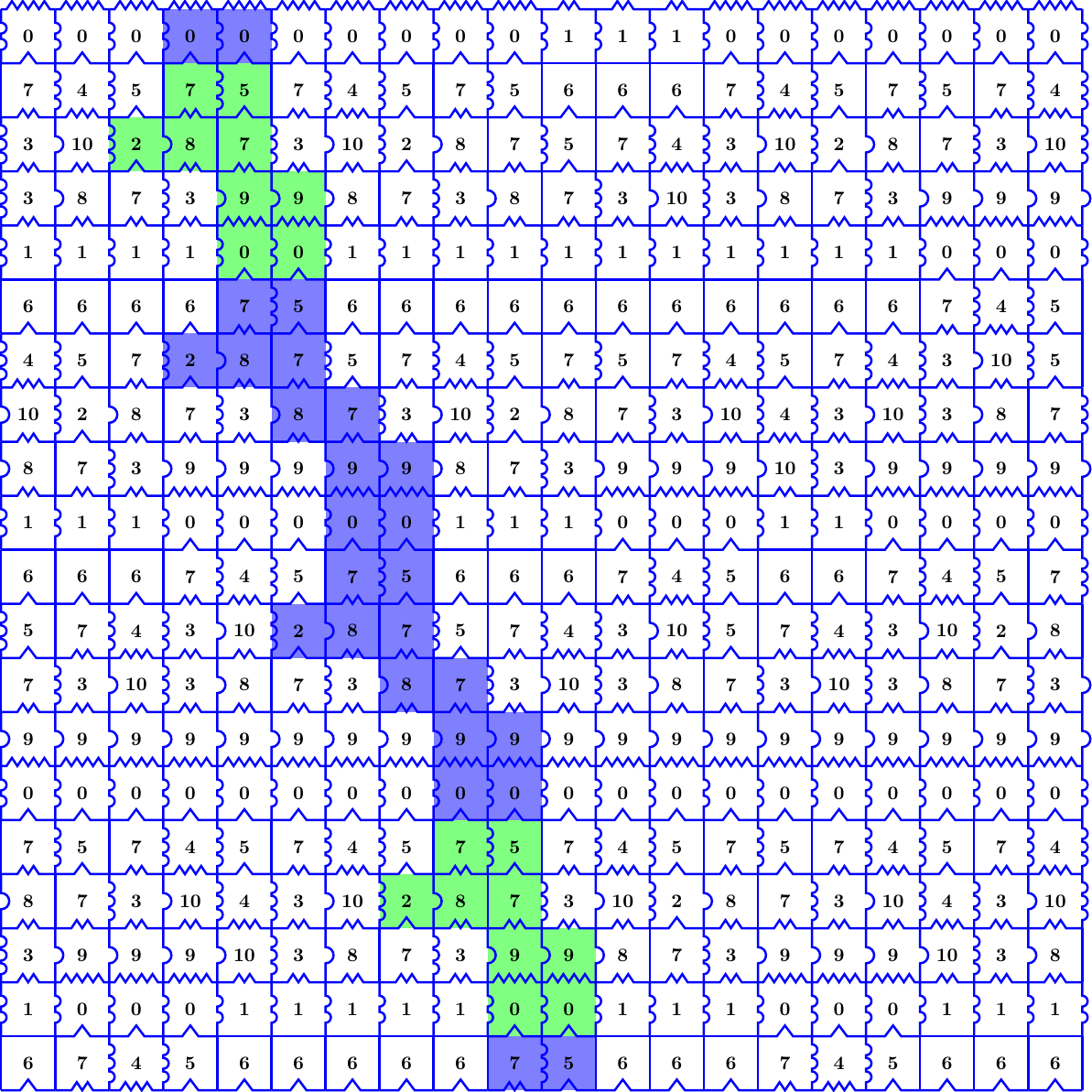}\\[2mm]
  \includegraphics[width=.47\linewidth]{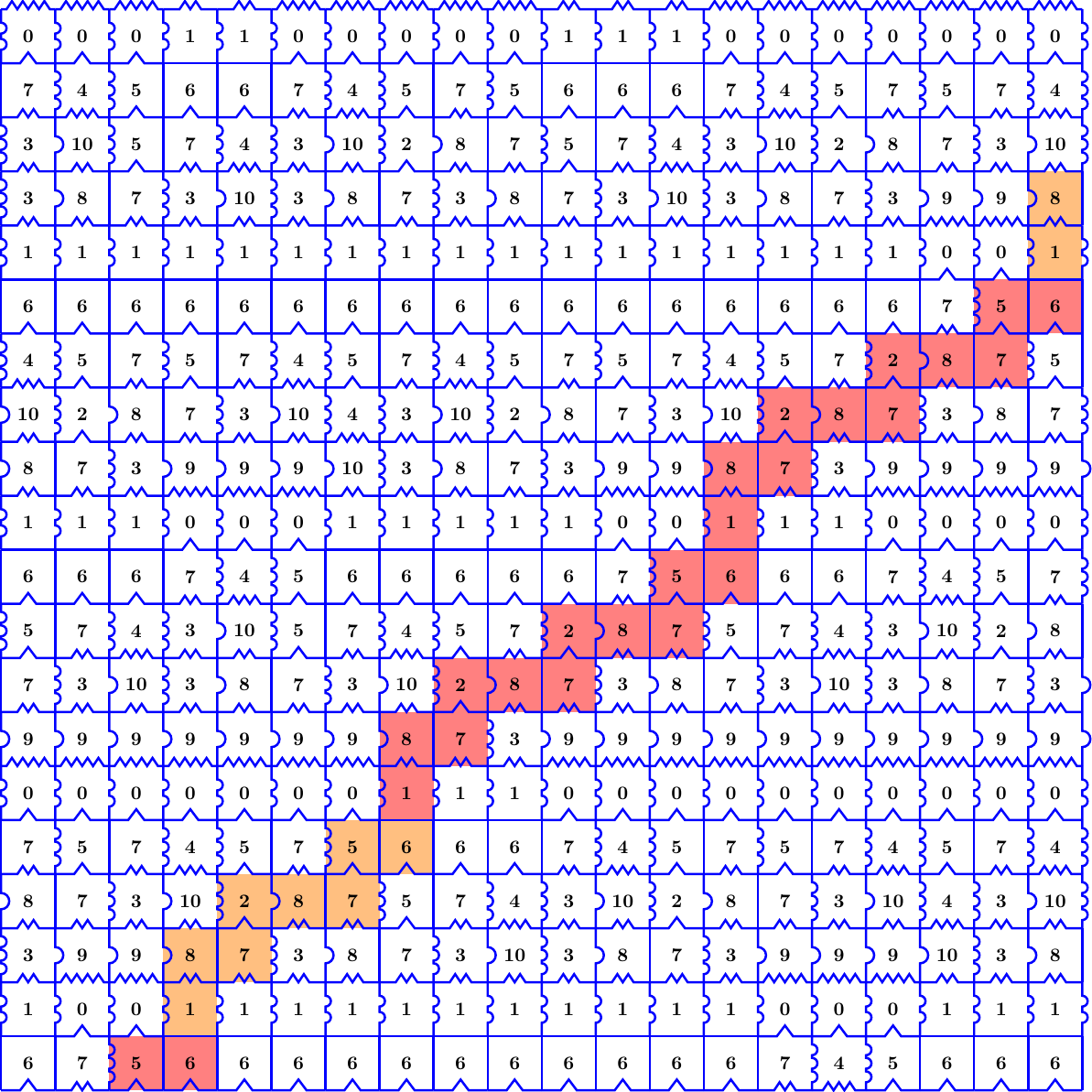}
& \includegraphics[width=.47\linewidth]{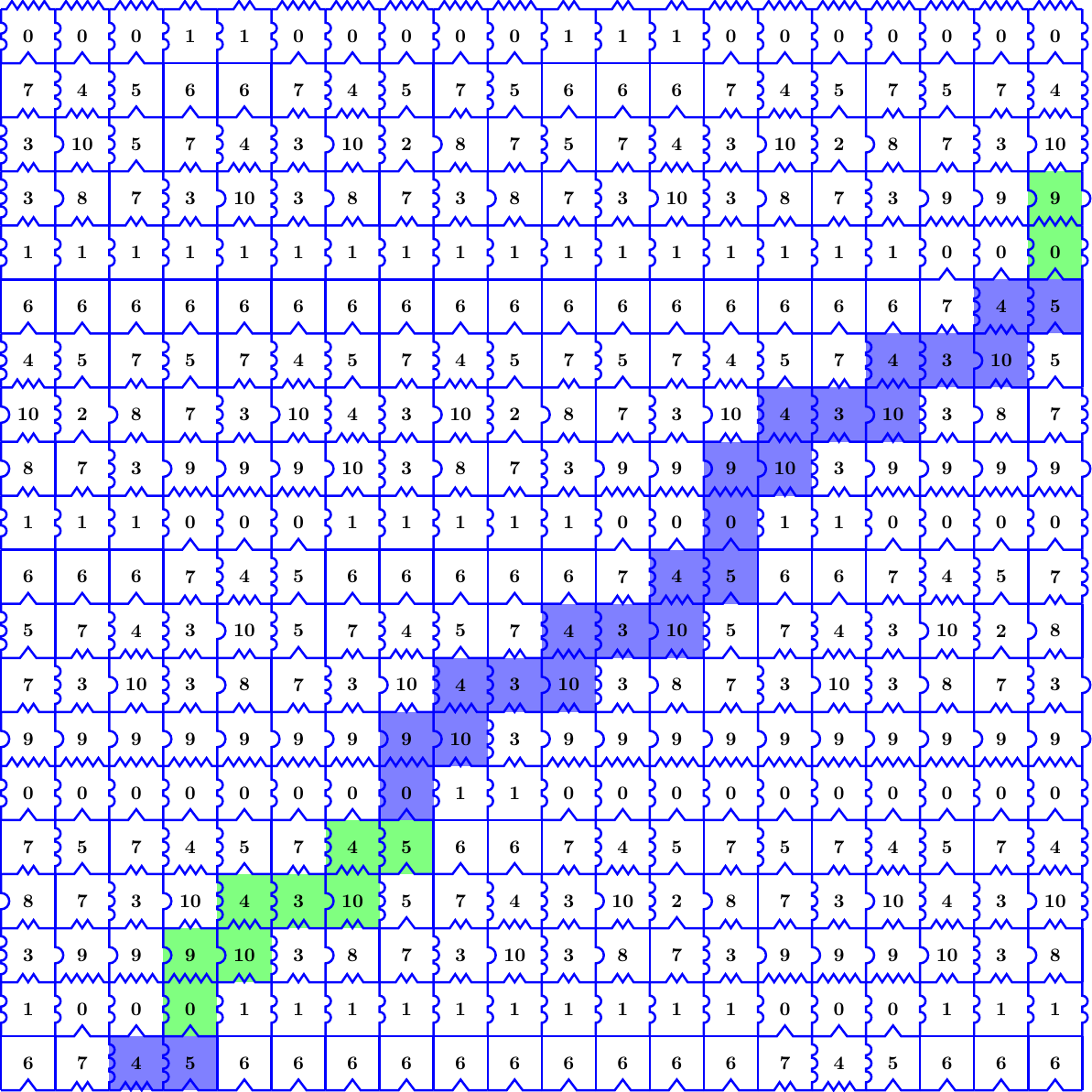}
\end{tabular}
\end{center}
    \caption{Tilings of a $20\times20$ square illustrating the Conway worms 
    of slope $\varphi+3$, $-3\varphi+2$ and $-\varphi+\frac{5}{2}$.
    The difference between the left and the right images is shown with a colored
    background.}
    \label{fig:other-3-Conway-worms-in-JR}
\end{figure}

\begin{figure}
\begin{center}
\begin{tabular}{c}
  \includegraphics[width=.72\linewidth]{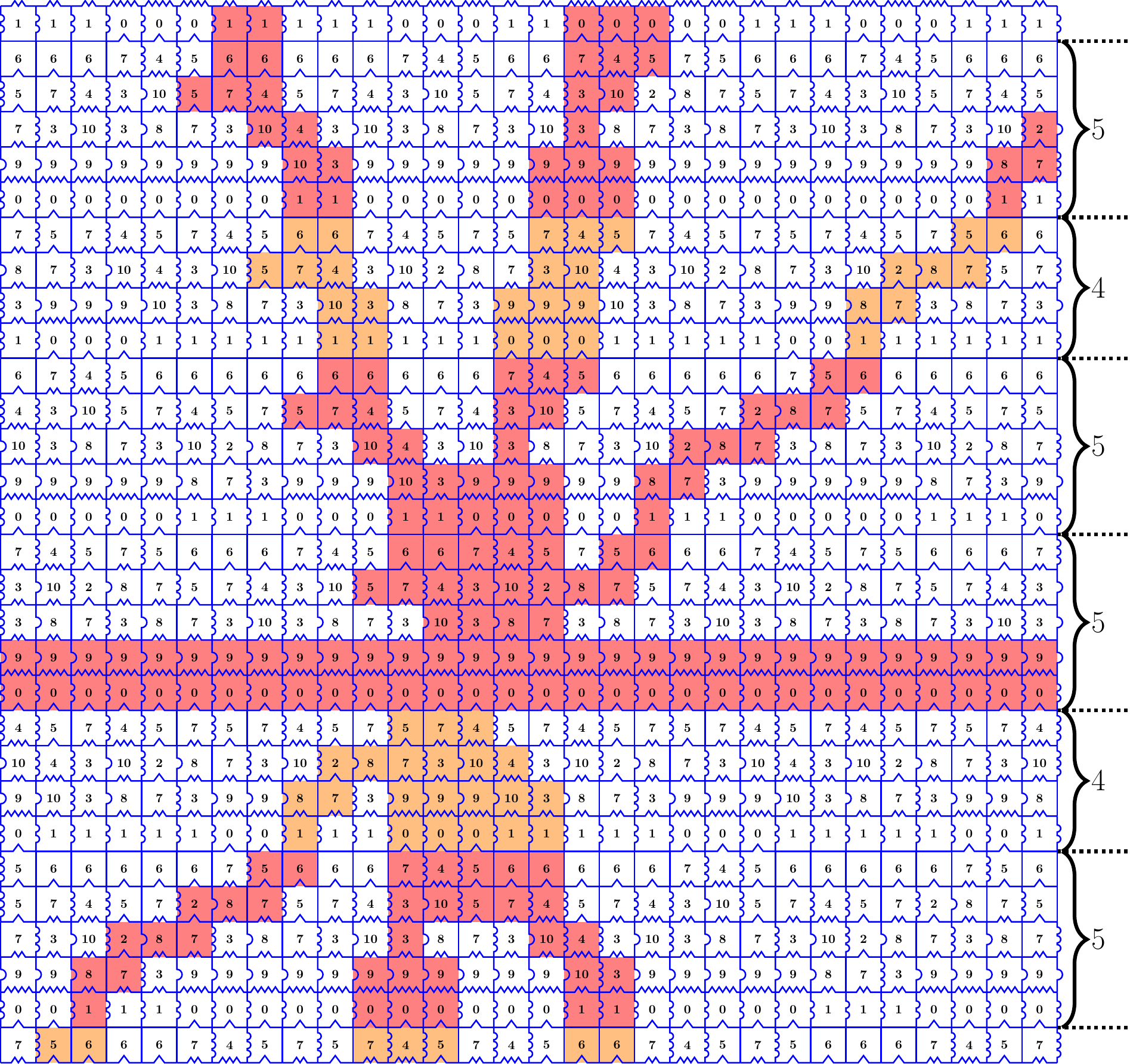}\\
  \includegraphics[width=.72\linewidth]{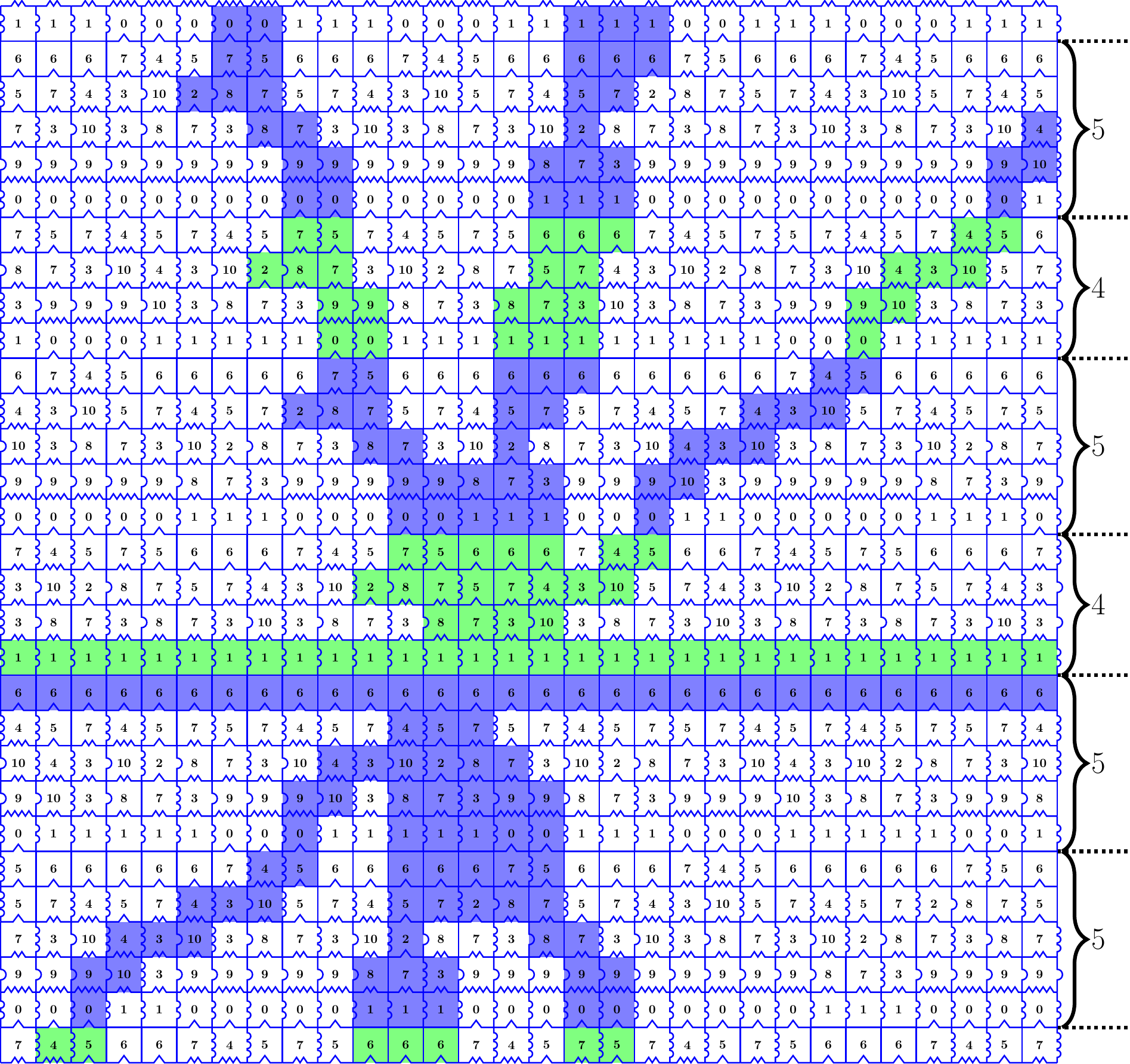}
\end{tabular}
\end{center}
    \caption{Tilings of a $30\times30$ square illustrating the four Conway worms.
    The difference between both images is shown with a colored background.
    This reminds of the cartwheel tiling in the context of Penrose tilings
    \cite[Figure 10.5.1 (c)]{GS1}.}
    \label{fig:all-Conway-worms-in-JR-in-one-image}
\end{figure}

\subsection*{Structure of the article }

In \Cref{section:dyn_sys}, we present notions from topological dynamical
systems.
In \Cref{sec:Nonexpansive-in-SDS}, we consider
nonexpansive directions in the context of minimal symbolic dynamical system 
corresponding to a $\Z^2$-action on a 2-dimensional torus and a partition of this torus.
The computation of the nonexpansive directions 
is reduced to the computations of sub-orbits under the $\Z^2$-action 
which remain in the boundary of the partition.
In \Cref{sec:jeandel-rao-wang-shift}, we recall previous results on
the Jeandel-Rao Wang shift associated to some particular polygonal partition of
a 2-dimensional torus.
In \Cref{sec:calculating}, we compute the slopes of nonexpansive directions
in the Jeandel-Rao Wang shift from the slopes of appearing in the polygonal partition.
In \Cref{sec:resolution}, we express
the resolution of the Conway worms in the Jeandel-Rao Wang shift
in terms of sequences in the Fibonacci subshift.
In \Cref{sec:octopods}, we propose open questions about octopods and essential
holes within the Jeandel-Rao Wang shift similarly to the 62 decapods known for
Penrose tilings.

\subsection*{Acknowledgments}
We are thankful to the reviewer for their valuable comments leading to a
improved presentation of the results. This work was supported by the Agence Nationale de la Recherche through the
projects ANR CODYS (ANR-18-CE40-0007) and ANR IZES (ANR-22-CE40-0011). The third author acknowledges support from the PIMS Europe Fellowship. The second and third author acknowledge support from Idex Bordeaux Visiting Scholars program.

\section{Topological Dynamical Systems}\label{section:dyn_sys}
We begin by describing a useful framework for understanding the Jeandel-Rao shift as a dynamical system. Most of the notions introduced here can be found in \cite{WaltersPeter1982Aite}. A \emph{dynamical system} is a triple $(X, G, T)$, where $X$ is a topological space, $G$ is a topological group and $T$ is a continuous function $G \times X \rightarrow X$ defining a left action of $G$ on $X$: if $x \in X$, $e$ is the identity element of $G$, and $g, h \in G$, then using additive notation for the operation in $G$ we have $T(e, x) = x$ and $T(g + h, x) = T(g, T(h, x))$. In other words, if one denotes the transformation $x \mapsto T(g, x)$ by $T^g$, then $T^{g+h} = T^gT^h$. The \emph{orbit} of a point $x \in X$ under the left action of $G$ by $T$ is the set $\mathcal{O}_T(x) = \{T^g(x):g \in G\}$.

If $Y \subset X$, let $\overline{Y}$ denote the topological closure of $Y$ and define the \emph{orbit} of $Y$ as $\mathcal{O}_T(Y) = \cup_{y \in Y}\mathcal{O}_T(y)$.  A subset $Y \subset X$ is $T$\emph{-invariant} if $\mathcal{O}_T(Y) = Y$. A dynamical system $(X, G, T)$ is called \emph{minimal} if $X$ does not contain any nonempty, proper, closed $T$-invariant subset. The left action of $G$ on $X$ is \emph{free} if $g = e$ whenever there exists $x \in X$ such that $T^g(x) = x$. 

Let $(X, G, T)$ and $(Y, G, S)$ be two dynamical systems with the same topological group $G$. A \emph{homomorphism} $\theta \!\!:\!\! (X, G, T) \rightarrow (Y, G, S)$ is a continuous function $\theta \!\!:\!\! X \rightarrow Y$ satisfying the commuting property that $T^g \circ \theta = \theta \circ S^g$ or every $g \in G$. A homomorphism $\theta \!\!:\!\! (X, G, T) \rightarrow (Y, G, S)$ is called an \emph{embedding} if it is one-to-one, a \emph{factor map} if it is onto, and a \emph{topological conjugacy} if it is both one-to-one and onto and its inverse map is continuous. If $\theta \!:\! (X, G, T) \rightarrow (Y, G, S)$ is a factor map, then $(Y, G, S)$ is called a \emph{factor} of $(X, G, T)$ and $(X, G, T)$ is called an \emph{extension} of $(Y, G, S)$. Two dynamical systems are \emph{topologically conjugate} if there is a topological conjugacy between them.

Let $\theta:(X,G,T)\to(Y,G,S)$ be a factor map.
We call the preimage set $\theta^{-1}(y)$ of a point $y\in Y$ the \emph{fiber} of $\theta$ over $y$.
The \emph{set of fiber cardinalities}
is the set $\{\card(\theta^{-1}(y)) : y \in Y\}\subset\N\cup\{\infty\}$, see \cite{MR1877329}.
Note that different terminology is used in \cite{MR1355301} as
the set of fiber cardinalities of a factor map is called \emph{thickness spectrum}
and
its supremum is called \emph{thickness} whereas
the supremum is called \emph{maximum rank} in \cite{MR3381481}.

\subsection{Maximal equicontinuous factor}
A metrizable dynamical system $(X,G,T)$ is called \emph{equicontinuous} if
the family of homeomorphisms $\{T^g\}_{g\in G}$ is equicontinuous, i.e., if for
all $\varepsilon>0$ there exists $\delta>0$ such that
\[
    \dist(T^g(x), T^g(y)) < \varepsilon
\]
for all $g\in G$ and all $x,y\in X$ with $\dist(x,y)<\delta$.
According to a well-known theorem~\cite[Theorem 3.2]{MR3381481},
equicontinuous
minimal systems defined by the action of an Abelian group
are rotations on groups.

We say that $\theta:(X,G,T)\to(Y,G,S)$ is an \emph{equicontinuous factor} if
$\theta$ is a factor map and $(Y,G,S)$ is equicontinuous.
We say that $(X_{\rm max}, G, T_{\rm max})$ is the \emph{maximal equicontinuous
factor} of $(X,G,T)$ if 
there exists an equicontinuous factor
$\pi_{\rm max}:(X,G,T)\to(X_{\rm max}, G, T_{\rm max})$,
such that for any
equicontinuous factor $\theta:(X,G,T)\to(Y,G,S)$,
there exists a unique factor map $\psi:(X_{\rm max}, G, T_{\rm max})\to(Y,G,S)$
with $\psi\circ\pi_{\rm max}=\theta$.
The maximal equicontinuous factor exists and is unique (up to topological
conjugacy), see \cite[Theorem 3.8]{MR3381481} and \cite[Theorem 2.44]{MR2041676}.

The maximal equicontinuous factor $f:(X,G,T)\to(X_{\rm max},G,T_{\rm max})$
defines an equivalence relation on the elements $a,b\in X$ as $a\equiv b$ if and only if
$f(a)=f(b)$.
A theorem of Auslander says that the equivalence relation 
is described by regionally proximal pairs.
Two elements $x, y \in X$ are said to be \emph{regionally proximal}
if there are sequences of elements $x_i,y_i\in X$ and a sequence of elements
$g_i\in G$ such that $\lim_{i\to\infty}x_i=x$, $\lim_{i\to\infty}y_i=y$ and 
$\lim_{i\to\infty} \dist(g_i x_i, g_iy_i)=0$.

\begin{theorem}\label{thm:auslander}
\cite[p.130]{MR956049} 
If $(X,G,T)$ is minimal and
    $f:(X,G,T)\to(X_{\rm max},G,T_{\rm max})$
is its maximal equicontinuous factor,
then $f(a)=f(b)$ if and only if $a$ and $b$ are regionally proximal.
\end{theorem}

	\subsection{Subshifts and Subshifts of Finite Type}\label{subsec:shifts}
	
Here we follow the notation of \cite{10.1007/978-1-4613-0165-3_3}. Let $\mathcal{A}$ be a finite set, $d \geq 1$, and let $\mathcal{A}^{\ZZ^d}$ be the set of all maps $x : \ZZ^d \rightarrow \mathcal{A}$, equipped with the compact product topology. An element $x \in \mathcal{A}^{\ZZ^d}$ is called \emph{configuration} and we write it as $x = (x_{\boldsymbol{m}}) = (x_{\boldsymbol{m}} :\boldsymbol{m} \in \ZZ^d)$, where $x_{\boldsymbol{m}} \in \mathcal{A}$ denotes the value of $x$ at $\boldsymbol{m}$. The topology on $\mathcal{A}^{\ZZ^d}$ is compatible with the metric defined for all configurations $x, x' \in \mathcal{A}^{\ZZ^d}$ by $\text{dist}(x, x') = 2^{-\min\{\|\boldsymbol{n}\|:x_{\boldsymbol{n}} \neq x'_{\boldsymbol{n}}\}}$ where $\|\boldsymbol{n}\| = |n_1| + \cdots + |n_d|$. The \emph{shift action} $\sigma :  \boldsymbol{n} \mapsto \sigma^{\boldsymbol{n}}$ of $\ZZ^d$ on $\mathcal{A}^{\ZZ^d}$ is defined by \begin{equation} (\sigma^{\boldsymbol{n}}(x))_{\boldsymbol{m}} = x_{\boldsymbol{m}+\boldsymbol{n}} \label{eq:shift_action}\end{equation} for every $x = (x_{\boldsymbol{m}}) \in \mathcal{A}^{\ZZ^d}$ and $\boldsymbol{n} \in \ZZ^d$. A subset $X \in \mathcal{A}^{\ZZ^d}$ is \emph{shift-invariant} if $\sigma(X) = X$ and a closed, shift-invariant subset $X \subset \mathcal{A}^{\ZZ^d}$ is a \emph{subshift}. If $X \subset \mathcal{A}^{\ZZ^d}$ is a subshift, we write $\sigma = \sigma^X$ for the restriction of the shift action \eqref{eq:shift_action} to $X$. When $X$ is a subshift, the triple $(X, \ZZ^d, \sigma)$ is a dynamical system.

A configuration $x \in X$ is \emph{periodic} if there is a nonzero vector $\boldsymbol{n} \in \ZZ^d \setminus \{0\}$ such that $x = \sigma^{\textbf{n}}(x)$ and otherwise it is said \emph{nonperiodic}. We say that a nonempty subshift $X$ is \emph{aperiodic} if the shift action $\sigma$ on $X$ is free. Note that our definition of aperiodicity of Wang tile protosets given in Section~\ref{sec:introduction} agrees with this more general definition of aperiodicity if we take $\mathcal{A} = \mathcal{T}$, $d = 2$, and the shift action on $\mathcal{T}^{\ZZ^2}$ being $\ZZ^2$ translation.

For any subset $S \subset \ZZ^d$, let $\pi_S : \mathcal{A}^{\ZZ^d} \rightarrow \mathcal{A}^{S}$ denote the projection map which restricts every $x \in \mathcal{A}^{\ZZ^d}$ to $S$. A pattern is a function $p \in \mathcal{A}^{S}$ for some finite subset $S \subset \ZZ^d$. To every pattern $p \in \mathcal{A}^{\ZZ^d}$ corresponds a subset $\pi_{S}^{-1}(p) \subset \mathcal{A}^{\ZZ^d}$ called \emph{cylinder}. A subshift $X \subset \mathcal{A}^{\ZZ^d}$ is a \emph{shift of finite type} (SFT) if there exists a finite set $\mathcal{F}$ of \emph{forbidden patterns} such that \begin{equation} X = \{x \in \mathcal{A}^{\ZZ^d} \,|\, \pi_S \circ \sigma^{\boldsymbol{n}}(x) \notin \mathcal{F} \text{ for all } \boldsymbol{n} \in \ZZ^d \text{ and } S \subset \ZZ^d\} \label{eqn:SFT}\end{equation} In this case, we write $X = SFT(\mathcal{F})$. In this article, we consider shifts of finite type on $\ZZ \times \ZZ$; that is, the case $d = 2$.
    Wang shifts, as defined in the introduction, are shifts of finite type on $\ZZ^2$.

\subsection{Nonexpansive directions}

The following definitions are taken from \cite[\S 2]{colle_nivats_2019}. Let $F$ be a subspace of $\RR^d$. For each $g\in\ZZd$, let $\dist(g,F)=\inf\{\Vert g-u\Vert\colon u\in F\}$, where $\Vert\cdot\Vert$ is the Euclidean norm in $\RR^d$. Given $t>0$, the \emph{$t$-neighborhood} of $F$ is defined by $F^t:=\{g\in\ZZ^d\colon\dist(g,F)\leq t\}$. Let $X\subset\Acal^\ZZd$ be a subshift. Following Boyle and Lind \cite{MR1355295}, we say that a subspace $F\subset\RR^d$ is \emph{expansive} on $X$ if there exists $t>0$ such that for any $x,y \in X$, $x|_{F^t}=y|_{F^t}$ implies that $x=y$. Thus, a subspace $F$ is \emph{nonexpansive} if for all $t > 0$, there exist $x,y \in X$ such that $x|_{F^t} = y|_{F^t}$ but $x \neq y$. 

Additionally, we see that if $F$ is expansive, then every translate of $F$ is expansive since $F^t \subset (F + \boldsymbol{v})^{t + \norm{\boldsymbol{v}}}$ for any $\boldsymbol{v} \in \RR^d$.   Thus, in the 2-dimensional case, which will be the focus of this article, we may refer to \emph{nonexpansive directions}. 

Equivalently, the notion of expansiveness can be defined on half-spaces in
$\RR^d$ rather than codimension one subspaces \cite[\S 2]{MR1869066}. 
Let $\mathsf{S}_{d-1}=\{\bv\in\RR^d\colon\Vert\bv\Vert=1\}$ be the unit $(d-1)$-sphere.
For $\bv\in \mathsf{S}_{d-1}$ define $H_\bv=\{\bx\in\RR^d\colon \langle\bx,\bv\rangle\leq0\}$ to be the 
half-space with outward unit normal $\bv$. Let $\mathsf{H}_d$ be the set of half-spaces in $\RR^d$,
which are identified with $\mathsf{S}_{d-1}$ via the parametrization $\bv\leftrightarrow H_\bv$.
For $H\in\mathsf{H}_d$, we denote its outward unit normal vector by $\bv_H$.
Let $\sigma$ be a $\Z^d$-action on the subshift $X$.
We say that a half-space $H\in\mathsf{H}_d$ is \emph{nonexpansive} for $\sigma$ if there exist
$x,y \in X$ such that $x|_{\Z^d\cap H} = y|_{\Z^d\cap H}$ but $x \neq y$.

\begin{lemma}{\rm\cite[Lemma 2.9]{MR1869066}}
    \label{lem:half-space-is-good}
    Let $\sigma$ be a $\Z^d$-action and $V$ be a codimension 1 subspace of $\R^d$.
    Then $V$ is nonexpansive for $\sigma$ if and only if
    there is a half-space $H\in\mathsf{H}_d$ 
    which is nonexpansive for $\sigma$ with $\partial H=V$.
\end{lemma}

Thus if $F$ is a nonexpansive codimension 1 subspace for a subshift $X$, then,
there exist $x,y\in X$ such that $x|_{\Z^d\cap H} = y|_{\Z^d\cap H}$ but $x\neq
y$ where $H$ is the half-space on one side of the space $F$.

The next lemma shows that the set of nonexpansive directions of a subshift (for
instance the one computed in \Cref{thm:main-theorem}) is a topological
invariant.

\begin{lemma}
    Let
    $(X,\Z^d,f)$
    and $(Y,\Z^d,g)$ be two topologically conjugate subshifts
    and $F\subset\R^d$ be a codimension 1 subspace.
    If $F$ is a nonexpansive in $X$, then $F$ is nonexpansive in $Y$.
    \label{lem:top_invar}
\end{lemma}

\begin{proof}
    Let $\theta:X\to Y$ be the topological conjugacy such that
    $\theta\circ f^\bn=g^\bn\circ\theta$ for all $\bn\in\Z^d$.
    Since $Y$ is a subshift,
    there exists $\delta>0$ such that for all $x,y\in Y$
    and for all $\bn\in\Z^d$,
    we have
    $d_Y(g^\bn(x),g^\bn(y))<\delta$ implies that $x_\bn=y_\bn$.
    From the continuity of $\theta$, there exists $\delta'>0$ such that
    for all $x,y\in X$
    we have
    $d_X(x,y)<\delta'$ implies that $d_Y(\theta(x),\theta(y))<\delta$.

    Based on Lemma~\ref{lem:half-space-is-good},
    we do the proof for nonexpansive half-spaces instead of nonexpansive
    codimension 1 spaces.
    Let $H\in\mathsf{H}_d$ be a half-space which is nonexpansive for 
    $\Z^d\overset{f}{\curvearrowright}X$. 
    Thus there exist $x,y \in X$ such that $x|_{\Z^d\cap H} = y|_{\Z^d\cap H}$
    but $x \neq y$.
    Since 
    $x$ and $y$ agree on a half-space
    and $X$ is a subshift,
    there exists $t<0$
    such that
    $d_X(f^{\bn}(x),f^{\bn}(y))<\delta'$
    for all $\bn\in\Z^d$ such that $\langle\bn,\bv_H\rangle<t$.
    Therefore
    from the continuity of $\theta$, we have
    $d_Y(\theta(f^{\bn}x),\theta(f^{\bn}y))<\delta$
    for all $\bn\in\Z^d$ such that $\langle\bn,\bv_H\rangle<t$.
    Therefore, 
    $\theta(x)_\bn= \theta(y)_\bn$
    for all $\bn\in\Z^d$ such that $\langle\bn,\bv_H\rangle<t$.
    Let $\bm\in\Z^d$ such that $\langle\bm,\bv_H\rangle<t$.
    We have
    $g^\bm(\theta(x))|_{\Z^d\cap H} = g^\bm(\theta(y))|_{\Z^d\cap H}$
    but
    $g^\bm(\theta(x))
    \neq g^\bm(\theta(y))$.
    Therefore $H$
    is nonexpansive for 
    $\Z^d\overset{g}{\curvearrowright}Y$. 
\end{proof}

Nonexpansive half-spaces (or \emph{non-deterministic directions}) are used to deduce results about the invertibility of any endomorphism of substitutive subshifts and also about the structure of its normalizer group. For more details, see \cite{cabezas2021homomorphisms}.

\subsection{Conway worms}

Adapting a terminology which was originally defined for tilings of the plane,
we define the notion of Conway worms in the context of subshifts as follows.

\begin{definition}[Conway worm]
Let $X\subset \Acal^{\Z^d}$ be a subshift.
We say that a support $S\subset\Z^d$
is a \emph{Conway worm} associated to a subspace $F$ if
there exist two configurations $x,y\in X$
such that $S=\{\bn\in\Z^d\mid x_\bn\neq y_\bn\}$
and there exists $t>0$ such that $\varnothing\neq S\subset F^t$.

Also, we say that the restriction of the configurations $x$ and $y$
to the support $S$ are two \emph{resolutions of the Conway
worm}.
\end{definition}

\begin{remark}
Observe that if $S\subset\Z^d$ is a Conway worm associated to a subspace $F$,
then $F$ is nonexpansive.
\end{remark}

In this article, we are interested in describing
the Conway worms and their resolutions in the Jeandel-Rao Wang shift.

\section{Nonexpansive directions in Symbolic dynamical systems}\label{sec:Nonexpansive-in-SDS}

\subsection{Symbolic dynamical systems}

We follow the section \cite[\S6.5]{LindDouglas1995AItS} on Markov partitions where we adapt it to the case of invertible $\ZZ^2$-actions. A \emph{topological partition} of a metric space $M$ is a finite collection $\{P_0, P_1,\ldots, P_{r-1}\}$ of disjoint open sets such that $M = \overline{P_0} \cup \overline{P_1} \cup \cdots \cup \overline{P_{r-1}}.$ Suppose that $M$ is a compact metric space, $(M, \ZZ^2,R)$ is a dynamical system, and that $\mathcal{P} = \{P_0, P_1,\ldots, P_{r-1}\}$ is a topological partition of $M$. Let $\mathcal{A} = \{0,1,\ldots,r-1\}$ and $S \subset \ZZ^2$ be a finite set. We say that a \emph{pattern} $w \in \mathcal{A}^S$ is \emph{allowed} for $\mathcal{P},R$ if 
\[
   \bigcap_{\boldsymbol{k}\in S}R^{-\boldsymbol{k}}(P_{w_{\boldsymbol{k}}}) \neq \varnothing.
\] 
Let $\Lcal_{\Pcal,R}$ be the collection of all allowed patterns for $\Pcal,R$.
The set $\Lcal_{\Pcal,R}$ is the language of a subshift 
$\Xcal_{\Pcal,R}\subseteq\Acal^{\Z^2}$ defined as follows,
see \cite[Prop.~9.2.4]{MR3525488},
\[
    \Xcal_{\Pcal,R} = 
    \{x\in\Acal^{\Z^2} \mid \pi_S\circ\sigma^\bn(x)\in\Lcal_{\Pcal,R}
    \text{ for every } \bn\in\Z^2 \text{ and finite subset } S\subset\Z^2\}.
\]

\begin{definition}
We call $\Xcal_{\Pcal,R}$ the \emph{symbolic dynamical
system} corresponding to $\Pcal,R$.
\end{definition}

For each $w \in \mathcal{X}_{\mathcal{P},R} \subset \mathcal{A}^{\ZZ^2}$ and $n \geq 0$ there is a corresponding nonempty open set 
\[D_n(w) = \bigcap_{\|\boldsymbol{k}\| \leq n} R^{-\boldsymbol{k}}(P_{w_{\boldsymbol{k}}}) \subseteq M.\] 
The closures $\overline{D}_n(w)$ of these sets are compact and decrease with $n$ in the sense that that $\overline{D_0}(w) \supseteq \overline{D_1}(w) \supseteq \overline{D_2}(w) \supseteq \cdots$. It follows that $\cap_{n = 0}^{\infty} \overline{D_n}(w) \neq \varnothing$. In order for configurations in $\mathcal{X}_{\mathcal{P},R}$ to correspond to points in $M$, this intersection should contain only one point. This leads to the following definition. 

\begin{definition} A topological partition $\mathcal{P}$ of $M$ gives a \textbf{\emph{symbolic representation}} of $(M, \ZZ^2,R)$ if for every $w \in \mathcal{X}_{\mathcal{P},R}$, the intersection $\cap_{n = 0}^{\infty}\overline{D_n}(w)$ consists of exactly one point $m \in M$. We call $w$ a \textbf{\emph{symbolic representation}} of $m$.  \label{dfn:symb_rep}\end{definition}

An important consequence of the fact that
a partition $\Pcal$ gives a symbolic representation of the dynamical system
$(M,\Z^2,R)$ is the existence of a
factor map $f:\Xcal_{\Pcal,R}\to M$ which commutes the
$\Z^2$-actions.
In the spirit of \cite[Prop.~6.5.8]{LindDouglas1995AItS} for $\Z$-actions,
we have the following proposition whose
proof can be found in \cite{Lab1}
when the compact metric space $M$ is a 2-dimensional torus.

From now one in this section, we assume that $M=\boldsymbol{T} = \RR^2 / \Gamma$
for some lattice $\Gamma$ in $\RR^2$, i.e., a discrete subgroup of the additive
group $\RR^2$ with 2 linearly independent generators,
and that $R$ is a $\ZZ^2$-rotation on $\boldsymbol{T}$. 

Let
\[
    \Delta_{\Pcal} := \bigcup_{a \in \mathcal{A}}\partial P_a 
\]
be the \emph{boundary} of $\mathcal{P}$
and
\[
    \Delta_{\mathcal{P},R} 
    = \mathcal{O}_R
    \left(\Delta_\Pcal\right)
    \subset \boldsymbol{T}
\] 
be the set of points whose orbits under the toral $\ZZ^2$-rotation $R$ intersect the boundary of the topological partition $\mathcal{P}$. We note that $\Delta_{\mathcal{P},R}$ is dense in $\boldsymbol{T}$, by the Baire Category Theorem \cite[Theorem 6.1.24]{LindDouglas1995AItS}. 

\begin{proposition}\label{prop:factor-map}
    {\rm\cite[Prop.~5.1]{Lab1}}
    Let $\Pcal$ give a symbolic representation of the dynamical system
    $(\boldsymbol{T},\Z^2,R)$ such that
    $R$ is a $\ZZ^2$-rotation on $\boldsymbol{T}$.
    Let $f:\Xcal_{\Pcal,R}\to \boldsymbol{T}$ be defined 
    such that $f(w)$ is the unique point
    in the intersection $\cap_{n=0}^{\infty}\overline{D}_n(w)$.
    The map $f$ is a factor map from
            $(\Xcal_{\Pcal,R},\Z^2,\sigma)$ to $(\boldsymbol{T},\Z^2,R)$
            such that $R^\bk\circ f = f\circ\sigma^\bk$
    for every $\bk\in\Z^2$.
    The map $f$ is one-to-one on
    $f^{-1}(\boldsymbol{T}\setminus\Delta_{\Pcal,R})$.
\end{proposition}

The proposition implies that points of $\boldsymbol{T}$ are encoded uniquely in $\mathcal{X}_{\mathcal{P},R}$ when the orbit of the point does not intersect the boundary of the partition.
In other words, the points in $\bT$ whose fibers  under the factor map $f:\mathcal{X}_{\mathcal{P},R} \rightarrow \boldsymbol{T}$ are not singletons is
\begin{equation} \label{eq:non-singleton-fibers}
    \{ y \in \bT : |f^{-1}(y)| > 1\} = \Delta_{\mathcal{P},R}.
\end{equation}

\subsection{Nonexpansive Directions in $\mathcal{X}_{\mathcal{P},R}$} 
\label{section:nonexpansive}

The next lemma says where to search for nonexpansive directions
in the symbolic dynamical system $\Xcal_{\Pcal,R}$.

\begin{lemma}
    Let $\Pcal$ give a symbolic representation of the dynamical system
    $(\boldsymbol{T},\Z^2,R)$ such that
    $R$ is a $\ZZ^2$-rotation on $\boldsymbol{T}$
    and assume that the subshift $\Xcal_{\Pcal,R}$ is minimal.
    Let $H$ be a nonexpansive half-space
    for $\mathcal{X}_{\mathcal{P},R}$.
    Then
    there exist $x,y\in\mathcal{X}_{\mathcal{P},R}$
    such that 
    $x|_{H\cap \Z^2}=y|_{H\cap \Z^2}$,
    $x\neq y$,
    and 
    $f(x)=f(y)\in\Delta_{\mathcal{P},R}$.\label{lem:non_exp_dense}
\end{lemma}

\begin{proof}
    Since $H$ is nonexpansive, 
    there exist $x,y\in\mathcal{X}_{\mathcal{P},R}$
    such that 
    $x|_{H\cap \Z^2}=y|_{H\cap \Z^2}$
    and $x\neq y$.
    Let $\bv_H$ be the outward unit normal vector of the half-space $H$.
    Let $(g_i)_{i\in\N}$ be a sequence of vectors $g_i\in H\cap\ZZ^2$ such that
    $\lim_{i\to\infty}\langle\bv_H, g_i\rangle=-\infty$.
    We have $\lim_{i\to\infty} \dist(\sigma^{g_i} x, \sigma^{g_i}y)=0$,
    thus the elements $x,y$ are regionally proximal.

    From Theorem~\ref{thm:auslander}, $x$ and $y$ must have the same image
    under the equicontinuous factor, i.e., we have $f(x)=f(y)$.
    If $f(x)\notin\Delta_{\mathcal{P},R}$, then from 
    Proposition~\ref{prop:factor-map}, $|f^{-1}(f(x))|=1$. Thus
    $x=y$ which is impossible.
    Therefore $f(x)\in\Delta_{\mathcal{P},R}$.
\end{proof}

Per Lemma~\ref{lem:non_exp_dense}, all nonexpansive directions correspond to
pairs of configurations whose images under the factor map are equal and belong
to $\Delta_{\mathcal{P},R}$. 

The strategy which will be used in the next sections is stated as follows.

\begin{lemma}\label{lem:strategy-to-get-Conway-worms}
    Let $x,y\in\Xcal_{\Pcal,R}$ such that $x\neq y$ and $f(x)=f(y)=\bp\in\Delta_{\mathcal{P},R}$.
    Then
    \begin{equation}\label{eq:diff-set-subset-orbit-partition-boundary}
        D(x,y)=\{\bn\in \ZZ^2\mid x_\bn \neq y_\bn\}
        \subseteq \{ \bn \in \ZZ^2 \mid R^\bn(\bp) \in \Delta_\Pcal\}.
    \end{equation}
    If there exist $t>0$ and a vector space $F\subset\RR^2$ 
    such that $\varnothing\neq D(x,y)\subset F^t$, then 
    $F$ is nonexpansive and
    $D(x,y)$ is a Conway worm associated to $F$.
\end{lemma} 

\begin{proof}
    Let $\bn \in \ZZ^2$ be such that $R^\bn(\bp) \notin \Delta_\Pcal$.
    Thus there exists $a\in\Acal$ such that 
    $f(\sigma^\bn(x))=f(\sigma^\bn(y))=R^\bn(\bp)\in P_a$.
    Thus $x_\bn=a=y_\bn$. This shows the set inclusion.
    The rest follows from the definition of Conway worms.
\end{proof}

In the sections that follow, we have pairs $x,y$ of configurations such that
Equation~\eqref{eq:diff-set-subset-orbit-partition-boundary} is in fact an
equality.
This allows to search for Conway worms and nonexpansive directions
in the context of the Jeandel-Rao Wang shift
by performing an exhaustive examination of the orbits
of points that intersect $\Delta_\Pcal$, the boundary of the partition.

\section{The Jeandel-Rao Wang shift}\label{sec:jeandel-rao-wang-shift}

In this section, we recall known results about the Jeandel-Rao Wang shift on
which our results are based. In \cite{Lab1}, the author presents a remarkable method of generating a proper, minimal, aperiodic subshift of the full Wang shift $\Omega_0$. We describe it loosely here, and refer the reader to \cite{Lab1} for the full details. The central aspect of this construction is a partition $\mathcal{P}_0$ of the torus $\boldsymbol{T} = \R^2/\Gamma_0$ where $\Gamma_0$ is the lattice generated by the vectors $\gamma_0 = (\varphi,0)$ and $\gamma_1 = (1,\varphi + 3)$. In Figure~\ref{fig:markov_partition} we see the partition $\mathcal{P}_0$ of the rectangular fundamental domain of the torus $\boldsymbol{T}$.

\begin{figure}[h]
\begin{center}
    \includegraphics{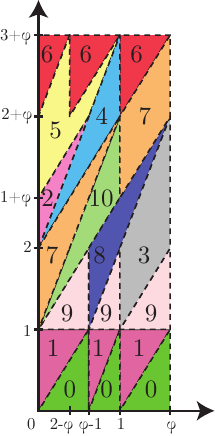}
\end{center}
    \caption{The partition $\mathcal{P}_0$ of $\R^2/\Gamma_0$.}
    \label{fig:markov_partition}
\end{figure}

Configurations (tilings) correspond to points in $\boldsymbol{T}$ by defining a dynamical system on $\boldsymbol{T}$. Specifically, we define the continuous $\Z^2$-action $R_0$ on $\boldsymbol{T}$ by $R_0^\bn(\bx):=R_0(\bn,\bx)=\bx + \bn$ for every $\bn=(n_1,n_2)\in\Z^2$, and $R_0$ gives rise to a dynamical system $(\R^2/\Gamma_0, \Z^2, R_0)$. From the labeling of the partition $\mathcal{P}_0$, we obtain the symbolic dynamical system $\Xcal_{\Pcal_0,R_0}$ corresponding  to $(\R^2/\Gamma_0, \Z^2, R_0)$ 

The following theorem about $\Xcal_{\Pcal_0,R_0}$ is proved in \cite{Lab1}.

\begin{theorem}\label{thm:JR-max-equi-factor}
    The Jeandel-Rao Wang shift $\Omega_0$ has the following properties:
\begin{enumerate}[\rm (i)]
\item $\Xcal_{\Pcal_0,R_0}\subsetneq\Omega_0$
is a proper minimal and aperiodic subshift of $\Omega_0$,
\item the partition $\Pcal_0$ gives a symbolic representation of 
$(\R^2/\Gamma_0,\Z^2,R_0)$,
\item the dynamical system $(\R^2/\Gamma_0,\Z^2,R_0)$ is the maximal
    equicontinuous factor of $(\Xcal_{\Pcal_0,R_0},\Z^2,\sigma)$,
\item the set of fiber cardinalities of the factor map
$\Xcal_{\Pcal_0,R_0}\to\R^2/\Gamma_0$ is $\{1,2,8\}$, and
\item
    the dynamical system $(\Xcal_{\Pcal_0,R_0},\Z^2,\sigma)$ is strictly
    ergodic and
    the measure-preserving dynamical system $(\Xcal_{\Pcal_0,R_0},\Z^2,\sigma,\nu)$
    is isomorphic 
    to $(\R^2/\Gamma_0,\Z^2,R_0,\lambda)$ 
    where $\nu$ is the unique shift-invariant probability measure on
    $\Xcal_{\Pcal_0,R_0}$
    and $\lambda$ is the Haar measure on $\R^2/\Gamma_0$.
\end{enumerate}
\end{theorem}

\begin{figure}[h]
\begin{center}
    \includegraphics[width=0.99\linewidth,trim={2cm 0cm 0cm 0cm},clip]{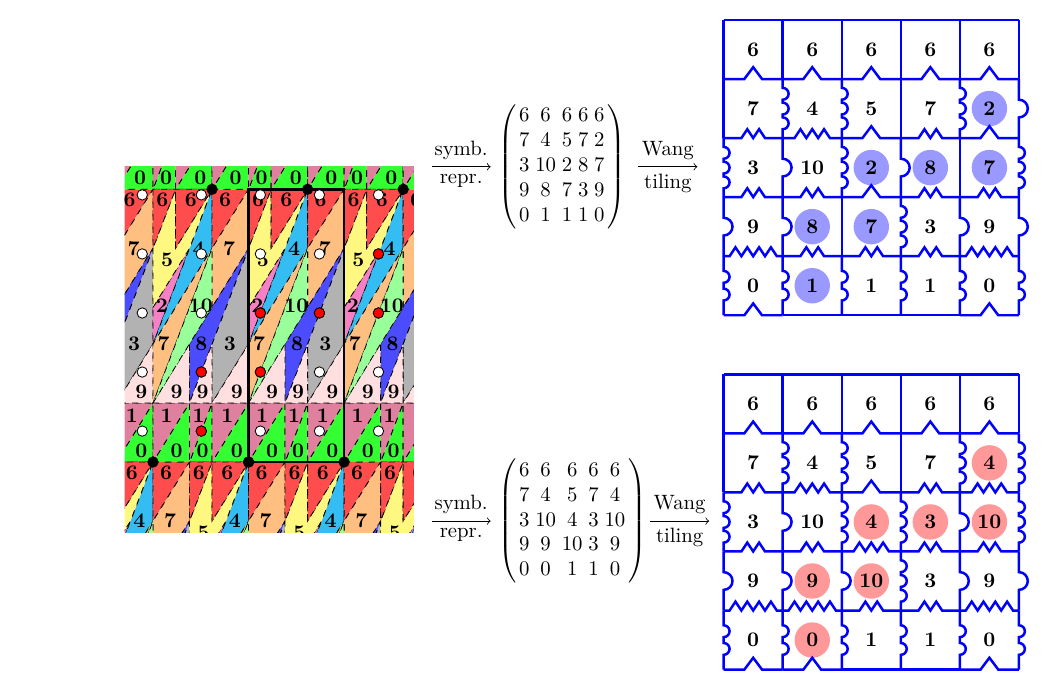}
\end{center}
    \caption{
        On the left, we illustrate the lattice $\Gamma_0=\langle(\varphi,0), (1,\varphi+3)\rangle_\Z$, where $\varphi=\frac{1+\sqrt{5}}{2}$, with black vertices, a rectangular fundamental domain of the flat torus $\boldsymbol{T} = \R^2/\Gamma_0$ with a black contour and  a polygonal partition $\Pcal_0$ of $\R^2/\Gamma_0$ with indices in the set $\{0,1,\dots,10\}$. For every starting point $\bp\in\R^2$, the coding of the orbit $\mathcal{O}_{R_0}(\boldsymbol{p})$, which is just the shifted lattice $\bp+\Z^2$ (the white dots), under the polygonal partition  yields a configuration $w:\Z^2\to\{0,1,\dots,10\}$ which is  a symbolic representation of $\bp$. The configuration $w$ corresponds to a valid tiling of the plane with Jeandel-Rao's set of 11 Wang tiles. As shown on the right, when the orbit of $\bp$ hits the boundary of the partition $\Pcal_0$, $\bp$ has more than one symbolic representations.} \label{fig:2d-walk}
\end{figure}

As stated in Theorem~\ref{thm:main-theorem}, only one of the slopes of
nonexpansive directions within the minimal subshift of the Jeandel-Rao Wang
shift is rational, namely the slope 0. When the slope of a nonexpansive
direction is rational, it may be associated to a pattern
which is repeated periodically
along a nonexpansive direction $(n_1,n_2)\in\Z^2$. 
This allows one to modify a valid configuration into a different valid configuration
by shifting half of the configuration on one side of
the nonexpansive direction by a vector $(n_1,n_2)$. If this is possible, the nonexpansive direction is
called a \emph{fault line}. The undesirable effect of a fault line is that of allowing valid patterns
of appearance probability zero which may appear only along the fault line and
never elsewhere.
In the case of the Jeandel-Rao Wang shift, it turns out that the line of slope
zero is a fault line. As a consequence, the Jeandel-Rao Wang shift $\Omega_0$
is not minimal \cite{MR4226493}. But it is conjectured to be uniquely ergodic
or, in other words, each pattern has a uniquely determined appearance frequency 
(shift-invariant measure) in every valid configuration in the Jeandel-Rao Wang
shift $\Omega_0$.

\begin{conjecture} 
\cite{MR4226493}
The Jeandel-Rao subshift $\Omega_0$ is uniquely ergodic. 
\label{conj:ergodic}
\end{conjecture}

The process of coding a point $\boldsymbol{p} \in \boldsymbol{T}$ is illustrated in Figure~\ref{fig:2d-walk}\footnote{We encourage the reader to view an electronic copy of this figure and others since they are color coded.}. Observe that at points of the orbit  $\mathcal{O}_{R_0}(\boldsymbol{p})$ that fall on the boundary of an atom of the partition $\mathcal{P}_0$, the coding is ambiguous, but this problem is easily addressed by specifying a direction $\vec{v}$ that is nonparallel to the boundary edges of the atoms in $\mathcal{P}_0$ and using $\vec{v}$ to determine the labeling from the atoms adjacent to the boundary point. For the tiling on the upper right of Figure~\ref{fig:2d-walk}, we use direction $\vec{v} = (1,-1)$ to resolve the ambiguity where the orbit $\mathcal{O}_{R_0}(\boldsymbol{p})$ intersects the boundary of $\mathcal{P}_0$, and for the tiling on the lower right side of Figure~\ref{fig:2d-walk}, we use the opposite direction $-\vec{v}$ to resolve the ambiguity.

In understanding the nonexpansive directions in $\mathcal{X}_{\mathcal{P}_0, R_0}$, the main point is this article, we must consider the orbits of points that intersect one (or more) of the boundaries 
$\Delta_{\Pcal_0}$
of the polygons in $\mathcal{P}_0$. 
Also we define a \emph{$\Delta_{\Pcal_0}$-line} to be any straight line segment forming part of the boundary of an atom in the partition $\mathcal{P}_0$.	

In next section, we  perform an exhaustive examination of the orbits of points that intersect $\Delta_{\Pcal_0}$, the boundary of the partition, to calculate the nonexpansive directions for $\mathcal{X}_{\mathcal{P}_0,R_0}$.

\section{Calculating the nonexpansive directions for $\mathcal{X}_{\mathcal{P}_0,R_0}$} 
\label{sec:calculating}

The goal of this section is prove Theorem~\ref{thm:main}
which gives the nonexpansive directions within
the minimal subshift $\mathcal{X}_{\mathcal{P}_0,R_0}$ of the Jeandel-Rao Wang shift.
If $x,y\in\Xcal_{\Pcal,R}$ are such that $x\neq y$ and $f(x)=f(y)=\bp\in\Delta_{\mathcal{P},R}$,
then from Lemma~\ref{lem:strategy-to-get-Conway-worms},
we have that the support of their differences satisfy $D(x,y)\subseteq W(\bp)$
where
\[
    W(\bp) = \{ \bn \in \ZZ^2 \mid R_{0}^{\bn}(\bp) \in \Delta_{\Pcal_0}\}
\]
for every $\bp \in \boldsymbol{T}$.
If there exist $t>0$ and a vector space $F\subset\RR^2$ 
such that $W(\bp)\subset F^t$, then 
$F$ is nonexpansive for $\mathcal{X}_{\mathcal{P}_0,R_0}$.

Thus we need to describe the part of the orbit $\mathcal{O}_{R_0}(\bp)$ 
of a point $\bp \in \boldsymbol{T}$ that stays in the boundary
$\Delta_{\Pcal_0}$ of the partition $\Pcal_0$.
When $\bp\in\Delta_{\Pcal_0,R_0}$, the set $W(\bp)$ can have 
different kinds of behaviors.
For example, the set $W(\bp)$ is illustrated in 
Figure~\ref{fig:delta_pts_phi} when $\bp = 1/4(\varphi -1 ,1)$
and in Figure~\ref{fig:delta_pts_origin} when $\bp = (0,0)$.
The results in this sections aim to describe these behaviors.

\begin{figure}[h]
\centering
\begin{subfigure}[h]{.45\textwidth} 
\centering
\includegraphics[scale = .283]{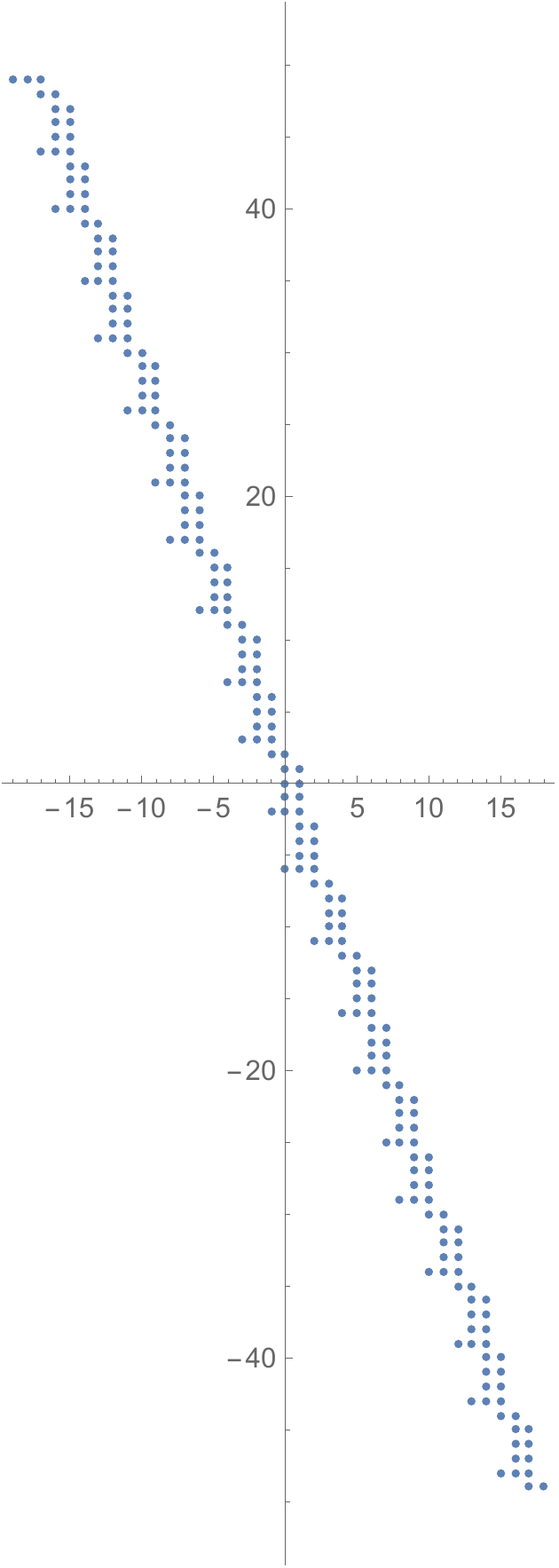}  
    \caption{$W(\bp)$ when $\boldsymbol{p} = (1/4)(\varphi - 1,1)$.}
\label{fig:delta_pts_phi}
\end{subfigure}
\begin{subfigure}[h]{.45\textwidth} 
\centering
\includegraphics[scale = .25]{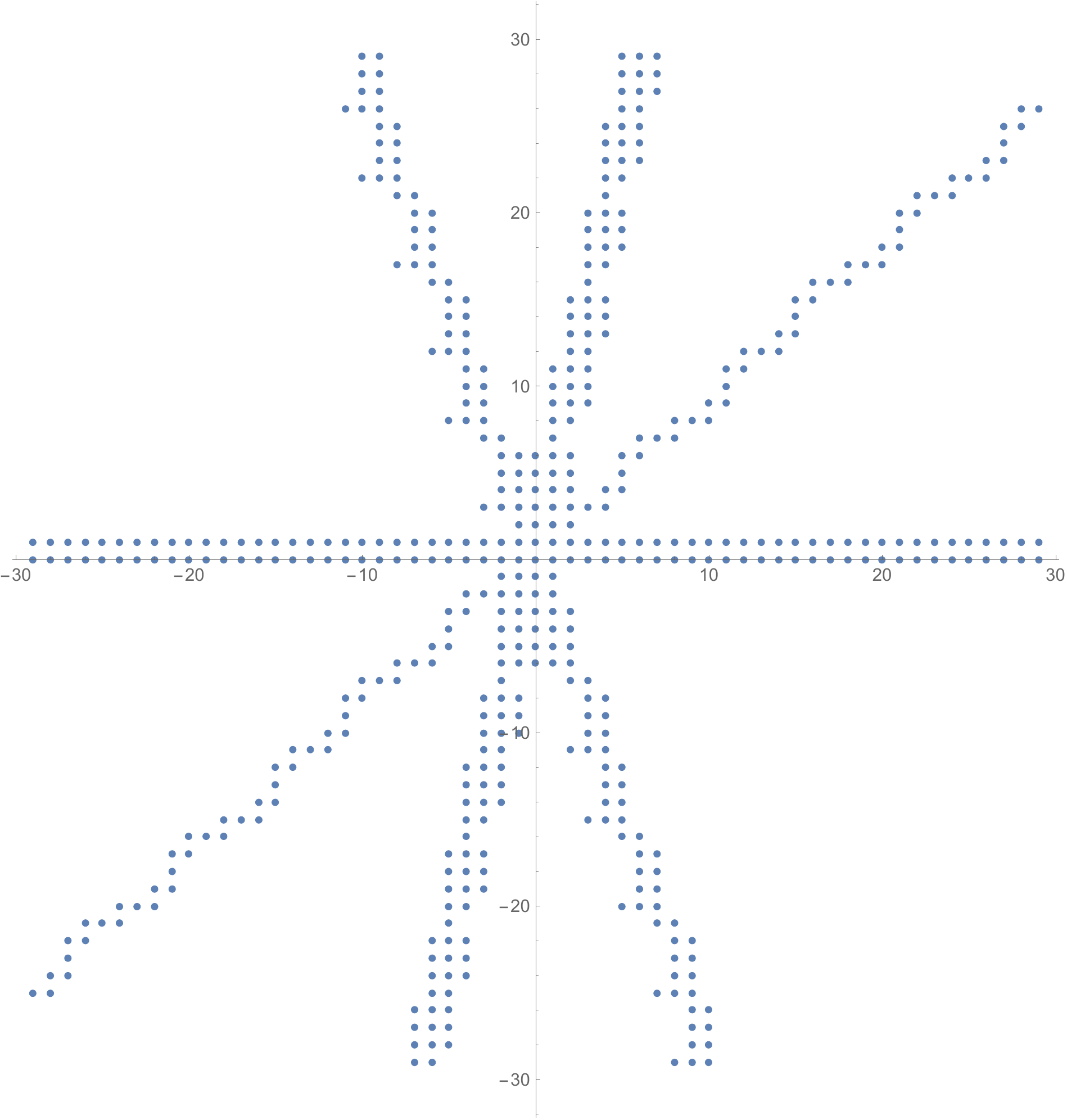}  
    \caption{$W(\bp)$ when $\boldsymbol{p} = \boldsymbol{0}.$}
    \label{fig:delta_pts_origin}
\end{subfigure}\caption{
    Illustration of the set $W(\bp)$ computed experimentally.}
    \label{fig:delta_pts} 
\end{figure}

\subsection{Returning to a fixed segment}

The next proposition shows that the set of return vectors $\bn\in\ZZ^2$ 
to a fixed segment is bounded away from an Euclidean line.
It describes the restriction of an
orbit under $R_0$ returning to an arbitrary line segment $\vecomega$ 
of slope in $\QQ(\varphi)\cup\{\infty\}$
in the 2-dimensional torus $\boldsymbol{T}=\RR^2/\Gamma_0$.
It even provides the normal vector of the associated 1-dimensional vector space
in terms of $\vecomega$.

\begin{proposition}\label{prop:return-to-a-generic-segment}
    Let $M=\left(\begin{smallmatrix}\varphi & 1\\0 & \varphi+3\end{smallmatrix}\right)$
        so that $\Gamma_0=M\cdot\ZZ^2$.
    Let $\vecomega = \left(\begin{smallmatrix}a\varphi+b\\c\end{smallmatrix}\right)$ 
        where $a,b,c\in \ZZ$ such that $\vecomega\neq0$ and
    \[
        K_{\vecomega}=\{ \bn \in \ZZ^2 \mid \bn+M\bg=\alpha\vecomega 
              \text{ for some }\bg\in \ZZ^2
              \text{ and }     -1<\alpha<1\}.
    \]
    There exists $t>0$ such that $K_{\vecomega}\subset F^t$, that is,
    $K_{\vecomega}$ is bounded away from $F$,
    where $F$ is the 1-dimensional vector space orthogonal to the normal vector
    $\bu=\left(\begin{smallmatrix}0\\1\end{smallmatrix}\right)$ if $c=0$,
    $\bu=\left(\begin{smallmatrix}c\\-b\end{smallmatrix}\right)$ if $a+3b-c=0$,
    and $\bu=\left((M A^{-1}B)^\top + {\rm Id}\right)\vecomega$
    if $c(a+3b-c)\neq0$ where
        $A = \left(\begin{smallmatrix}
            c & -4a-b\\
            0 &  a+3b-c 
        \end{smallmatrix}\right)$
        and
        $B = \left(\begin{smallmatrix}
            0 & a\\
            c & -b
        \end{smallmatrix}\right)$, that is,
    \begin{equation}\label{eq:bu-generic-normal-vector}
        \bu=
        \begin{pmatrix}
        (5a^2c+8abc+b^2c+c^3)\varphi + (4a^2c+2abc+3b^2c+3c^3)\\
        (a^3-2ab^2-b^3-a^2c-2abc-bc^2)\varphi + (a^3-a^2b-ab^2-a^2c-b^2c+ac^2-c^3)
        \end{pmatrix}.
    \end{equation}
\end{proposition}

\begin{proof} 
    Let $\bn=\left(\begin{smallmatrix}n_1\\n_2\end{smallmatrix}\right) \in \ZZ^2$ be such that $\bn+M\bg=\alpha\vecomega$ for some
    $\bg=\left(\begin{smallmatrix}g_1\\g_2\end{smallmatrix}\right) \in\ZZ^2$ and $0<\alpha<1$.
    Also let $\vecomega^\perp = \left(\begin{smallmatrix}c\\-a\varphi-b\end{smallmatrix}\right)$.
    We have
    \begin{align}\label{eq:alphaomega-omega}
        \begin{split}
        0 &= \langle \alpha\vecomega,\vecomega^\perp\rangle
           = \langle \bn+M\bg,       \vecomega^\perp\rangle\\
          &= (cg_1-4ag_2-bg_2-an_2)\varphi + (-ag_2-bn_2-3bg_2+cn_1+cg_2).
        \end{split}
    \end{align}
    As $\varphi \notin \QQ$ and all coefficients are integers, both parenthesis must be zero.
    Thus we deduce two equations from the one above:
    \begin{align*}
        cg_1 - (4a+b)g_2 &=        an_2\\
            (a+3b-c) g_2 &= c n_1 -bn_2,
    \end{align*}
    a system which can be rewritten as
    \[
        A\bg = B\bn
    \]
    where
    \[
        A = \begin{pmatrix}
            c & -4a-b\\
            0 &  a+3b-c 
        \end{pmatrix}
        \qquad
        \text{ and }
        \qquad
        B = \begin{pmatrix}
            0 & a\\
            c & -b
        \end{pmatrix}.
    \]
    Note that $\det(A)=c(a + 3 b - c)$.

    Assume that $\det(A)=0$ and $c=0$.
    We have
    \[
        \left\{ \begin{array}{rl}
                - (4a+b)g_2 &= an_2\\
                (a+3b) g_2 &= -bn_2
        \end{array} \right.
        \implies
        \left\{ \begin{array}{rl}
                - (4ab+b^2)g_2 &= abn_2\\
                (a^2+3ab) g_2 &= -abn_2
        \end{array} \right.
        \implies
        (a^2-ab-b^2) g_2 = 0.
    \]
    The equation $a^2-ab-b^2=0$ has only $a=b=0$ as integer solution which implies that
    $\vecomega=0$ which is a contradiction.
    Therefore $g_2=0$, which implies that $n_2=0$, and hence 
    $K_{\vecomega}\subseteq \ZZ \times \{0\}$
    which are vectors orthogonal to the normal vector 
    $\bu=\left(\begin{smallmatrix}0\\1\end{smallmatrix}\right)$.

    Assume that $\det(A)=0$ and $a+3b-c=0$. Then, $cn_1=bn_2$.
    Thus, $K_{\vecomega}$
    consists of vectors orthogonal to the normal vector 
    $\bu=\left(\begin{smallmatrix}c\\-b\end{smallmatrix}\right)$.

    Now assume that $\det(A)\neq 0$.
    We have
    \begin{align*}
        \langle \alpha\vecomega,\vecomega\rangle
           &= \langle \bn+M\bg, \vecomega\rangle
           = \langle \bn, \vecomega\rangle+\langle M\bg, \vecomega\rangle\\
           &= \langle \bn, \vecomega\rangle+\langle M A^{-1}B\bn, \vecomega\rangle\\
           &= \langle \bn, \vecomega\rangle+\langle\bn, (M A^{-1}B)^\top \vecomega\rangle
           = \langle\bn, (M A^{-1}B)^\top \vecomega + \vecomega\rangle
    \end{align*}
    Let
    \begin{equation}\label{eq:normal-vector-u}
        \bu 
        = \left((M A^{-1}B)^\top + {\rm Id}\right)\vecomega
    \end{equation}
    and $F$ be the vector space $F$ orthogonal to the vector $\bu$.
    The distance from a point $\bn\in\ZZ^2$ to $F$ can be defined
    as $\dist(\bn,F)=|\langle \bn,\bu\rangle|$.
    We have that $\bn$ is bounded away from the vector space $F$ 
    orthogonal to the vector $\bu$, i.e.,
    \[
        \dist(\bn,F)
        = |\langle \bn,\bu\rangle|
        = |\langle \alpha\vecomega,\vecomega\rangle|
        = |\alpha|\cdot \norm{\vecomega}^2 
        \leq \norm{\vecomega}^2 
        = (a^{2}+ 2ab)\varphi + a^{2} + b^{2} + c^{2}.
    \]
    Thus if $t=\norm{\vecomega}^2=(a^{2}+ 2ab)\varphi + a^{2} + b^{2} + c^{2}$,
    we have $\bn\in F^t$.
    Note that
    \begin{align*}
        A^{-1}&=
        \frac{1}{\det(A)}
    \left(\begin{array}{rr}
        a + 3 b - c & 4 a + b \\
        0 & c
    \end{array}\right),\\
        A^{-1}B &=
        \frac{1}{\det(A)}
        \left(\begin{array}{rr}
            4 a c + b c & a^{2} - a b - b^{2} - a c \\
            c^{2} & -b c
        \end{array}\right),\\
        M A^{-1}B &=
        \frac{1}{\det(A)}
    \left(\begin{array}{rr}
        (4 a c + b c)\varphi  + c^{2} & 
        (a^{2} - a b - b^{2} - a c)\varphi - b c \\
        c^2\varphi + 3c^{2} & - b c\varphi - 3 b c
    \end{array}\right).
    \end{align*}
    Thus, we compute that the normal vector is
    \begin{align*}
        \bu
        &=\left((M A^{-1}B)^\top + {\rm Id}\right)\vecomega
        =\left((M A^{-1}B)^\top \vecomega + \vecomega\right)\\
        &=
        \frac{1}{\det(A)}
    \left(\begin{array}{r}
        (5a^2c+8abc+b^2c+c^3)\varphi + (4a^2c+2abc+3b^2c+3c^3)\\
        (a^3-2ab^2-b^3-a^2c-2abc-bc^2)\varphi + (a^3-a^2b-ab^2-a^2c-b^2c+ac^2-c^3)
    \end{array}\right).\qedhere
    \end{align*}
\end{proof}

Equation~\eqref{eq:normal-vector-u} provides
a generic formula for the vector $\bu$, but 
the matrices $A$ and $B$ can not be deduced
from $\vecomega$ and the matrix $M$ only from linear algebra.
Indeed, the matrices $A$ and $B$ are obtained from the single
Equation~\eqref{eq:alphaomega-omega} which involves properties of the ring
$\ZZ[\varphi]$ allowing to deduce two equations from one.

\subsection{Returning to segments of a fixed slope}

\begin{figure}[h]
\centering
\includegraphics[scale=0.95]{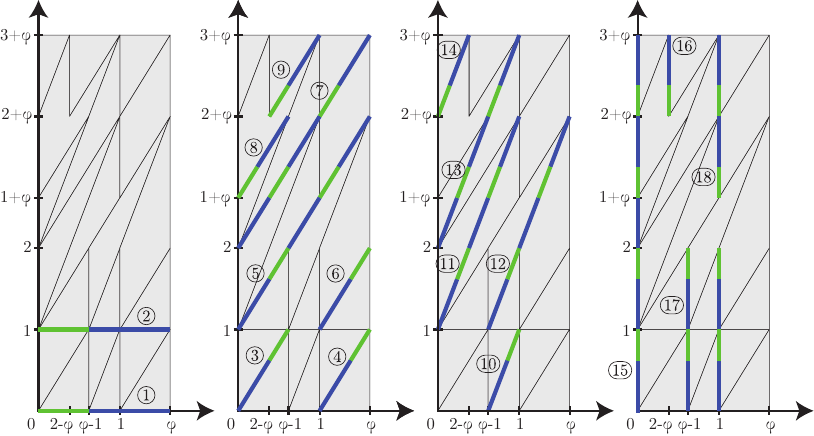}
    \caption{$\Delta_{\Pcal_0}$-lines in $\mathcal{P}_0$}
    \label{fig:markov_partition_edges}
\end{figure}

\begin{table}[h]
\centering
\footnotesize
\begin{tabular}{|c|c|c|c|}\hline
    Slope & $\Delta_{\Pcal_0}$-line & Domain Restriction & $(a,b)$  \\ \hline\hline

\multirow{2}{*}{0} 
  & \Circled{1} & $0 \leq x \leq \varphi$ & $(0,0)$ \\
  & \Circled{2} & $0 \leq x \leq \varphi $ & $(0,-1)$\\\hline 
    
\multirow{14}{*}{$\varphi$}   
  & \Circled{3} & $0 \leq x \leq \varphi - 1$ & $(0,0)$\\
  & \Circled{4} & $1 \leq x \leq \varphi$ & $(-1,0)$\\
  & \Circled{5} & $0 \leq x \leq \varphi - 1$ & $(0,-1)$\\
  & \Circled{5} & $\varphi -1 \leq x \leq 2\varphi - 2$ & $(1,-2)$\\
  & \Circled{5} & $2\varphi - 2 \leq x \leq \varphi$ & $(2,-3)$\\
  & \Circled{6} & $1 \leq x \leq \varphi$ & $(-1,-1)$\\
  & \Circled{7} & $0 \leq x \leq \varphi - 1$ & $(0,-2)$\\
  & \Circled{7} & $\varphi - 1 \leq x \leq 1$ & $(1,-3)$\\
  & \Circled{7} & $1 \leq x \leq 2\varphi - 2$ & $(1,-3)$\\
   & \Circled{7} & $2\varphi - 2 \leq x \leq \varphi$ & $(2,-4)$\\
   & \Circled{8} & $0 \leq x \leq 2\varphi - 3$ & $(2,-2)$\\
  & \Circled{8} & $2\varphi - 3 \leq x \leq \varphi - 1$ & $(3,-3)$\\
  & \Circled{9} & $2 - \varphi \leq x \leq \varphi - 1$ & $(0,-3)$\\
  & \Circled{9} & $2 - \varphi \leq x \leq \varphi - 1$ & $(1,-4)$\\ \hline
\end{tabular}
\quad
\begin{tabular}{|c|c|c|c|}\hline
    Slope & $\Delta_{\Pcal_0}$-line & Domain Restriction & $(a,b)$  \\ \hline\hline
\multirow{7}{*}{$\varphi^2$}   
  & \Circled{10} & $\varphi - 1 \leq x \leq 1$          & $(0,  0)$\\
  & \Circled{11} & $0 \leq x \leq 2-\varphi$            & $(-1,-1)$\\
  & \Circled{11} & $2-\varphi \leq x \leq 4-2\varphi$   & $(-3,-2)$\\
  & \Circled{11} & $4-2\varphi \leq x \leq 1$           & $(-5,-3)$\\
  & \Circled{12} & $\varphi - 1 \leq x \leq 1$          & $(0, -1)$\\
   & \Circled{12} & $1 \leq x \leq 3 - \varphi$         & $(-2,-2)$\\
    & \Circled{12} & $3 - \varphi \leq x \leq \varphi$  & $(-4,-3)$\\
  & \Circled{13} & $0 \leq x \leq 2-\varphi$            & $(-1,-2)$\\
  & \Circled{13} & $2-\varphi \leq x \leq 4-2\varphi$   & $(-3,-3)$\\
  & \Circled{13} & $4-2\varphi \leq x \leq 1$           & $(-5,-4)$\\
   & \Circled{14} & $0 \leq x \leq 5-3\varphi$          & $(-4,-3)$\\
  & \Circled{14} & $5-3\varphi \leq x \leq 2 - \varphi$ & $(-6,-4)$\\\hline
\multirow{14}{*}{$\infty$}   
   & \Circled{15} & $x = 0, 0 \leq y \leq 1$                 & $(-1,0)$\\
   & \Circled{15} & $x = 0, 1 \leq y \leq 2$                 & $(-1,-1)$\\
   & \Circled{15} & $x = 0, 2 \leq y \leq 3$                 & $(-1,-2)$\\
   & \Circled{15} & $x = 0, 3 \leq y \leq 4$                 & $(-1,-3)$\\
   & \Circled{15} & $x = 0, 4 \leq y \leq 3+\varphi$         & $(-1,-4)$\\
   & \Circled{16} & $x = 2-\varphi, 2+\varphi \leq y \leq 4$ & $(-3,-3)$\\
   & \Circled{16} & $x = 2-\varphi, 4 \leq y \leq 3+\varphi$ & $(-3,-4)$\\
   & \Circled{17} & $x = \varphi-1, 0 \leq y \leq 1$         & $(0,0)$\\
   & \Circled{17} & $x = \varphi-1, 1 \leq y \leq 2$         & $(0,-1)$\\ 
   & \Circled{18} & $x = 1, 0 \leq y \leq 1$                 & $(-2,0)$\\
   & \Circled{18} & $x = 1, 1 \leq y \leq 2$                 & $(-2,-1)$\\
   & \Circled{18} & $x = 1, 1+\varphi \leq y \leq  3$        & $(-2,-2)$\\
   & \Circled{18} & $x = 1, 3 \leq y \leq 4$                 & $(-2,-3)$\\
   & \Circled{18} & $x = 1, 4 \leq y \leq 3+\varphi$         & $(-2,-4)$\\ \hline
\end{tabular}
    \caption{Values of $\boldsymbol{n} = (a,b) \in \ZZ^2$ such that $R^{\boldsymbol{n}}$ moves $\Delta_{\Pcal_0}$-lines to some base segment $Z_i$, see Figure~\ref{fig:markov_partition_edges}.}
\label{tab:move_to_0}
\end{table}

It is convenient to split $\Delta_{\Pcal_0}$
as $\Delta_{\Pcal_0}=\Delta_0\cup\Delta_\infty\cup\Delta_\varphi\cup\Delta_{\varphi^2}$
where,
for $i \in \{0,\infty,\varphi,\varphi^2\}$, $\Delta_i$ denotes the union of the slope-$i$ line segments in $\Delta_{\Pcal_0}$ (the \emph{\textbf{slope-$i$ part}} of $\Delta_{\Pcal_0}$). 
Also, we designate certain specific segments $Z_i\subset\Delta_i$ to be the \textbf{\emph{slope-$i$ base segments}} as follows: \begin{itemize}
\item $Z_0$ is the segment from $(0,0)$ to $(\varphi,0)$,
\item $Z_{\infty}$ is the segment from $(\varphi-1,0)$ to $(\varphi-1,1)$,
\item $Z_{\varphi}$ is the segment from $(0,0)$ to $(\varphi-1,1)$,
\item $Z_{\varphi^2}$ is the segment from $(\varphi-1,0)$ to $(1,1)$.
\end{itemize}
For every $i \in \{0,\infty,\varphi,\varphi^2\}$, $Z_i$ is the left most among
the bottom most segment of slope $i$ in the partition. 
These particular segments are also used in the proof of
Theorem~\ref{thm:resolution} in Section~\ref{sec:resolution}.
The next lemma shows that every point in $\Delta_i$ is in the orbit under $R_0$
of a point in the base segment $Z_i$.

\begin{lemma} For each $\boldsymbol{x} \in \Delta_{i}$ where $i \in \{0,\infty,\varphi, \varphi^2\}$, there exists $\boldsymbol{n} = (n_1,n_2) \in \ZZ^2$ with $\|\boldsymbol{n}\| = |n_1| + |n_2| \leq 10$ such that $R_{0}^{\boldsymbol{n}}(\boldsymbol{x}) \in Z_i$. \label{lem:move_to_origin}\end{lemma}

\begin{proof} The proof is established by inspection of the $\Delta_{\Pcal_0}$-lines as numbered in Figure~\ref{fig:markov_partition_edges}. The values of $\boldsymbol{n} \in \ZZ^2$ such that $R_{0}^{\boldsymbol{n}}$ moves segments in $\Delta_i$ to the base segment $Z_i$ are given in Table~\ref{tab:move_to_0}. All of them satisfy $\|\boldsymbol{n}\| = |n_1| + |n_2| \leq 10$.
\end{proof}

The description of $W(\bp)$ is best done when restricting it to vectors
returning to segments of a fixed slope.
Also, from the previous lemma, it can be approximated by restricting it to
vectors returning to the base segments $Z_i$. Thus, for every $\bp \in
\boldsymbol{T}$ and
$i \in \{0,\infty,\varphi,\varphi^2\}$, we define
\begin{align*}
    W_i(\bp) &= \{ \bn \in \ZZ^2 \mid R_{0}^{\bn}(\bp) \in \Delta_i\},\\
    V_i(\bp) &= \{ \bn \in \ZZ^2 \mid R_{0}^{\bn}(\bp) \in Z_i\}.
\end{align*}
which satisfies 
$W(\bp) = W_0(\bp) \cup W_\infty(\bp) \cup W_\varphi(\bp) \cup W_{\varphi^2}(\bp)$.

Also from Lemma~\ref{lem:move_to_origin}, we have
$W_i(\bp) \subset V_i(\bp)+[-10,10]^2 $.
Thus the global structure of $W_i(\bp)$ follows the one of $V_i(\bp)$.
We may now use Proposition~\ref{prop:return-to-a-generic-segment} to deduce
that $V_i(\bp)$ is at a bounded distance from a Euclidean line.

\begin{lemma}\label{lem:orbit-remaining-in-partition-boundary-onlyVi}
    Let
    $\boldsymbol{p} \in \Delta_{\Pcal_0,R_0}$ and
    $i \in \{0,\infty,\varphi,\varphi^2\}$.
    If $V_i(\bp)\neq\varnothing$, then
    $V_i(\bp)$ is at bounded distance from a line of slope $m_i$
    where
    \[
       m_0=0,
       \qquad m_{\infty}=\varphi+3,
       \qquad m_{\varphi}=-3\varphi+2,
       \qquad m_{\varphi^2}=-\varphi+\frac{5}{2}.
    \]
\end{lemma}

\begin{proof} 
    Let $i \in \{0,\infty,\varphi,\varphi^2\}$. 
    Since $V_i(\bp)\neq\varnothing$, then there exists $\bm\in\ZZ^2$ such that
    $\bq=R_{0}^{\bm}(\bp) \in Z_i+\Gamma_0$.
    Note that $V_i(\bp)=V_i(\bq)+\bm$.

    Now let $\boldsymbol{n}\in V_i(\bq)$ so that
    $R_{0}^{\bn}(\bq)\in Z_i+\Gamma_0$.
    Let $\vecomega = (\omega_1,\omega_2)$ be the vector with the direction
    and length of $Z_i$. Because $\bq,\bq+\bn \in Z_i+\Gamma_0$, there exists some $\alpha,\beta
    \in (0,1)$ such that 
    $\bq+\bn = \alpha\vecomega+\Gamma_0$ and $\bq = \beta\vecomega+\Gamma_0$.  
    Let $D=[0,\varphi)\times[0,\varphi+3)$ be a fundamental domain of 
    $\RR^2/(M\ZZ^2)$
    where $M=\left(\begin{smallmatrix}\varphi & 1\\0 & \varphi+3\end{smallmatrix}\right)$.
    Since $D$ is a fundamental domain for the additive group $\Gamma_0=M\ZZ^2$ acting on $\RR^2$,
    there exists $\boldsymbol{g} \in \ZZ^2 $ such that 
    \[
        (\alpha-\beta)\vecomega
        =(\bq+\bn)-\bq = \bn+M\bg
    \]
    with $-1<\alpha-\beta<1$
    which implies $\bn\in K_\vecomega$. Thus $V_i(\bq)\subset K_{\vecomega}$.

Recall that $\vecomega$ is one of the four vectors $(\varphi,0)$, $(0,1)$, $(\varphi-1,1)$, and $(2-\varphi,1)$.
From Proposition~\ref{prop:return-to-a-generic-segment} and
using Equation~\eqref{eq:bu-generic-normal-vector}, we compute the normal vector $\bu$:
    \begin{itemize}
\item when $\vecomega=(1,0)$, we have $a=0$, $b=1$ and $c =0$, which gives 
    $\bu_{0} = (1,0)$,
\item when $\vecomega=(0,1)$, we have $a = b =0$ and $c =1$, which gives 
    $\bu_{\infty} = (\varphi+3,-1)$,
\item when $\vecomega=(\varphi-1,1)$, we have $a = 1$, $b = -1$ and $c =1$, which gives 
    $\bu_{\varphi} = (-\varphi+8,2\varphi-1)$, and
\item when $\vecomega=(2-\varphi,1)$, we have $a = -1$, $b = 2$, and $c =1$ which gives 
    $\bu_{\varphi^2} = (-6\varphi+15,-6)$. 
    \end{itemize} 
    From Proposition~\ref{prop:return-to-a-generic-segment}, 
    there exists $t>0$ such that $K_{\vecomega}\subset F^t$, 
    where $F$ is the 1-dimensional vector space orthogonal to the normal vector $\bu$.
Therefore, 
$V_{i}(\bq) \subset K_{\vecomega}\subset F^t$ where $F$ 
    is a Euclidean 2-dimensional line with slope 
    $m_{i} =0$, $\varphi+3,
    -3\varphi+2,$ or 
    $\varphi+\frac{5}{2}$ 
    if $i = 0, \infty, \varphi$, or $\varphi^2$, respectively.
    Finally since $V_i(\bp)=V_i(\bq)+\bm$,
    there exists $t'>0$ such that
    $V_{i}(\bp)\subset F^{t'}$.
\end{proof}

The same results holds from $W_i(\bp)$.

\begin{proposition}\label{prop:orbit-remaining-in-partition-boundary}
    Let
    $\boldsymbol{p} \in \Delta_{\Pcal_0,R_0}$ and
    $i \in \{0,\infty,\varphi,\varphi^2\}$.
    If $W_i(\bp)\neq\varnothing$, then
    $W_i(\boldsymbol{p})$ is at bounded distance from a line of slope $m_i$
    where
    \[
       m_0=0,
       \qquad m_{\infty}=\varphi+3,
       \qquad m_{\varphi}=-3\varphi+2,
       \qquad m_{\varphi^2}=-\varphi+\frac{5}{2}.
    \]
\end{proposition}

\begin{proof}
    Let $i \in \{0,\infty,\varphi,\varphi^2\}$ and 
    $\bp \in \Delta_i$.
From the definition of $V_i(\bp)$, we have $V_i(\bp) \subset W_i(\bp)$, but
from Lemma~\ref{lem:move_to_origin}, we also have
    \[
        W_i(\bp)
        \subset
        V_i(\bp)+[-10,10]^2.
    \]
    From Lemma~\ref{lem:orbit-remaining-in-partition-boundary-onlyVi},
    $V_i(\bp)$ is at bounded distance from a line of slope $m_i$.
    Thus the same holds for $W_i(\bp)$.
\end{proof}

\subsection{Proof of Theorem~\ref{thm:main}}

It remains to describe $W(\bp)$ from the restricted $W_i(\bp)$.
The next lemma shows that only the orbit of the origin intersect segments of
different slopes in $\Delta_{\Pcal_0}$.

\begin{lemma} Let $\boldsymbol{p} \in \Delta_{\Pcal_0} \subset \boldsymbol{T}$ and let $\boldsymbol{0} = (0,0) \in \boldsymbol{T}$. If $\mathcal{O}_{R_0}(\boldsymbol{p})$ contains points from $\Delta_i$ and $\Delta_j$ with $i \neq j$, then $\bp \in \mathcal{O}_{R_{0}}(\boldsymbol{0})$.\label{lem:origin_orbit}\end{lemma}

\begin{proof} Because $\{\mathcal{O}_{R_{0}}(\boldsymbol{x}) : \boldsymbol{x} \in \boldsymbol{T}  \}$ partitions $\boldsymbol{T}$, we need only show that $\mathcal{O}_{R}(\boldsymbol{p}) \cap \mathcal{O}_{R}(\boldsymbol{0}) \neq \varnothing$. To that end, let $\boldsymbol{x}\in \mathcal{O}_{R_{0}}(\boldsymbol{p}) \cap \Delta_i$ and $\boldsymbol{y}\in \mathcal{O}_{R_{0}}(\boldsymbol{p}) \cap \Delta_j$, and without loss of generality, suppose that $i \neq \infty$. By Lemma~\ref{lem:move_to_origin}, we may assume without loss of generality that $\boldsymbol{x} \in Z_i$ and $\boldsymbol{y} \in Z_j$.  Then $\boldsymbol{x} = \beta (1,i) + (0,1)$ for some $\beta \in \RR$ and $\boldsymbol{y} = \alpha(d_1,d_2) + (0,1)$ for some $\alpha \in \RR$, where $d_1 = 0$ and $d_2 = 1$ if $j = \infty$ and $d_1 = 1$ and $d_2 = j$ if $j \neq \infty$. We know that $\boldsymbol{y} \in \mathcal{O}_{R}(\boldsymbol{x})$, so there exists some and $\boldsymbol{n} = (n_1,n_2) \in \ZZ^2$
such that $\boldsymbol{y} = R_{0}^{\boldsymbol{n}}(\boldsymbol{x}) = \boldsymbol{x} + (n_1,n_2) \pmod{\Gamma_0}$. Thus there exists some and $\boldsymbol{g} = (g_1,g_2) \in \ZZ^2$ such that \[\alpha(d_1,d_2) = \beta(1,i) + (n_1,n_2) + g_1(\varphi,0) + g_2(1,\varphi+3), \] from which we obtain the equation \[
     \left<\beta(1,i) + (n_1,n_2) + g_1(\varphi,0) + g_2(1,\varphi+3),(d_2,-d_1)\right> = 0.\] Solving this equation for $\beta$ gives \begin{equation*}\label{eqn:alpha_vals}
         \beta = \frac{d_1(g_2(\varphi+3)+n_2)-d_2(g_1 \varphi + g_2 + n_1)}{d_2-d_1i}.
     \end{equation*} Next, to have $\boldsymbol{x} = \beta(1,i) + (0,1) \in \mathcal{O}_{R_{0}}(\boldsymbol{0})$, there must exist $(a,b),(z_1,z_2) \in \ZZ^2$ such that \begin{equation}\label{eq:origin_orbit}\beta(1,i) = a(\varphi,0) + b(1,\varphi+3) + (z_1,z_2).\end{equation} In Table~\ref{tab:origin_orbit}, we give the various possible values of $i$, $j$, $d_1$, and $d_2$ and the resulting integers $a,b,z_1,z_2$ making Equation~\eqref{eq:origin_orbit} valid, and thus we have $\mathcal{O}_{R_{0}}(\boldsymbol{p})= \mathcal{O}_{R_{0}}(\boldsymbol{x}) = \mathcal{O}_{R_{0}}(\boldsymbol{y}) = \mathcal{O}_{R_{0}}(\boldsymbol{0})$.\end{proof}

\begin{table}[h]
\centering
\begin{tabular}{ |c|c|c|c|c|c|c|c|c|} 
 \hline
 $i$ & $j$ & $d_1$ & $d_2$ & $a$ & $b$ & $z_1$ & $z_2$ \\\hline
0 & $\infty$  & 0 & 1 & $-g_1$ & 0 & $-g_2-n_1$ & 0 \\
$\varphi$ & $\infty$ & 0 & 1 &  $-g_1$ & $-g_1-g_2-n_1$ & $g_1$  & $2g_1 +3g_2 + 3n_1$\\
 $\varphi^2$ & $\infty$  & 0 & 1 &  $-g_1$ & $-2g_1-g_2-n_1$ & $2g_1$ & $5g_1 +2g_2 + 2n_1$ \\
$\varphi$ & 0 & 1 & 0  & $-3g_2-n_2$ & $-g_2$ & $3g_2 + n_2$ & $-n_2$  \\
$\varphi^2$ & 0 & 1 & 0 & $2g_2+n_2$ & $-g_2$ & $-4g_2 - 2n_2$ & $-n_2$ \\
$\varphi$ & $\varphi^2$ & 1 & $\varphi^2$ &  $-2g_1-n_1$ & $-3g_1+2g_2-2n_1+n_2$ & $2g_1 + n_1$ & $7g_1 - 6g_2 + 5n_1 - 3n_2$   \\\hline
\end{tabular}\caption{}\label{tab:origin_orbit}
\end{table}

\begin{lemma}\label{lem:Wp=W_ip}
    If $\boldsymbol{p} \in \Delta_{\Pcal_0,R_0} \setminus \mathcal{O}_{R_0}(\boldsymbol{0})$,
    then there exists $i \in \{0,\infty,\varphi,\varphi^2\}$ such that
    $W(\bp)= W_i(\bp)$.
\end{lemma}

         \begin{proof}
             For every $i \in \{0,\infty,\varphi,\varphi^2\}$,
             we have 
             $W_i(\bp)\subset W(\bp)$.
             If there exists $i,j\in \{0,\infty,\varphi,\varphi^2\}$
             with $i\neq j$
             such that 
             $W_i(\bp) \neq\varnothing$ and
             $W_j(\bp) \neq\varnothing$, then
             there exists $\bn,\bm\in\ZZ^2$ such that
             $R_0^\bn(\bp)\in\Delta_i$
             and
             $R_0^\bm(\bp)\in\Delta_j$.
From Lemma~\ref{lem:origin_orbit},
$\bp \in \mathcal{O}_{R_{0}}(\boldsymbol{0})$ which is a contradiction.
Thus there exists a unique $i\in \{0,\infty,\varphi,\varphi^2\}$ such that
    $W(\bp)= W_i(\bp)$.
\end{proof}

\begin{lemma}\label{lem:Dxy=Wip}
    Let $x,y\in\Xcal_{\Pcal_0,R_0}$ such that $x\neq y$ and $f(x)=f(y)=\bp\in\Delta_{\Pcal_0,R_0}$.
    If $\bp\in\mathcal{O}_{R_0}(\Delta_i)\setminus\mathcal{O}_{R_0}(\boldsymbol{0})$,
    for some $i \in \{0,\infty,\varphi,\varphi^2\}$,
    then $D(x,y)=W(\bp)=W_i(\bp)$.
\end{lemma}

\begin{proof}
    From Lemma~\ref{lem:strategy-to-get-Conway-worms}
    we have $D(x,y)\subseteq W(\bp)$.
    From Lemma~\ref{lem:Wp=W_ip}, we have
    then there exists $i \in \{0,\infty,\varphi,\varphi^2\}$ such that
    $W(\bp)=W_i(\bp)$.

    We now want to show $ W_i(\bp)\subseteq D(x,y)$.
    Let $\bn\in W_i(\bp)$.
    By contradiction, suppose that
    $\bn\notin D(x,y)$.
    Thus $x_\bn=y_\bn$.
    This implies that configurations $x$ and $y$ are limits
    $x=\lim_{k\to\infty} x^{(k)}$
    and $y=\lim_{k\to\infty} y^{(k)}$
    of configurations $x^{(k)},y^{(k)}\in\Xcal_{\Pcal_0,R_0}$, $k\in\NN$,
    such that $f(x^{(k)})$ and $f(y^{(k)})$ are approaching
    a segment of slope~$i$ from the same side 
    (that is, from the left or from the right if the segment is vertical, etc.).

    Now let $\bm\in D(x,y)$.
    By definition,
    $x_\bm$ 
    ($y_\bm$, resp.)
    is the label of the atom of the partition containing the point 
    $R_0^\bm(f(x^{(k)}))$ 
    ($R_0^\bm(f(y^{(k)}))$, resp.)
    for arbitrarily large $k$.
    Since both sequence are approaching a segment of slope $i$ from the same side
    toward the point $R_0^{\bm}(\bp)$ which is not contained in segment of other slopes 
    of the partition,
    we must have $x_\bm=y_\bm$.
    Thus $\bm\notin D(x,y)$ which is a contradiction.
    We obtain that $W_i(\bp)\subseteq D(x,y)$
             and we conclude that $D(x,y)=W(\bp)=W_i(\bp)$.
         \end{proof}

\begin{figure}[h]
\begin{center}
\includegraphics[width=0.75\textwidth]{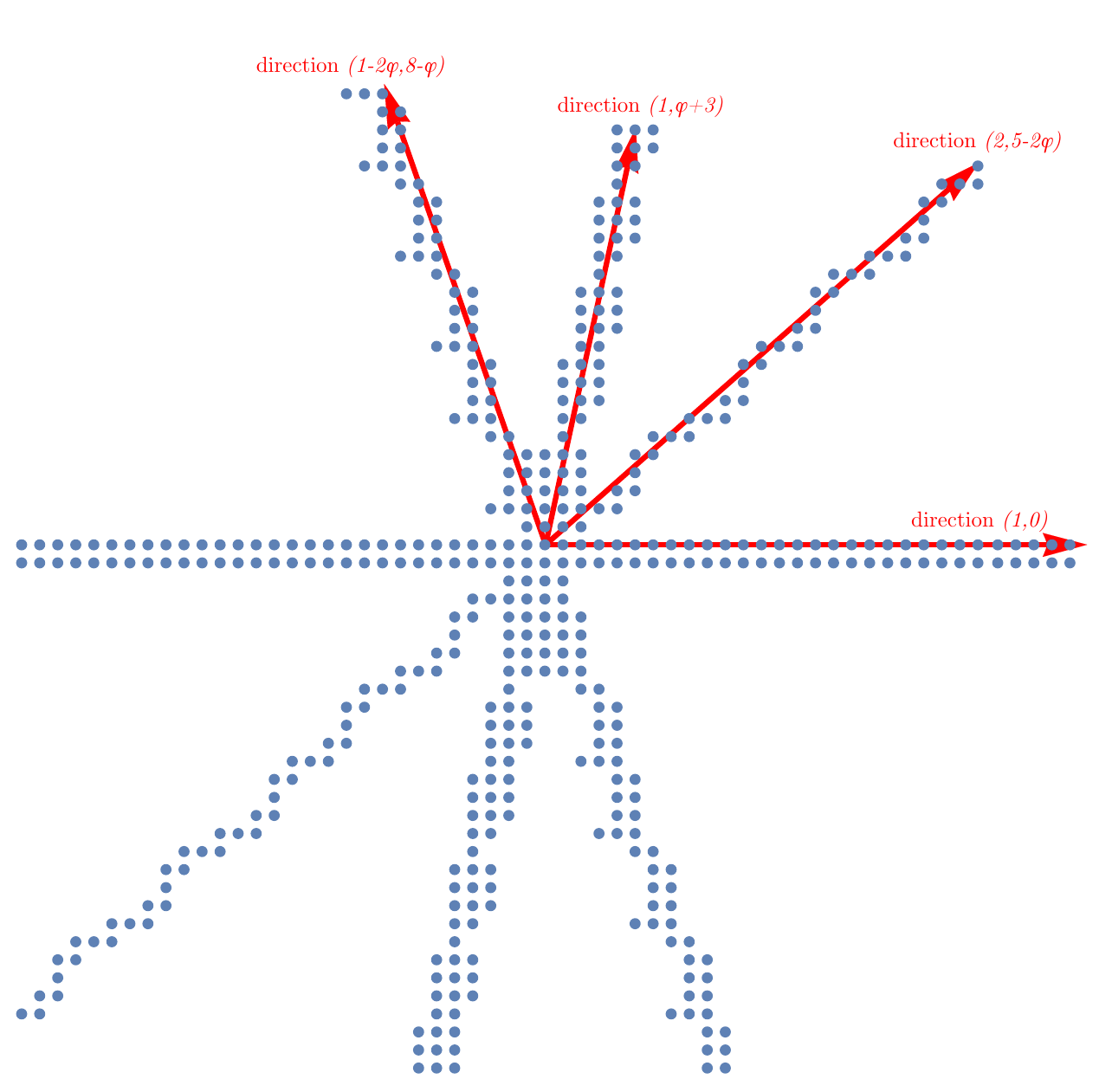}\caption{All four
    nonexpansive directions shown in Theorem~\ref{thm:main}
    are exhibited by the set $W(\boldsymbol{0})$.}
    \label{fig:all_four}
\end{center}
\end{figure}

We may now prove the main result.

\begin{proof}[Proof of Theorem~\ref{thm:main}] By Lemma~\ref{lem:non_exp_dense}, 
    if $H$ is a nonexpansive half-space
    for the subshift $\mathcal{X}_{\mathcal{P}_0,R_0}$,
    then there exist $x,y\in\mathcal{X}_{\mathcal{P}_0,R_0}$
    such that 
    $x|_{H\cap \Z^2}=y|_{H\cap \Z^2}$,
    $x\neq y$
    and $f(x)=f(y)=\bp\in\Delta_{\mathcal{P}_0,R_0}$.

    First assume that
    $\bp\notin\mathcal{O}_{R_0}(\boldsymbol{0})$.
    From Lemma~\ref{lem:Dxy=Wip}, we have
    that there exists $i \in \{0,\infty,\varphi,\varphi^2\}$ such that
    $D(x,y)= W(\bp)=W_i(\bp)$.
    From Proposition~\ref{prop:orbit-remaining-in-partition-boundary},
    $W_i(\boldsymbol{p})$ 
    is at bounded distance from a line of slope $m_i$
    where $m_0=0$,
       $m_{\infty}=\varphi+3$,
       $m_{\varphi}=-3\varphi+2$ and
       $m_{\varphi^2}=-\varphi+\frac{5}{2}$.
    Thus there exists $t>0$ such that $D(x,y)\subset F^t$ where
    $F\subset\RR^2$ is the 1-dimensional vector space of slope $m_i$.
    We conclude that $F$ is nonexpansive 
    for the $\ZZ^2$-shift action $\sigma$ on $\Xcal_{\Pcal_0,R_0}$
    and $D(x,y)$ is a Conway worm associated to $F$.

    Now suppose that
    $\boldsymbol{p}\in\mathcal{O}_{R_0}(\boldsymbol{0})$.
    We have from Lemma~\ref{lem:strategy-to-get-Conway-worms} that
    \[
    D(x,y)\subset W(\bp)
    = W_0(\bp) \cup W_\infty(\bp) \cup W_\varphi(\bp) \cup W_{\varphi^2}(\bp).
    \]
    If there exists a unique $i\in \{0,\infty,\varphi,\varphi^2\}$ 
    such that $D(x,y) \cap W_i(\bp)$ is infinite, 
    then 
    we conclude as above 
    from Proposition~\ref{prop:orbit-remaining-in-partition-boundary}
    that $D(x,y)$ is at bounded distance from a line of slope $m_i$.
    If there exists two $i,j\in \{0,\infty,\varphi,\varphi^2\}$ 
    with $i\neq j$ such that
    $D(x,y) \cap W_i(\bp)$ and
    $D(x,y) \cap W_j(\bp)$ are infinite, 
    then
    we obtain a contradiction because
    $x$ and $y$ must be equal on the half-space $H\cap\ZZ^2$.
\end{proof}

The nonexpansive directions 
mentioned in Theorem~\ref{thm:main}
can be seen in the illustration of $W(\boldsymbol{0})$, see
Figure~\ref{fig:all_four}.

\section{Resolutions of Conway Worms in the Jeandel-Rao Wang shift}\label{sec:resolution}

Let $x,y\in\Xcal_{\Pcal_0,R_0}$ be two configurations in the Jeandel-Rao Wang
shift such that $x\neq y$
and $f(x)=f(y)=\bp\in\Delta_{\mathcal{P}_0,R_0}$.
If $\bp\notin\mathcal{O}_{R_0}(\boldsymbol{0})$, then
    from Lemma~\ref{lem:Dxy=Wip},
    there exists $i \in \{0,\infty,\varphi,\varphi^2\}$ such that
    \[
        D(x,y)=\{\bn\in \ZZ^2\mid x_\bn \neq y_\bn\}
        =\{ \bn \in \ZZ^2 \mid R_0^\bn(\bp) \in \Delta_i\}
        =W_i(\bp).
    \]
From Proposition~\ref{prop:orbit-remaining-in-partition-boundary},
the set $D(x,y)$ is a Conway worm associated to some nonexpansive subspace $F$
whose slope was computed in the previous section.
The goal of this section is to describe the Conway worms $D(x,y)$ and
their resolutions $x|_{D(x,y)}$ and $y|_{D(x,y)}$ within the minimal subshift
$\Xcal_{\Pcal_0,R_0}$ of the Jeandel-Rao Wang shift.

In particular, we show that 
the Conway worms
follow the structure of a two-sided
Sturmian word. 
The resolutions of the Conway worms are coded by a sequence in the Fibonacci
subshift, a Sturmian (mechanical) sequence whose ratio of letter frequencies is the golden
mean.
For every $\alpha\in\RR\setminus\QQ$ with $0<\alpha<1$ and every $n\in\ZZ$, let
\begin{align}
    \label{eq:sturmian1}
    s_{\alpha,\rho}(n) &= \lfloor \alpha(n+1) + \rho \rfloor - \lfloor \alpha n + \rho\rfloor,\\
    \label{eq:sturmian2}
    s^{\prime}_{\alpha,\rho}(n) &= \lceil \alpha(n+1) + \rho \rceil - \lceil \alpha n + \rho\rceil.
\end{align}
be the two-sided \emph{lower (resp. upper) mechanical words} with slope
$\alpha$ and intercept $\rho$. Mechanical words are one of the many equivalent
definitions of Sturmian sequences, see \cite{lothaire_2002}.
A mechanical sequence is the binary encoding of a path in the grid $\ZZ^2$
which is a digital approximation of an Euclidean line of slope $\alpha$. In what follows,
it is convenient to split the support of this path according to the value of
$s_{\alpha,\rho}$ at position $n\in\ZZ$. Thus, let
    \begin{align}
        \label{eq:H}
        H_{\alpha,\rho} &= \{(n, \lfloor\alpha n + \rho \rfloor) \in \ZZ^2 \,|\, s_{\alpha,\rho}(n) = 0\},\\
        \label{eq:V}
        V_{\alpha,\rho} &= \{(n, \lfloor\alpha n + \rho \rfloor) \in \ZZ^2 \,|\, s_{\alpha,\rho}(n) = 1\}.
    \end{align}

In this section, we prove the following result about the resolution
of the Conway worms within the Jeandel-Rao Wang shift.
Upon inspection of the structure of $D(x,y)$, we observe there exist subsets
$B,G\subset\ZZ^2$ that can be arranged in such a way to create $D(x,y)$, see
Figure~\ref{fig:other-3-Conway-worms-in-JR} 
and Figure~\ref{fig:all-Conway-worms-in-JR-in-one-image}.
The structure of that arrangement, along with details about the two resolutions
of the Conway worm are given in the following result.

\begin{maintheorem} 
    \label{thm:resolution}
    Let $x,y\in\Xcal_{\Pcal_0,R_0}$ be two configurations in the Jeandel-Rao Wang
    shift such that $x\neq y$
    and $f(x)=f(y)=\bp\in\Delta_{\Pcal_0,R_0}\setminus\mathcal{O}_{R_0}(\boldsymbol{0})$.
    There exist 
    finite subsets $B,G\subset\ZZ^2$ and
    a matrix $M$
    such that the Conway worm is
    \[
        D(x,y)
        = \bk + 
        \big(\left(M H_{\alpha,\rho} + B\right) 
        \cup
        \left(M V_{\alpha,\rho} + G\right) \big)
    \]
    for some $\bk\in\ZZ^2$, $0<\rho<1$
    where 
    $\alpha = 2 - \varphi=2-\frac{1+\sqrt{5}}{2}$.
    Moreover, 
    there exist 
    patterns $b^{-}$, $b^{+}$ of support $B$
    and patterns $g^{-}$, $g^{+}$ of support $G$
    such that the two resolutions $x|_{D(x,y)}$ and $y|_{D(x,y)}$ of the Conway
    worm are constructed using these, that is,
    \begin{align*}
      x|_{\bk+M \boldsymbol{h}+B} = b^{+}
              \quad
              \text{ and }
              \quad
      x|_{\bk+M \boldsymbol{v}+G} = g^{+}\\
      y|_{\bk+M \boldsymbol{h}+B} = b^{-}
              \quad
              \text{ and }
              \quad
      y|_{\bk+M \boldsymbol{v}+G} = g^{-}
    \end{align*}
    for every $\boldsymbol{h} \in H_{\alpha,\rho}$ 
          and $\boldsymbol{v} \in V_{\alpha,\rho}$.
\end{maintheorem}

\subsection{The Fibonacci subshift appearing vertically in the Jeandel-Rao Wang shift}

As observed by Jeandel and Rao \cite{JR1}, in every configuration in the
Jeandel-Rao Wang shift, the tiles labeled by 0 or 1 appear in rows where only
tiles 0 and 1 can appear. Also, the distance between these rows is 4 or 5.
Therefore, every configurations can be split by infinite horizontal strips of height 4 or 5
with the tiles 0 or 1 appearing at the bottom of each infinite strip.
To prove aperiodicity of the Jeandel-Rao Wang shift, 
Jeandel and Rao proved that the possible sequences of 4 and 5 is exactly the language
of the Fibonacci word \cite{berstel_mots_1980} over the alphabet $\{4,5\}$.
In other words, a biinfinite sequence in $\{4,5\}^\ZZ$ is in the Fibonacci subshift
generated by the substitution $4\mapsto 5,5\mapsto 54$
if and only if it is the sequence of heights of horizontal strips of a valid configuration in
the Jeandel-Rao Wang shift.

These strips are illustrated by the color of the tiles in Figures
\ref{fig:other-3-Conway-worms-in-JR} and their heights is shown in the right margin of
Figure~\ref{fig:all-Conway-worms-in-JR-in-one-image}.
The sequence of heights of complete blocks seen in the figures from bottom to top is either
\begin{align*}
    &5,4,5,5,4,5 \text{ or }\\
    &5,5,4,5,4,5
\end{align*}
which could be extended to longer biinfinite sequences
\begin{align*}
    &\dots
 ,5,4,5,5,4,5,4,5,5,4,
 \underline{5,4,5,5,4,5},
 5,4,5,4,5,5,4,5,5,4,\dots %
    \text{ or }\\
    &\dots
 ,5,4,5,5,4,5,4,5,5,4,
 \underline{5,5,4,5,4,5},
 5,4,5,4,5,5,4,5,5,4,\dots %
\end{align*}
where the heights of complete blocks seen in the figures are underlined.
Notice that flipping the Conway worm of slope 0 from one resolution to the other
flips a $4,5$ into a $5,4$ in the above sequence, see
Figure~\ref{fig:all-Conway-worms-in-JR-in-one-image}.
After replacing $5$ by $\texttt{0}$
and $4$ by $\texttt{1}$,
the above sequences correspond to the two biinfinite Fibonacci word \cite{berstel_mots_1980},
that is,
the lower and upper Sturmian sequences of slope $\alpha=\varphi^{-2}=2-\varphi$:
\begin{align*}
    s_{\varphi^{-2},0}&=  \dots \texttt{01001010010}\fbox{\texttt{10}}.\texttt{0100101001001}\dots\\
    s'_{\varphi^{-2},0}&= \dots \texttt{01001010010}\fbox{\texttt{01}}.\texttt{0100101001001}\dots
\end{align*}
Thus the structure of the Fibonacci word, a particular example of a Sturmian
sequence, is required in the sequence of Lemmas to follow.

For $\alpha \in (0,1)$, define the rotation $T_{\alpha}$ of $[0,1)$ by \begin{equation}T_{\alpha}(\rho) = \rho + \alpha \text{ (mod } 1 \text{)}\label{eqn:IET-T}\end{equation} For $\rho \in [0,1)$, $T_\alpha$ encodes biinfinite words $u_{\rho} = \ldots u_{-3} u_{-2} u_{-1} u_0 u_1 u_2 u_3 \ldots$ in the alphabet $\{0,1\}$ according to the formula \[u_n = \begin{cases} 
0 & \text{ if } T_{\alpha}^{n}(\rho) \in [0,1-\alpha) \\  
1 & \text{ if } T_{\alpha}^{n}(\rho) \in [1-\alpha,1)\end{cases}. \] 
We wish to connect $T_\alpha$ to a geometric line of slope $\alpha$, and happily such a connection is well-known.
The words $s_{\alpha,\rho}$ and $s^{\prime}_{\alpha,\rho}$ 
from Equations~\eqref{eq:sturmian1} and~\eqref{eq:sturmian2}
are associated with polygonal paths with corners in $\ZZ^2$ that approximate the line $y = \alpha x + \rho$; the lower mechanical path contains the $\ZZ^2$ points $P_{\alpha,\rho}(n) = (n, \lfloor\alpha n + \rho \rfloor)$ and the upper mechanical path contains the points $P^{\prime}_{\alpha,\rho}(n)  = (n,\lceil\alpha n + \rho\rceil)$. 
It is seen that $s(n) = 0$ when $P_{\alpha,\rho}$ is horizontal between $P_{\alpha,\rho}(n)$ and $P_{\alpha,\rho}(n+1)$, and $s(n) = 1$ when $P_{\alpha,\rho}$ angles upward between $P_{\alpha,\rho}(n)$ and $P_{\alpha,\rho}(n+1)$ (and similarly for the relationship between $s^{\prime}_{\alpha,\rho}$ and $P^{\prime}_{\alpha,\rho}$). 
We note that when $\alpha$ is irrational, $s_{\alpha,\rho} = s^{\prime}_{\alpha,\rho}$ most of the time (except at possibly 2 consecutive values of $n$) and for the application in this article, they are equal all the time, so we shall refer only to $s$ and $P$ from here on.  
We get the following nice fact from \cite{lothaire_2002}: \begin{equation}u_n = s_{\alpha,\rho}(n)\label{eqn:T=s}\end{equation} for all $n \in \ZZ$. 
That is, the word $u_\rho$ encoded by $T_{\alpha}$ on input $\rho$ is the same word encoded by $s_{\alpha,\rho}$. 
Also, notice that for all $n \in \ZZ^2$, \begin{equation} P_{\alpha,\rho}(n) = k(1,0) + \ell(1,1)\label{eqn:num_hor_diag}\end{equation} where $k$ is the number of horizontal segments and $\ell$ is the number of diagonal segments in the mechanical line between $(0,0)$ and $P_{\alpha,\rho}(n)$.
The set $P_{\alpha,\rho}$ is partitioned into two sets of points, $H_{\alpha,\rho}$ and $V_{\alpha,\rho}$, see Equations~\eqref{eq:H} and~\eqref{eq:V};
$H_{\alpha,\rho}$ is the set of starting points of horizontal segments in the lower mechanical path and $V_{\alpha,\rho}$ is the starting points of diagonal segments in the lower mechanical path. 
 
 \subsection{Patterns $B$ of height 5 and patterns $G$ of height 4}
Next we want to describe the resolutions of the Conway worms. 
In what follows, except for slope $i=0$, 
$B\subset\ZZ^2$ will be the support of a pattern of height $5$ and 
$G\subset\ZZ^2$ will be the support of a pattern of height $4$, see
Figure~\ref{fig:other-3-Conway-worms-in-JR} and 
Figure~\ref{fig:all-Conway-worms-in-JR-in-one-image}.

We assume that the support $B$ is shifted in such a way that
$(0,0)\in B$, 
$\min\{j\mid (i,j) \in B\}=0$
and $\min\{i\mid (i,j) \in B \text{ and } j=0\}=0$
and similarly for $G$.
The support $B$ and $G$ will be translated by
some \emph{placement vectors}
$\boldsymbol{b} = (b_1,b_2)$ and $\boldsymbol{g}=(g_1,g_2)$.
The placement vectors $\boldsymbol{b}$ and $\boldsymbol{g}$ give the translation from the origin 
of a $B$ (resp. $G$) pattern in a Conway worm to the origin of the pattern that follows.

The points in $H_{\alpha,\rho}$ correspond to starting points of $B$ patterns and the points of $V_{\alpha,\rho}$ correspond to the starting points of $G$ patterns, but the points of $H$ and $V$ must be appropriately rescaled by the placement vectors to account for the sizes of $B$ and $G$. We accomplish this by utilizing Equation~\eqref{eqn:num_hor_diag}: Notice that \begin{equation}\left(\begin{smallmatrix} b_1 & g_1 - b_1 \\ b_2 & g_2 - b_2 \end{smallmatrix}\right)\left(k \left(\begin{smallmatrix} 1 \\ 0 \end{smallmatrix}\right)  + \ell \left(\begin{smallmatrix} 1 \\ 1 \end{smallmatrix}\right) \right) = k \boldsymbol{b} + \ell \boldsymbol{g}.\label{eqn:shifted_patterns}\end{equation}  Thus the set of starting points for all $B$ patterns in the Conway worm will be of the form \begin{equation} \left(\begin{smallmatrix} b_1 & g_1 - b_1 \\ b_2 & g_2-b_2 \end{smallmatrix}\right) H_{\alpha,\rho} \label{eqn:B_places}\end{equation} and the set of starting points for all $G$ patterns in the Conway worm will be of the form  \begin{equation} \left(\begin{smallmatrix} b_1 & g_1 - b_1 \\ b_2 & g_2 - b_2 \end{smallmatrix}\right) V_{\alpha,\rho}.\label{eqn:G_places}\end{equation} 
    
\subsection{Proof of Theorem~\ref{thm:resolution}}

We proceed by separately considering the four possible slopes of
$\Delta_{\Pcal_0}$-lines on which $\boldsymbol{p}$ lies. 

\subsection*{Case: $\boldsymbol{p} \in \Delta_{\varphi^2}$.}\mbox{}\\

\begin{lemma} 
    Let $x,y\in\Xcal_{\Pcal_0,R_0}$ be two configurations in the Jeandel-Rao Wang
    shift such that $x\neq y$
    and $f(x)=f(y)=\bp\in\Delta_{\varphi^2}\setminus\mathcal{O}_{R_0}(\boldsymbol{0})$.
    Suppose that $\bp$ lies on the segment $Z_{\varphi^2}$ with endpoints $(\varphi^{-1},0)$ 
            and $(1,1)$, that is, let 
    $\bp=(1-\rho)(\varphi^{-1},0)+\rho(1,1)$ for some $\rho\in\RR$ such that 
    $0<\rho<1$.
    Then there exist 
    finite subsets $B_{\varphi^2},G_{\varphi^2}\subset\ZZ^2$ such that the Conway worm is
    \begin{equation}
        W_{\varphi^2}(\bp) = \left[\left(\begin{smallmatrix} 6 & -2 \\ 5 & -1 \end{smallmatrix}\right) H_{\alpha,\rho} + B_{\varphi^2}\right] 
        \cup
        \left[\left(\begin{smallmatrix} 6 & -2 \\ 5 & -1 \end{smallmatrix}\right) V_{\alpha,\rho} + G_{\varphi^2}\right] \label{eqn:orbit_formula} \end{equation} 
            where $\alpha=\varphi^{-2}=2-\varphi$.
    Moreover, 
    there exist 
        patterns $b^{-}_{\varphi^2}$, $b^{+}_{\varphi^2}$ of support $B_{\varphi^2}$
        and patterns $g^{-}_{\varphi^2}$, $g^{+}_{\varphi^2}$ of support $G_{\varphi^2}$
    such that the two resolutions $x|_{D(x,y)}$ and $y|_{D(x,y)}$ of the Conway
    worm are constructed using these, that is,
    \begin{align*}
      x|_{\left(\begin{smallmatrix} 6 & -2 \\ 5 & -1 \end{smallmatrix}\right) 
              \boldsymbol{h}+B_{\varphi^2}} = b^{+}_{\varphi^2}
              \quad
              \text{ and }
              \quad
      x|_{\left(\begin{smallmatrix} 6 & -2 \\ 5 & -1 \end{smallmatrix}\right) 
              \boldsymbol{v}+G_{\varphi^2}} = g^{+}_{\varphi^2}\\
      y|_{\left(\begin{smallmatrix} 6 & -2 \\ 5 & -1 \end{smallmatrix}\right) 
              \boldsymbol{h}+B_{\varphi^2}} = b^{-}_{\varphi^2}
              \quad
              \text{ and }
              \quad
      y|_{\left(\begin{smallmatrix} 6 & -2 \\ 5 & -1 \end{smallmatrix}\right) 
              \boldsymbol{v}+G_{\varphi^2}} = g^{-}_{\varphi^2}
    \end{align*}
    for every $\boldsymbol{h} \in H_{\alpha,\rho}$ 
          and $\boldsymbol{v} \in V_{\alpha,\rho}$.
\label{lem:phi-sq-rot-new}\end{lemma}

\begin{proof}
    Assume $\boldsymbol{p}$ lies on the $\Delta_{\varphi^2}$-line from $P=(\varphi -1,0)$ to $Q=(1,1)$. Divide $\overline{PQ}$ into two segments, one labeled $B$ of length $b =\sqrt{15 - 9 \varphi}$ starting at $P$ and another labeled $G$ of length $g = \sqrt{39 - 24 \varphi}$ starting where $B$ ends.  
We illustrate in Figure~\ref{fig:IET_phi_sq} the orbit of the point $\bp$ under the
    $\ZZ^2$-action $R_0$ which stays inside $\Delta_{\varphi^2}+\Gamma_0$ until it returns
    to the original segment $\overline{PQ}$.
    It reveals a few key observations: \begin{itemize} 
\item If $\boldsymbol{p}$ lies on the segment labeled $G \subseteq \overline{PQ}$ and 
    \[
G_{\varphi^2} = \{(0,0),(0,1),(1,1),(1,2),(2,2),(3,2),(3,3), (4,3)\},
    \]
    then $\boldsymbol{p} + G_{\varphi^2}$ traces out a set of points starting at $\boldsymbol{p}$ and returning $\boldsymbol{p}$ to $\overline{PQ}$ in the torus $\boldsymbol{T}$.
\item If $\boldsymbol{p}$ lies on the segment labeled $B \subseteq \overline{PQ}$ and
    \[
B_{\varphi^2} = \{(0,0),(0,1),(1,1),(1,2),(2,2),(3,2),(3,3), (4,3), (5,3),(5,4),(6,4)\},
    \]
    then $\boldsymbol{p} + B_{\varphi^2}$ traces out a set of points starting at $\boldsymbol{p}$ and returning $\boldsymbol{p}$ to $\overline{PQ}$ in $\boldsymbol{T}$.
\item The action of moving $\boldsymbol{p}$ in this way is captured by a rotation $T$. Indeed, observe that $ b/(b+g) = \varphi - 1$ and $ g/(b+g) = 2 - \varphi$, so that, after scaling by $1/(b+g)$, we can see that the exchange of intervals $B$ and $G$ is captured by the rotation $T_\alpha$ where $\alpha = 2 - \varphi = (3 - \sqrt{5})/2$. In this rotation, the point $\boldsymbol{p}$ corresponds to $\rho = d(\boldsymbol{p},(\varphi-1,0))/|\overline{PQ}|$, and in the encoding $u_\rho$ of $\rho$, 0 corresponds to $B$ and 1 corresponds to $G$. Thus $T_\alpha$ encodes two-sided Sturmian words in the alphabet $\{B,G\}$.

\end{itemize}

\begin{figure}[h]
\begin{center}
    \includegraphics[width=1\textwidth]{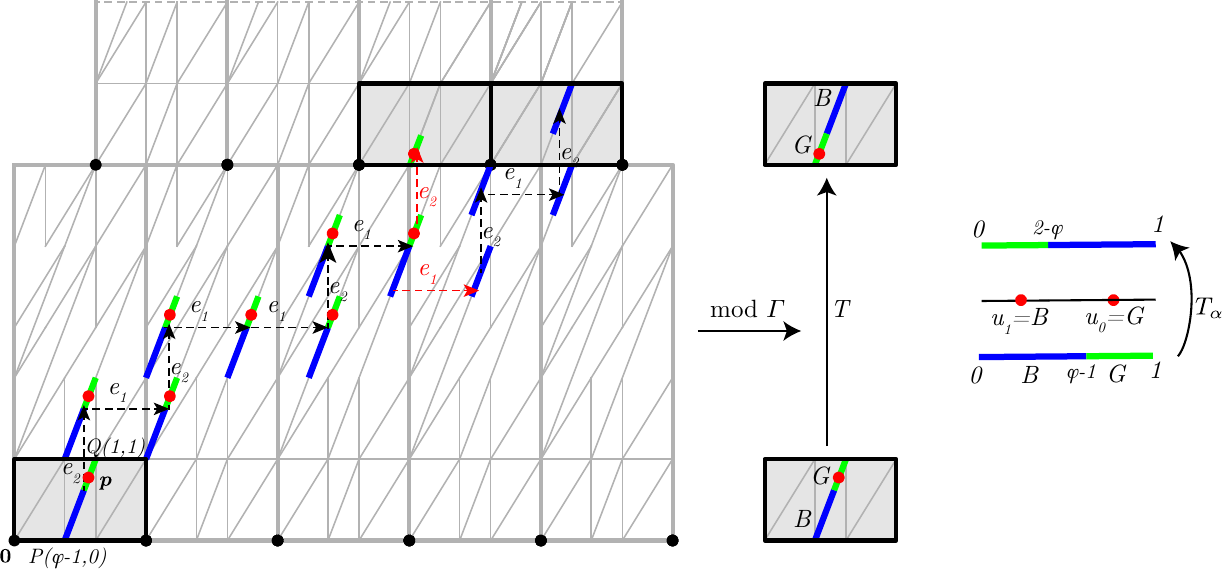}\caption{The $\ZZ^2$ action induces an exchange of the intervals (or rotation of) $B$ and $G$. A point $\boldsymbol{p}$ on the segment $\overline{PQ}$ from $(\varphi - 1,0)$ to $(1,1)$ will return to $\overline{PQ}$ in a manner captured by the rotation $T$ of $[0,1]$.}\label{fig:IET_phi_sq}
\end{center}
\end{figure}

    Next, we define the patterns 
    $b^{+}_{\varphi^2}$ and $b^{-}_{\varphi^2}$ of support $B_{\varphi^2}$
    and the patterns
    $g^{+}_{\varphi^2}$ and $g^{-}_{\varphi^2}$ of support $G_{\varphi^2}$,
    as depicted in Figure~\ref{fig:blocks-phi-sq}. 
    The patterns 
    $b^{+}_{\varphi^2}$ 
    and
    $g^{+}_{\varphi^2}$ 
    are encoding the behavior of points approaching the segment $\overline{PQ}$ ``from the right"
    and
    the patterns 
    $b^{-}_{\varphi^2}$ 
    and
    $g^{-}_{\varphi^2}$ 
    are encoding the behavior of points approaching the segment $\overline{PQ}$ ``from the left".

\begin{figure}[h]
\centering
\begin{subfigure}[b]{.2\textwidth} 
\centering
\includegraphics[scale = .5]{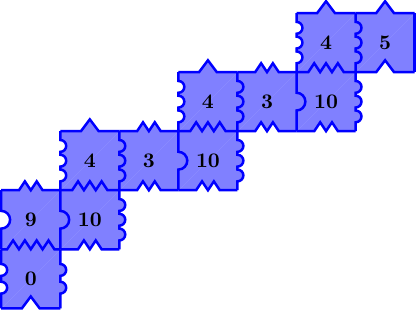}  
\caption{$b^{+}_{\varphi^2}$}
\label{fig:b-plus-phi-sq}
\end{subfigure}
\begin{subfigure}[b]{.2\textwidth} 
\centering
\includegraphics[scale = .5]{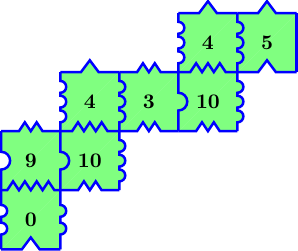}  
\caption{$g^{+}_{\varphi^2}$}
\label{fig:g-plus-phi-sq}
\end{subfigure}
\begin{subfigure}[b]{.2\textwidth} 
\centering
\includegraphics[scale = .5]{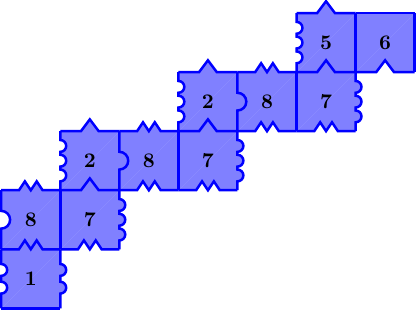}  
\caption{$b^{-}_{\varphi^2}$}
\label{fig:b-minus-phi-sq}
\end{subfigure}
\begin{subfigure}[b]{.2\textwidth} 
\centering
\includegraphics[scale = .5]{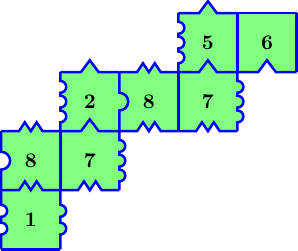}  
\caption{$g^{-}_{\varphi^2}$}
\label{fig:g-minus-phi-sq}
\end{subfigure}
\caption{The patterns $b^+_{\varphi^2}$, $g^+_{\varphi^2}$, $b^-_{\varphi^2}$, and $g^-_{\varphi^2}$.}
\label{fig:blocks-phi-sq} 
\end{figure}

    Thus we see that $T_\alpha$ encodes a sequence of patterns (in the alphabet $\{b^{+}_{\varphi^2},g^{+}_{\varphi^2} \}$ or $\{b^{-}_{\varphi^2},g^{-}_{\varphi^2}\}$) in the nonexpansive direction associated with the slope $\varphi^2$, so $\mathcal{O}_{R_0}(\boldsymbol{p}) \cap \Delta_{\Pcal_0}$ has the structure of a Sturmian word. As described in Equation~\eqref{eqn:T=s}, the mechanical word $s_{\alpha,\rho}$ emulates $T_\alpha$ and is easier to relate to the translations needed to place the patterns, so we switch to $s_{\alpha,\rho}$ for this encoding from here on. From the $b_{\varphi^2}$ and $g_{\varphi^2}$ patterns in Figure~\ref{fig:blocks-phi-sq}, we see that the placement vectors for the $B$ and $G$ patterns are $\boldsymbol{b} = (6,5)$ and $\boldsymbol{g} =(4,4)$, respectively. Thus, utilizing Equations~\eqref{eqn:B_places} and~\eqref{eqn:G_places}, we get set of starting points for the $B$ patterns to be \[ \left(\begin{smallmatrix} 6 & -2 \\ 5 & -1 \end{smallmatrix}\right) H_{\alpha,\rho} \] and the set of starting points for the $G$ patterns will be  \[\left(\begin{smallmatrix} 6 & -2 \\ 5 & -1 \end{smallmatrix}\right) V_{\alpha,\rho}.\] Knowing that these are the starting points for the $B$ and $G$ patterns, truth of the lemma follows.\end{proof}

        To illustrate ideas with an example, let us choose the point $\boldsymbol{p} = (\varphi - 1,0) + (9/10)(2-\varphi,1)$, which is 9/10 of the way along the slope-$\varphi^2$ segment $\overline{PQ}$ in Figure~\ref{fig:IET_phi_sq}, so $\rho = 9/10$. To compute $\mathcal{O}_{R_0}(\boldsymbol{p}) \cap \Delta_{\Pcal_0}$, let $\alpha = 2-\varphi$ and $\rho = 9/10$. The word $s_{\alpha,\rho} = \ldots 010100.101001\ldots$ gives the pattern of $B$ and $G$ patterns along the Conway worm ($B$ is associated with 0 and $G$ with 1). Next, we find \[H_{\alpha,\rho} = \{\ldots,(-6, -2), (-4, -1), (-2, 0), (-1, 0), \text{\hspace{.2in}} (1, 1), (3, 2), (4, 2),\ldots\}\] and \[V_{\alpha,\rho} = \{\ldots,(-5, -2), (-3, -1), \text{\hspace{.2in}} (0, 0), (2, 1), (5, 2),\ldots\}\] (the larger horizontal space in $H_{\alpha,\rho}$ and $V_{\alpha,\rho}$ corresponds to the radix point in $s_{\alpha,\rho}$). From $H_{\alpha,\rho}$ and $V_{\alpha,\rho}$ we calculate \[\left(\begin{smallmatrix} 6 & -2 \\ 5 & -1 \end{smallmatrix}\right) H_{\alpha,\rho} = \{\ldots,(-32, -28), (-22, -19), (-12, -10), (-6, -5), \text{\hspace{.2in}} (4, 4), (14, 13), (20, 18),\ldots\}\] and \[\left(\begin{smallmatrix} 6 & -2 \\ 5 & -1 \end{smallmatrix}\right) V_{\alpha,\rho} = \{\ldots,(-26, -23), (-16, -14), \text{\hspace{.2in}} (0, 0), (10, 9), (26, 23),\ldots\}.\] Notice that the points in $\left(\begin{smallmatrix} 6 & -2 \\ 5 & -1 \end{smallmatrix}\right) H_{\alpha,\rho}$ are exactly the beginning points of the $B$ patterns in Figure~\ref{fig:phi-sq-blocks-tiling} and the points $\left(\begin{smallmatrix} 6 & -2 \\ 5 & -1 \end{smallmatrix}\right) V_{\alpha,\rho}$ are exactly the starting points of the $G$ patterns in Figure~\ref{fig:phi-sq-blocks-tiling}.

\begin{figure}[h]
\centering
\begin{subfigure}[b]{.35\textwidth} 
\centering
\includegraphics[scale = .6 ]{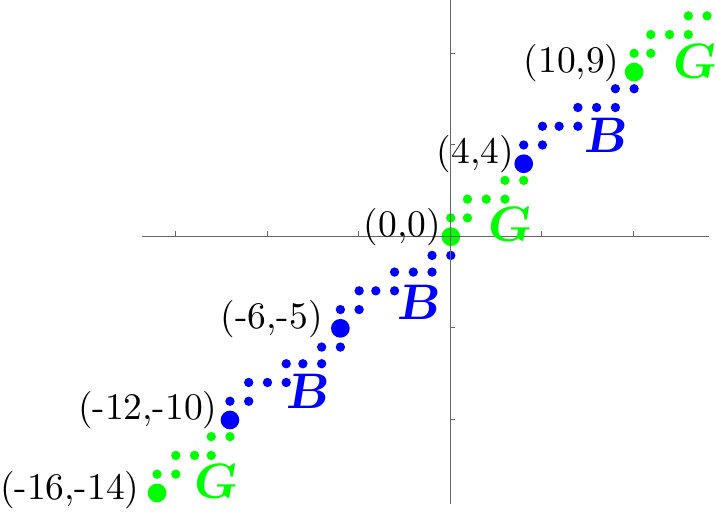}  
    \caption{The elements of $\ZZ^2$ taking\\ $\boldsymbol{p} = (\varphi - 1,0) + (9/10)(2-\varphi,1)$\\ to $\Delta_{\Pcal_0}$.}\label{fig:phi_sq_orbit}
\end{subfigure}
\begin{subfigure}[b]{.6\textwidth} 
\centering
\includegraphics[scale = .3]{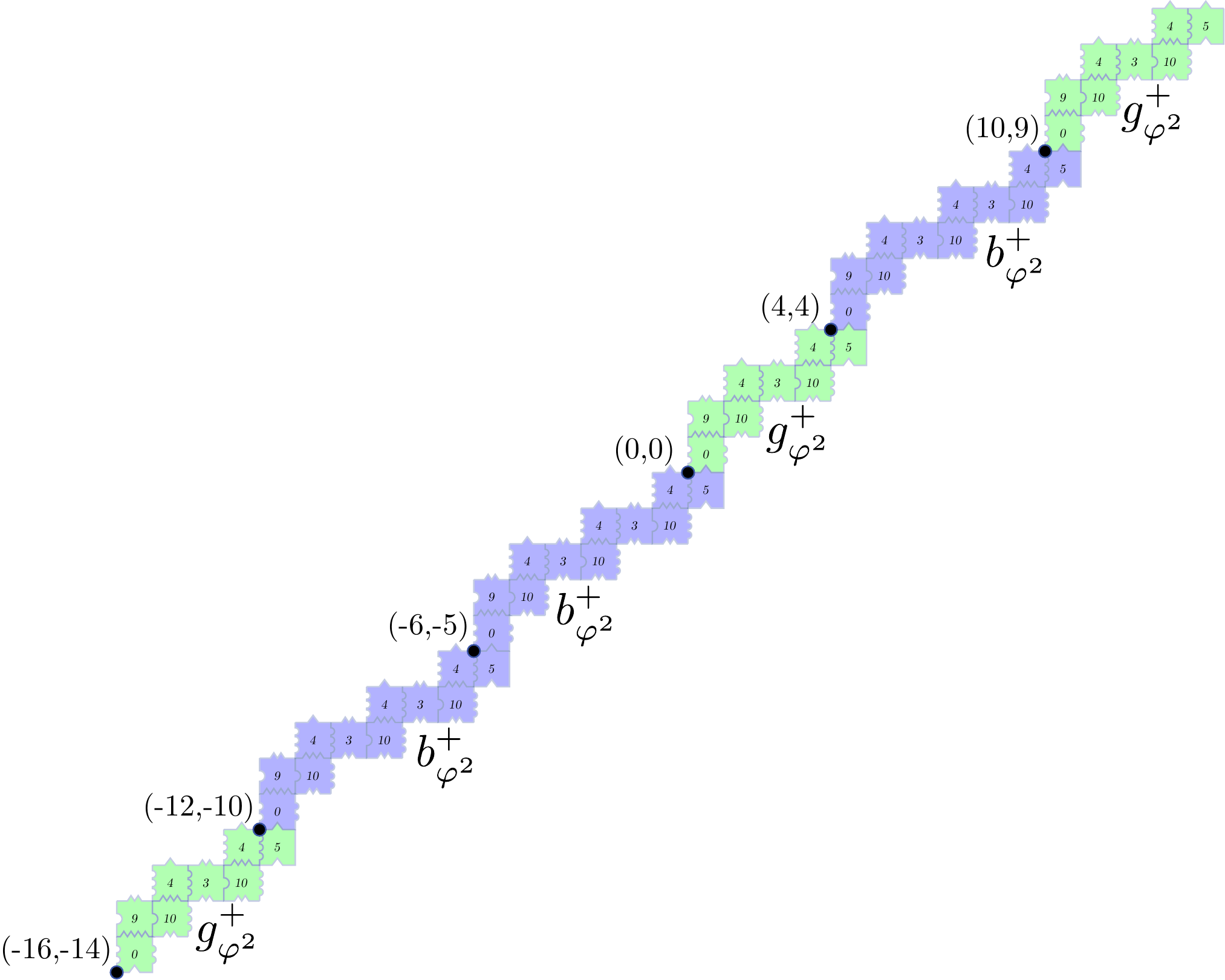}  
    \caption{A partial tiling corresponding to $\mathcal{O}_{R_0}(\boldsymbol{p}) \cap \Delta_{\Pcal_0}$ (with the ``$+$" choice for labeling)}\label{fig:phi-sq-BG-blocks}
\end{subfigure}\caption{An example of the Sturmian structure of a worm $\mathcal{O}_{R_0}(\boldsymbol{p}) \cap \Delta_{\Pcal_0}$}\label{fig:phi-sq-blocks-tiling} 
\end{figure}

\subsection*{Case: $\boldsymbol{p} \in \Delta_{\varphi}$}\mbox{}\\

The following lemma describe the resolution of the Conway worm corresponding to $\Delta_{\varphi}$. 
This case and the remaining cases proceed along the same lines as the slope-$\varphi^2$ case. 

\begin{lemma} 
    Let $x,y\in\Xcal_{\Pcal_0,R_0}$ be two configurations in the Jeandel-Rao Wang
    shift such that $x\neq y$
    and $f(x)=f(y)=\bp\in\Delta_{\varphi}\setminus\mathcal{O}_{R_0}(\boldsymbol{0})$.
    Suppose that $\bp$ lies on the segment $Z_\varphi$ with endpoints $(0,0)$ 
            and $(\varphi^{-1},1)$, that is, let 
    $\bp=(1-\rho)(0,0)+\rho(\varphi^{-1},1)$ for some $\rho\in\RR$ such that $0<\rho<1$.
    Then there exist 
    finite subsets $B_{\varphi},G_{\varphi}\subset\ZZ^2$ such that the Conway worm is
    \begin{equation}
        W_{\varphi}(\bp) = \left[\left(\begin{smallmatrix} -2 & 1 \\ 5 & -1 \end{smallmatrix}\right) H_{\alpha,\rho} + B_{\varphi}\right] 
        \cup
        \left[\left(\begin{smallmatrix} -2 & 1 \\ 5 & -1 \end{smallmatrix}\right) V_{\alpha,\rho} + G_{\varphi}\right] \label{eqn:orbit_formula} \end{equation} 
            where $\alpha=2-\varphi$.
    Moreover, 
    there exist 
        patterns $b^{-}_{\varphi}$, $b^{+}_{\varphi}$ of support $B_{\varphi}$
        and patterns $g^{-}_{\varphi}$, $g^{+}_{\varphi}$ of support $G_{\varphi}$
    such that the two resolutions $x|_{D(x,y)}$ and $y|_{D(x,y)}$ of the Conway
    worm are constructed using these, that is,
    \begin{align*}
      x|_{\left(\begin{smallmatrix} -2 & 1 \\ 5 & -1 \end{smallmatrix}\right) 
              \boldsymbol{h}+B_{\varphi}} = b^{+}_{\varphi}
              \quad
              \text{ and }
              \quad
      x|_{\left(\begin{smallmatrix} -2 & 1 \\ 5 & -1 \end{smallmatrix}\right) 
              \boldsymbol{v}+G_{\varphi}} = g^{+}_{\varphi}\\
      y|_{\left(\begin{smallmatrix} -2 & 1 \\ 5 & -1 \end{smallmatrix}\right) 
              \boldsymbol{h}+B_{\varphi}} = b^{-}_{\varphi}
              \quad
              \text{ and }
              \quad
      y|_{\left(\begin{smallmatrix} -2 & 1 \\ 5 & -1 \end{smallmatrix}\right) 
              \boldsymbol{v}+G_{\varphi}} = g^{-}_{\varphi}
    \end{align*}
    for every $\boldsymbol{h} \in H_{\alpha,\rho}$ 
          and $\boldsymbol{v} \in V_{\alpha,\rho}$.
\label{lem:phi-rot-new}\end{lemma}

\begin{proof} Assume $\boldsymbol{p}$ lies on the $\Delta_{\varphi}$-line
    from $P=(0,0)$ to $Q=(\varphi-1,1)$, as seen at the bottom of
    Figure~\ref{fig:IET_phi}).
We illustrate in Figure~\ref{fig:IET_phi} the orbit of the point $\bp$ under the
    $\ZZ^2$-action $R_0$ which stays inside $\Delta_{\varphi}+\Gamma_0$ until it returns
    to the original segment $\overline{PQ}$.
    We divide the segment $\overline{PQ}$ into two subintervals: the blue subinterval $B$ has length $b = \sqrt{7 - 4\varphi}$ and the green subinterval $G$ has length $g = \sqrt{18-11\varphi}$. In Figure~\ref{fig:IET_phi} the nearby $\ZZ^2$-translates of these subintervals are depicted, and by inspection one sees that, as before, an irrational rotation of the unit interval captures the $\ZZ^2$ action on these points; specifically, comparing the shaded box on the lower right and the two shaded boxes toward the top of Figure~\ref{fig:IET_phi}, which are equivalent modulo $\Gamma_0$, shows this rotation. Just as in the $\varphi^2$ case, we observe that $ b/(b+g) = \varphi - 1$ and $ g/(b+g) = 2 - \varphi$, so that, after scaling by $1/(b+g)$, we see this rotation is again captured by the rotation $T_\alpha$ of Equation~\eqref{eqn:IET-T} with $\alpha = 2 - \varphi$.  

\begin{figure}[h]
\centering
\begin{subfigure}[b]{.65\textwidth} 
\centering
\includegraphics[scale = .75 ]{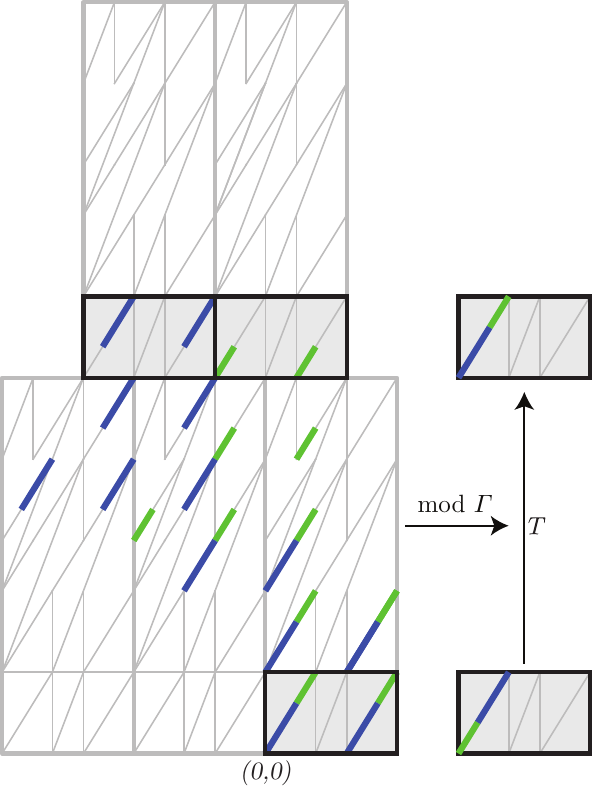}  
\caption{The rotation $T_{\varphi}$}\label{fig:IET_phi}
\end{subfigure}
\begin{subfigure}[b]{.3\textwidth} 
\centering
\includegraphics[scale = .48]{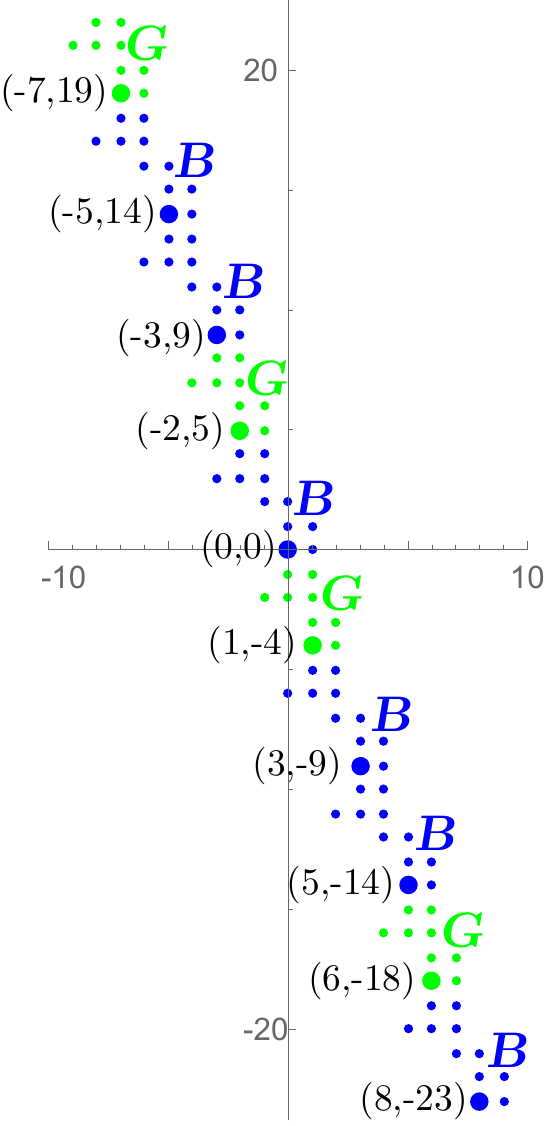}  
    \caption{The elements of $\ZZ^2$ taking the point $p = (1/4, \varphi/4)$ to $\Delta_{\Pcal_0}$.}\label{fig:phi_orbit}
\end{subfigure}\caption{The dynamics of slope-$\varphi$ segments in $\Delta_{\Pcal_0}$ under $\ZZ^2$ translation.}\label{fig:phi_dynamics} 
\end{figure}

\begin{figure}[h]
\centering
\begin{subfigure}[b]{.2\textwidth} 
\centering
\includegraphics[scale = .5]{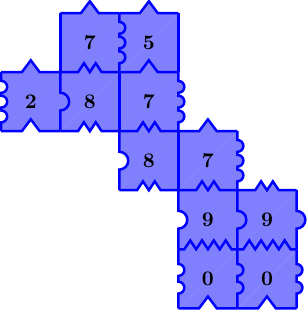}  
\caption{$b^{+}_{\varphi}$}
\label{fig:b-plus-phi}
\end{subfigure}
\begin{subfigure}[b]{.2\textwidth} 
\centering
\includegraphics[scale = .5]{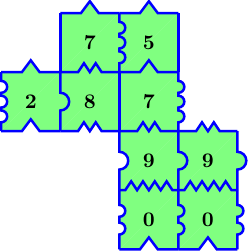}  
\caption{$g^{+}_{\varphi}$}
\label{fig:g-plus-phi}
\end{subfigure}
\begin{subfigure}[b]{.2\textwidth} 
\centering
\includegraphics[scale = .5]{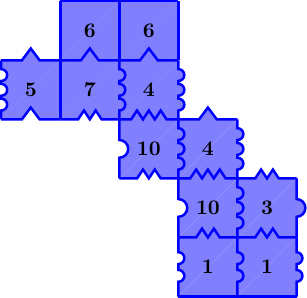}  
\caption{$b^{-}_{\varphi}$}
\label{fig:b-minus-phi}
\end{subfigure}
\begin{subfigure}[b]{.2\textwidth} 
\centering
\includegraphics[scale = .5]{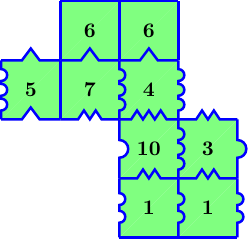}  
\caption{$g^{-}_{\varphi}$}
\label{fig:g-minus-phi}
\end{subfigure}
\caption{The patterns $b^+_{\varphi}$, $g^+_{\varphi}$, $b^-_{\varphi}$, and $g^-_{\varphi}$.}
\label{fig:blocks-phi} 
\end{figure}

    Keeping track of the intermediate unit $\ZZ^2$ shifts that move the blue and green segments from their starting position to their reversed positions at the top of Figure~\ref{fig:IET_phi} under $T_{\alpha}$ and associating the symbols $B$ and $G$ with these lists of $\ZZ^2$ shifts prompts us to define the supports
    \begin{align*}
B_{\varphi} &=  \{(0,0),(1,0),(0,1),(1,1), (-1,2), (0,2),(-3,3), (-2,3), (-1,3),(-2,4),(-1,4)\},\\
G_{\varphi} &= \{(0,0),(1,0),(0,1),(1,1), (-2,2), (-1,2), (0,2), (-1,3),(0,3)\}.
    \end{align*} 
    We also define the patterns 
    $b^{+}_{\varphi}$ and $b^{-}_{\varphi}$ of support $B_{\varphi}$
    and the patterns
    $g^{+}_{\varphi}$ and $g^{-}_{\varphi}$ of support $G_{\varphi}$, see
    Figure~\ref{fig:blocks-phi}.
    These patterns can be seen in the middle row of Figure~\ref{fig:other-3-Conway-worms-in-JR}. 
    The patterns 
    $b^{+}_{\varphi}$ 
    and
    $g^{+}_{\varphi}$ 
    are encoding the behavior of points approaching the segment $\overline{PQ}$ ``from the right"
    and
    the patterns 
    $b^{-}_{\varphi}$ 
    and
    $g^{-}_{\varphi}$ 
    are encoding the behavior of points approaching the segment $\overline{PQ}$ ``from the left".

The placement vectors 
giving the relative positions of the translated support $B_\varphi$ and $G_\varphi$
    are $\boldsymbol{b} = (-2,5)$ and $\boldsymbol{g} = (-1,4)$ 
from which we deduce the matrix 
$\left(\begin{smallmatrix} -2 & 1 \\ 5 & -1 \end{smallmatrix}\right)$
    using Equations~\eqref{eqn:B_places} and~\eqref{eqn:G_places}.
\end{proof}

As an example, consider the point $\boldsymbol{p} = (1/4,\varphi/4)$, which lies 1/4 of the way along $\overline{PQ}$ (in the $B$ subinterval), so $\rho = 1/4$. From there we get $s_{\alpha,\rho} = \ldots 01001.01001\ldots$. Next we find \[\left(\begin{smallmatrix} -2 & 1 \\ 5 & -1  \end{smallmatrix}\right) H_{\alpha,\rho} = \{\ldots,(8, -23), (5, -14), (3, -9), \text{\hspace{.2in}} (0, 0), (-3, 9), (-5, 14), (-8, 23),\ldots\}\] and \[\left(\begin{smallmatrix} -2 & 1 \\ 5 & -1  \end{smallmatrix}\right) V_{\alpha,\rho} = \{\ldots,(9, -27), (6, -18), (1, -4), \text{\hspace{.2in}}(-2, 5), (-7, 19),\ldots\}.\]  In Figure~\ref{fig:phi_orbit} we see how these points serve to move the supports of $b^{+}_{\varphi}$, $b^{-}_{\varphi}$, $g^{+}_{\varphi}$, and $g^{-}_{\varphi}$ into place along the nonexpansive strip.

\subsection*{Case: $\boldsymbol{p} \in \Delta_{\infty}$.}\mbox{}\\

    \begin{lemma} 
    Let $x,y\in\Xcal_{\Pcal_0,R_0}$ be two configurations in the Jeandel-Rao Wang
    shift such that $x\neq y$
    and $f(x)=f(y)=\bp\in\Delta_{\infty}\setminus\mathcal{O}_{R_0}(\boldsymbol{0})$.
    Suppose that $\bp$ lies on the vertical segment $Z_\infty$ with endpoints $(\varphi -
        1,0)$ and $(\varphi-1,1)$, that is, let 
        $\bp=(1-\rho)(\varphi^{-1},0)+\rho(\varphi^{-1},1)$ for some $\rho\in\RR$ such that 
    $0<\rho<1$.
    Then there exist 
    finite subsets $B_{\infty},G_{\infty}\subset\ZZ^2$ such that the Conway worm is
    \begin{equation}
        W_{\infty}(\bp) = \left[\left(\begin{smallmatrix} 1 & 0 \\ 5 & -1 \end{smallmatrix}\right) H_{\alpha,\rho} + B_{\infty}\right] 
        \cup
        \left[\left(\begin{smallmatrix} 1 & 0 \\ 5 & -1 \end{smallmatrix}\right) V_{\alpha,\rho} + G_{\infty}\right] \label{eqn:orbit_formula} \end{equation} 
            where $\alpha=2-\varphi$.
    Moreover, 
    there exist 
        patterns $b^{-}_{\infty}$, $b^{+}_{\infty}$ of support $B_{\infty}$
        and patterns $g^{-}_{\infty}$, $g^{+}_{\infty}$ of support $G_{\infty}$
    such that the two resolutions $x|_{D(x,y)}$ and $y|_{D(x,y)}$ of the Conway
    worm are constructed using these, that is,
    \begin{align*}
      x|_{\left(\begin{smallmatrix} 1 & 0 \\ 5 & -1 \end{smallmatrix}\right) 
              \boldsymbol{h}+B_{\infty}} = b^{+}_{\infty}
              \quad
              \text{ and }
              \quad
      x|_{\left(\begin{smallmatrix} 1 & 0 \\ 5 & -1 \end{smallmatrix}\right) 
              \boldsymbol{v}+G_{\infty}} = g^{+}_{\infty}\\
      y|_{\left(\begin{smallmatrix} 1 & 0 \\ 5 & -1\end{smallmatrix}\right) 
              \boldsymbol{h}+B_{\infty}} = b^{-}_{\infty}
              \quad
              \text{ and }
              \quad
      y|_{\left(\begin{smallmatrix} 1 & 0 \\ 5 & -1 \end{smallmatrix}\right) 
              \boldsymbol{v}+G_{\infty}} = g^{-}_{\infty}
    \end{align*}
    for every $\boldsymbol{h} \in H_{\alpha,\rho}$ 
          and $\boldsymbol{v} \in V_{\alpha,\rho}$.
\label{lem:phi-infty-new}\end{lemma}

\begin{proof}
The proof is similar to proof of Lemma~\ref{lem:phi-sq-rot-new} and Lemma~\ref{lem:phi-rot-new}.
    Assume $\boldsymbol{p}$ lies on the $\Delta_{\infty}$-line from 
    $P=(\varphi - 1,0)$ and $Q=(\varphi-1,1)$.
We illustrate in Figure~\ref{fig:IET_vert} the orbit of the point $\bp$ under the
    $\ZZ^2$-action $R_0$ which stays inside $\Delta_{\infty}+\Gamma_0$ until it returns
    to the original segment $\overline{PQ}$.

\begin{figure}[h]
\centering
\begin{subfigure}[b]{.65\textwidth} 
\centering
\includegraphics[scale = .75 ]{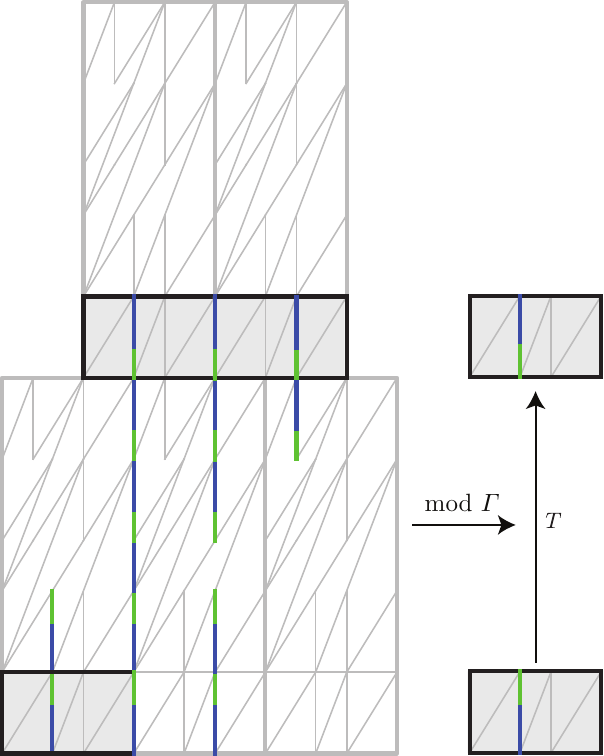}  
\caption{The rotation $T_{\infty}$}\label{fig:IET_vert}
\end{subfigure}
\begin{subfigure}[b]{.3\textwidth} 
\centering
\includegraphics[scale = .4]{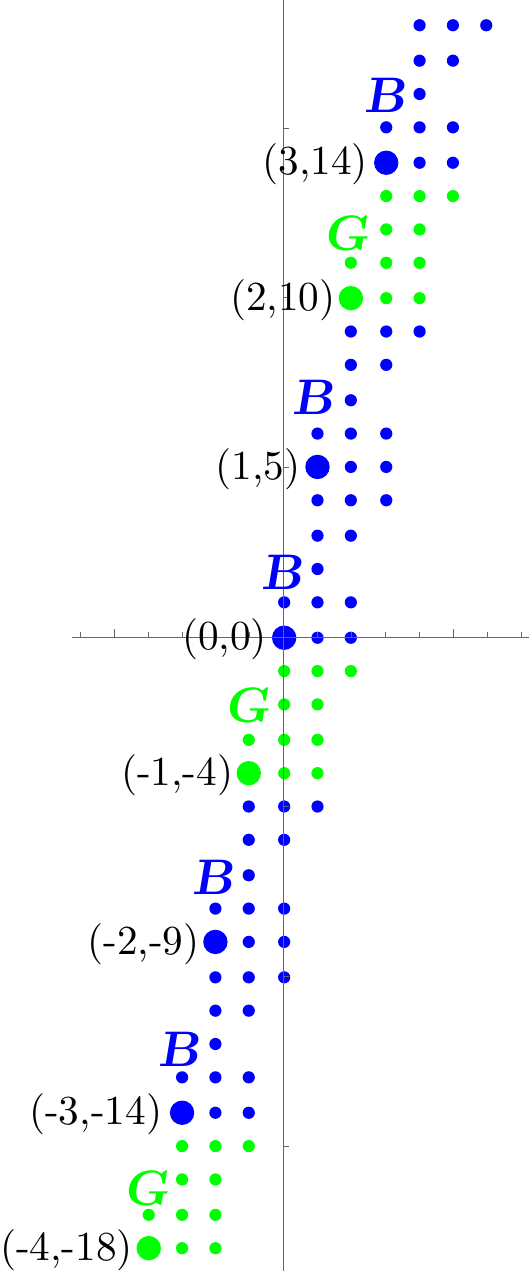}  
    \caption{The elements of $\ZZ^2$ taking the point $\boldsymbol{p} = (\varphi - 1, 1/5)$ on a vertical segment to $\Delta_{\Pcal_0}$.}\label{fig:vert_orbit}
\end{subfigure}\caption{The dynamics of vertical segments in $\Delta_{\Pcal_0}$ under $\ZZ^2$ translation.}\label{fig:vert_dynamics} 
\end{figure}

\begin{figure}[h]
\centering
\begin{subfigure}[b]{.2\textwidth} 
\centering
\includegraphics[scale = .5]{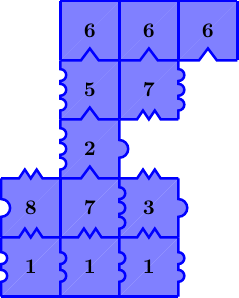}  
\caption{$b^{+}_{\infty}$}
\label{fig:b-plus-infty}
\end{subfigure}
\begin{subfigure}[b]{.2\textwidth} 
\centering
\includegraphics[scale = .5]{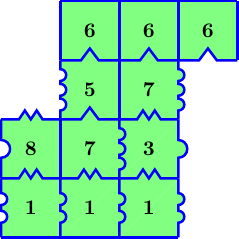}  
\caption{$g^{+}_{\infty}$}
\label{fig:g-plus-infty}
\end{subfigure}
\begin{subfigure}[b]{.2\textwidth} 
\centering
\includegraphics[scale = .5]{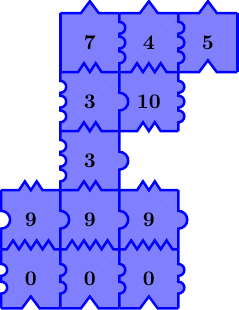}  
\caption{$b^{-}_{\infty}$}
\label{fig:b-minus-infty}
\end{subfigure}
\begin{subfigure}[b]{.2\textwidth} 
\centering
\includegraphics[scale = .5]{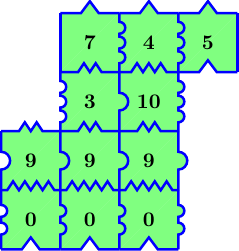}  
\caption{$g^{-}_{\infty}$}
\label{fig:g-minus-infty}
\end{subfigure}
\caption{The patterns $b^+_{\infty}$, $g^+_{\infty}$, $b^-_{\infty}$, and $g^-_{\infty}$.}
\label{fig:blocks-infty} 
\end{figure}

    We divide the segment $\overline{PQ}$ into two subintervals: the blue subinterval $B$ has length $b = \varphi - 1$ and the green subinterval $G$ has length $g = 2 - \varphi$. By inspection of this figure we see that once again the rotation $T_\alpha$ of Equation~\eqref{eqn:IET-T} with $\alpha = 2 - \varphi$ captures the exchange of intervals $B$ and $G$. This time we have the supports
    \begin{align*}
B_{\infty} &= \{(0,0),(1,0),(2,0),(0,1),(1,1),(2,1),(1,2),(1,3),(2,3),(1,4),(2,4),(3,4)\},\\
G_{\infty} &= \{(0,0),(1,0),(2,0),(0,1),(1,1),(2,1),(1,3),(2,3),(1,4),(2,4),(3,4)\}.
    \end{align*}
    We define the patterns 
    $b^{+}_{\infty}$ and $b^{-}_{\infty}$ of support $B_{\infty}$
    and the patterns
    $g^{+}_{\infty}$ and $g^{-}_{\infty}$ of support $G_{\infty}$,
    see Figure~\ref{fig:blocks-infty}.
    These patterns can be seen in the top row of Figure~\ref{fig:other-3-Conway-worms-in-JR}. 
    The patterns 
    $b^{+}_{\infty}$ 
    and
    $g^{+}_{\infty}$ 
    are encoding the behavior of points approaching the segment $\overline{PQ}$ ``from the right"
    and
    the patterns 
    $b^{-}_{\infty}$ 
    and
    $g^{-}_{\infty}$ 
    are encoding the behavior of points approaching the segment $\overline{PQ}$ ``from the left".
    The placement vectors for these patterns are $\boldsymbol{b} = (1,5)$ and
    $\boldsymbol{g} = (1,4)$.
\end{proof}

For example, we consider the orbit of the point $\boldsymbol{p} = (\varphi - 1,1/5)$, which corresponds to $\rho = 1/5$, yielding $s_{\rho,\alpha} = \ldots1001.1101\ldots$ for the pattern of $B$ and $G$ patterns along the Conway worm in the corresponding nonexpansive direction. From there we get the pattern starting points \begin{align*}\left(\begin{smallmatrix} 1 & 0 \\ 5 & -1  \end{smallmatrix}\right) H_{\alpha,\rho} &= \{\ldots,(-3, -14), (-2, -9), \text{\hspace{.2in}} (0, 0), (1, 5), (3, 14),\ldots\} \text{, and}\\ 
    \left(\begin{smallmatrix} 1 & 0 \\ 5 & -1  \end{smallmatrix}\right) V_{\alpha,\rho} &= \{\ldots,(-4, -18), (-1, -4), \text{\hspace{.2in}}(2, 10),\ldots\}.\end{align*} In Figure~\ref{fig:vert_orbit} we give a visualization of  $\mathcal{O}_{R_0}(\boldsymbol{p}) \cap \Delta_{\Pcal_0}$, shifted so that $\boldsymbol{p}$ corresponds to the origin so that we are seeing the elements of $\ZZ^2$ taking $\boldsymbol{p} = (\varphi - 1, 1/5)$ to $\Delta_{\Pcal_0}$. 

\subsection*{Case: $\boldsymbol{p} \in \Delta_0$}\mbox{}\\

\begin{lemma} 
    Let $x,y\in\Xcal_{\Pcal_0,R_0}$ be two configurations in the Jeandel-Rao Wang
    shift such that $x\neq y$
    and $f(x)=f(y)=\bp\in\Delta_{0}\setminus\mathcal{O}_{R_0}(\boldsymbol{0})$.
    Suppose that $\bp$ lies on the horizontal segment $Z_0$ with endpoints $(0,0)$ 
            and $(\varphi,0)$, that is, let 
        $\bp=(1-\rho)(0,0)+\rho(\varphi,0)$ for some $\rho\in\RR$ such that 
    $0<\rho<1$.
    Then there exist 
    finite subsets $B_{0},G_{0}\subset\ZZ^2$ such that the Conway worm is
    \begin{equation}
        W_{0}(\bp) = \left[\left(\begin{smallmatrix} 1 & 0 \\ 0 & 0 \end{smallmatrix}\right) H_{\alpha,\rho} + B_{0}\right] 
        \cup
        \left[\left(\begin{smallmatrix} 1 & 0 \\ 0 & 0 \end{smallmatrix}\right) V_{\alpha,\rho} + G_{0}\right] \label{eqn:orbit_formula} \end{equation} 
            where $\alpha=2-\varphi$.
    Moreover, 
    there exist 
        patterns $b^{-}_{0}$, $b^{+}_{0}$ of support $B_{0}$
        and patterns $g^{-}_{0}$, $g^{+}_{0}$ of support $G_{0}$
    such that the two resolutions $x|_{D(x,y)}$ and $y|_{D(x,y)}$ of the Conway
    worm are constructed using these, that is,
    \begin{align*}
      x|_{\left(\begin{smallmatrix} 1 & 0 \\ 0 & 0 \end{smallmatrix}\right) 
              \boldsymbol{h}+B_{0}} = b^{+}_{0}
              \quad
              \text{ and }
              \quad
      x|_{\left(\begin{smallmatrix} 1 & 0 \\ 0 & 0 \end{smallmatrix}\right) 
              \boldsymbol{v}+G_{0}} = g^{+}_{0}\\
      y|_{\left(\begin{smallmatrix} 1 & 0 \\ 0 & 0\end{smallmatrix}\right) 
              \boldsymbol{h}+B_{0}} = b^{-}_{0}
              \quad
              \text{ and }
              \quad
      y|_{\left(\begin{smallmatrix} 1 & 0 \\ 0 & 0 \end{smallmatrix}\right) 
              \boldsymbol{v}+G_{0}} = g^{-}_{0}
    \end{align*}
    for every $\boldsymbol{h} \in H_{\alpha,\rho}$ 
          and $\boldsymbol{v} \in V_{\alpha,\rho}$.
\label{lem:phi-hor-new}\end{lemma}

            \begin{proof}
    Assume $\boldsymbol{p}$ lies on the $\Delta_{0}$-line from 
    $P=(0,0)$ to $Q=(\varphi,0)$.
We illustrate in Figure~\ref{fig:IET_hor} the orbit of the point $\bp$ under the
    $\ZZ^2$-action $R_0$ which stays inside $\Delta_{0}+\Gamma_0$ until it returns
    to the original segment $\overline{PQ}$.

\begin{figure}[h]
\centering
\begin{subfigure}[b]{.6\textwidth} 
\centering
\includegraphics[scale = 1 ]{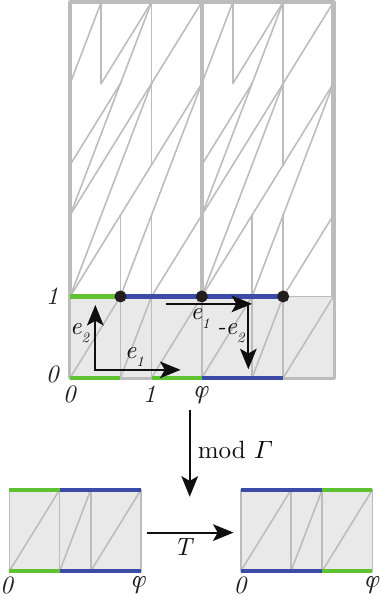}  
\caption{The rotation $T_{\text{hor}}$}\label{fig:IET_hor}
\end{subfigure}
\begin{subfigure}[b]{.38\textwidth} 
\centering
\includegraphics[scale = .55]{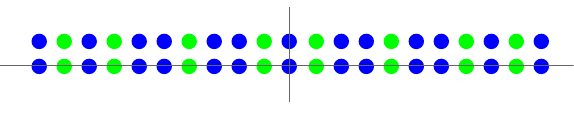}  
    \caption{The set $W_0(\bp)$.}\label{fig:hor_orbit}
\end{subfigure}
    \caption{The dynamics of horizontal segments in $\Delta_{\Pcal_0}$ under $\ZZ^2$ translation.}\label{fig:hor_dynamics} 
\end{figure}

We appeal to Figure~\ref{fig:all-Conway-worms-in-JR-in-one-image} and Figure~\ref{fig:IET_hor} to obtain the required support sets, which are $B_0 = G_0 = \{(0,0), (0,1)\}$ and corresponding patterns $b^{+}_{0}$, $b^{-}_{0}$, $g^{+}_{0}$, and $g^{-}_{0}$, see Figure~\ref{fig:blocks-0}. Also, from Figure~\ref{fig:IET_hor} we deduce that once again the rotation $T_\alpha$ of Equation~\eqref{eqn:IET-T} with $\alpha = \varphi - 2$ captures the rotation happening there.
            \end{proof}

\begin{figure}[h]
\centering
\begin{subfigure}[b]{.2\textwidth} 
\centering
\includegraphics[scale = .5]{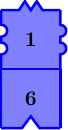}  
\caption{$b^{+}_{0}$}
\label{fig:b-plus-0}
\end{subfigure}
\begin{subfigure}[b]{.2\textwidth} 
\centering
\includegraphics[scale = .5]{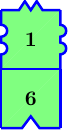}  
\caption{$g^{+}_{0}$}
\label{fig:g-plus-0}
\end{subfigure}
\begin{subfigure}[b]{.2\textwidth} 
\centering
\includegraphics[scale = .5]{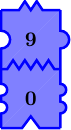}  
\caption{$b^{-}_{0}$}
\label{fig:b-minus-0}
\end{subfigure}
\begin{subfigure}[b]{.2\textwidth} 
\centering
\includegraphics[scale = .5]{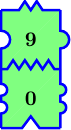}  
\caption{$g^{-}_{0}$}
\label{fig:g-minus-0}
\end{subfigure}
\caption{The patterns $b^+_{0}$, $g^+_{0}$, $b^-_{0}$, and $g^-_{0}$.}
\label{fig:blocks-0} 
\end{figure}

As an example, we consider the orbit of the point $\boldsymbol{p} = 3/10(\varphi,0)$, which corresponds to $\rho= 3/10$. We obtain 
$s_{\alpha,\rho} = \ldots 0101001001.01001001010\ldots$ with 
{\small 
\begin{align*}
\left(\begin{smallmatrix} 1 & 0 \\ 0 & 0  \end{smallmatrix}\right) H_{\alpha,\rho} 
     &= \{\ldots,(-10, 0), (-8, 0), (-6, 0), (-5, 0), (-3, 0), (-2, 0), \text{\hspace{.2in}} (0, 0), (2, 0), (3, 0), (5, 0), (6, 0), (8, 0),\ldots\},\\
\left(\begin{smallmatrix} 1 & 0 \\ 0 & 0  \end{smallmatrix}\right) V_{\alpha,\rho} 
&= \{\ldots,(-9, 0), (-7, 0), (-4, 0), (-1, 0), \text{\hspace{.2in}}(1, 0), (4, 0), (7, 0), (9, 0),\ldots\}.
\end{align*}}
In Figure~\ref{fig:hor_orbit} we give a visualization of $W_0(\bp)$.

We may now prove the second main result. 

\begin{proof}[Proof of Theorem~\ref{thm:resolution}]
    Let $x,y\in\Xcal_{\Pcal_0,R_0}$ be two configurations in the Jeandel-Rao Wang
    shift such that $x\neq y$
    and $f(x)=f(y)=\bp\in\Delta_{\Pcal_0,R_0}\setminus\mathcal{O}_{R_0}(\boldsymbol{0})$.
    From Lemma~\ref{lem:Dxy=Wip},
    there exists $i \in \{0,\infty,\varphi,\varphi^2\}$ such that
    \[
        D(x,y)=\{\bn\in \ZZ^2\mid x_\bn \neq y_\bn\}
        =W_i(\bp).
    \]
    From Proposition~\ref{prop:orbit-remaining-in-partition-boundary},
    the set $D(x,y)$ is a Conway worm associated to some nonexpansive subspace $F$.

Since $\bp\in\Delta_{\Pcal_0,R_0}$, there exists $\bk_1\in\Z^2$ such
    that $R_0^{\bk_1}(\bp)\in\Delta_{\Pcal_0}$
    which implies $R_0^{\bk_1}(\bp)\in\Delta_i$.
From Lemma~\ref{lem:move_to_origin}, there exists $\bk_2\in\Z^2$ 
    that $R_0^{\bk_2}(R_0^{\bk_1}(\bp))\in Z_i$
    where $Z_i$ is the slope-$i$ base segment in the partition $\Pcal_0$.
    Let $\bk=\bk_1+\bk_2$.
    We consider the shifted configurations
    $\sigma^{\bk}(x)$
    and
    $\sigma^{\bk}(y)$.
    We have that
    $f(\sigma^{\bk}(x))
         =R_0^{\bk}(f(x))
         =R_0^{\bk}(\bp)
         =R_0^{\bk}(f(y))
         =f(\sigma^{\bk}(y))\in Z_i$
    and
    \[
        D(x,y) 
        = \bk+ D(\sigma^{\bk}(x), \sigma^{\bk}(y))
        = \bk+ W_i(R_0^\bk(\bp)).
    \]
    The result follows from
Lemma~\ref{lem:phi-sq-rot-new},
Lemma~\ref{lem:phi-rot-new},
Lemma~\ref{lem:phi-infty-new} and
Lemma~\ref{lem:phi-hor-new}
applied to the point $R_0^\bk(\bp)\in Z_i$.
\end{proof}

\section{Octopods and related questions}\label{sec:octopods}

Penrose's low-order aperiodic protoset and its connection to quasicrystalline
structures prompted research into the idea of ``seeds" around which a perfectly
formed quasicrystal can form, yet the seed itself is not part of the perfect
quasicrystalline structure, such as when an ice crystal forms around a particle
of some non-ice material. In terms of 2-D tilings, an \emph{essential hole} is
an empty finite region of the plane around which a unique valid tiling can be
formed, yet the hole itself cannot be filled in by copies of tiles from the
protoset as part of a valid tiling \cite[p.566-567]{GS1}. 
Essential holes in the context of Penrose tilings can be obtained from the
tiling where 10 half-infinite Conway worms emerge out of a central regular
dodecagon cartwheel.
Flipping the 10 half-infinite Conway worms from one resolution
to the other around the cartwheel leads to $2^{10}=1024$ combinations. 
Up to isometry, Conway (\cite{MR968892,zbMATH01014067}) 
identified 62 finite regions called \emph{decapods}.
Two decapods called \emph{Batman} and \emph{Asterix} are legal patterns and thus can be completed
to an infinite tiling of the plane in infinitely many ways.
The 60 other decapods are such that a unique infinite Penrose tilings can be
formed around them but such that the hole formed by the decapod cannot be
filled in by the Penrose rhombs as part of a valid tiling. 

Later, in Onodo et.
al \cite{PhysRevLett.60.2653}, it is pointed out that completing the tiling
around a decapod can be carried out using only local interactions, based on
restricting the matching rules of Penrose tiles through vertex matching rules.
This bolstered the idea that quasicrystals can form around seeds in a natural
way, based only on the local attraction of the molecules of the crystal and
without need of any scaffolding to enforce long-range structure. More recently,
Galanov explored this topic in the context of cut-and-project schemes
\cite{galan2019}, coining the term ``bad seed" to describe such holes that can
be extended to unique tilings.

We suggest to say that an \emph{octopod} is a decapod in the context of Jeandel-Rao Wang shift 
due to the 4 nonexpansive directions and to the 8 half-infinite Conway worms
running out of a central region in some specific tiling, see
Figure \ref{fig:all-Conway-worms-in-JR-in-one-image}. An octopod with
Jeandel-Rao Wang tiles is shown in Figure~\ref{fig:octopod}.
We believe it is an essential hole as it can be extended to a tiling of the
plane except the hole, see Figure~\ref{fig:bad_seed_completed}.

\begin{question}
    Prove that the octopod shown in Figure~\ref{fig:octopod} is an essential hole.
\end{question}

An answer to this question might rely on Conjecture \ref{conj:ergodic}. This
conjecture implies that the only configurations in $\Omega_0 \setminus
\mathcal{X}_{\mathcal{P}_0,R_0}$ are configurations in
$\mathcal{X}_{\mathcal{P}_0,R_0}$ that have had their bottom (or top) halves
shifted along an horizontal fault line.

\begin{figure}[h]
    \includegraphics[width=3cm]{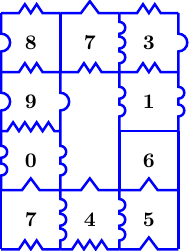}
    \caption{An octopod for the Jeandel-Rao Wang shift.}
    \label{fig:octopod}
\end{figure}

\begin{figure}[h]
\begin{center}
  \includegraphics[width=0.75\textwidth]{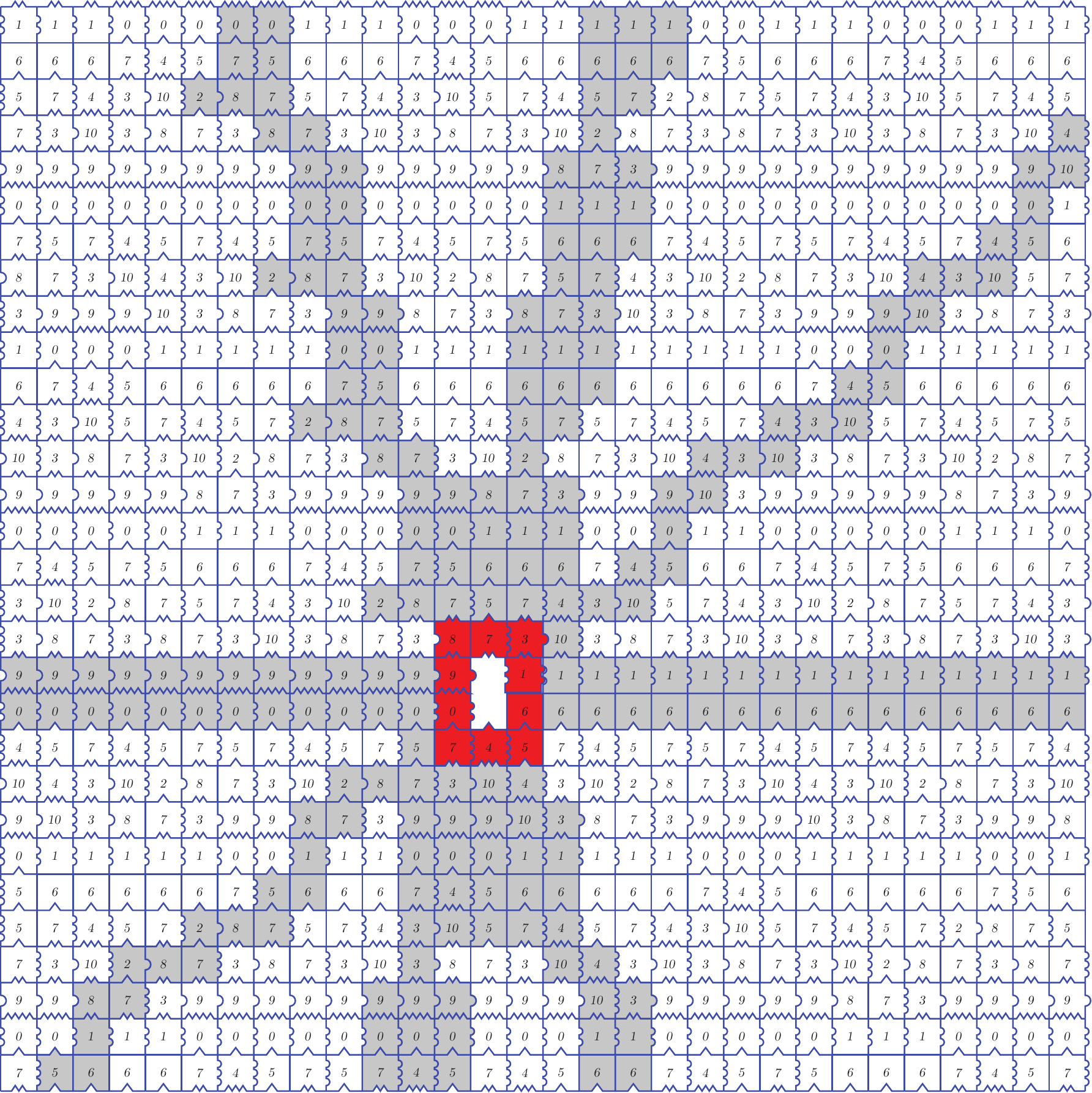}
\end{center}
    \caption{A valid tiling around an octopod in the Jeandel-Rao Wang shift.
    The hole can not be filled by copies of the Jeandel-Rao Wang tiles.}
    \label{fig:bad_seed_completed}
\end{figure}

Figure \ref{fig:all-Conway-worms-in-JR-in-one-image} shows the 8
half-infinite Conway worms running out of a central region. 
Similarly as for the Penrose tilings, we can flip each of the 8 half-infinite
Conway worms to obtain a total of $2^8=256$ possible combinations for a central
octopod.

\begin{question}
    How many different octopods exists for the Jeandel-Rao Wang shift? 
    How many of them are essential holes?
\end{question}

\bibliographystyle{alpha}
\bibliography{biblio}

\end{document}